\documentclass[11pt]{amsart}
\usepackage{amsmath,amssymb,amsthm,mathrsfs,enumerate,bm,xcolor,multirow,pbox}
\usepackage{graphicx,color,framed,tikz,enumitem}
\usepackage{subfigure,float}

\setlist{leftmargin=5mm}
\usepackage[colorlinks,linkcolor=red,citecolor=blue,urlcolor=blue]{hyperref}

\allowdisplaybreaks[4]
\numberwithin{equation}{section}
\newcommand{\N}{\mathbb{N}}
\newcommand{\R}{\mathbb{R}}

\newcommand{\E}{\mathbb{E}}
\newcommand{\Prob}{\mathbb{P}}
\newcommand{\G}{\mathbb{G}}

\newcommand{\pnorm}[2]{\lVert#1\rVert_{#2}}

\newcommand{\abs}[1]{\lvert#1\rvert}
\newcommand{\bigabs}[1]{\big\lvert#1\big\rvert}
\newcommand{\biggabs}[1]{\bigg\lvert#1\bigg\rvert}

\renewcommand{\epsilon}{\varepsilon}

\renewcommand{\d}[1]{\mathrm{d}#1}

\AtBeginDocument{%
	\def\MR#1{}
}

\theoremstyle{definition}\newtheorem{problem}{Problem}[section]
\theoremstyle{definition}
\theoremstyle{remark}\newtheorem{assumption}{Assumption}

\theoremstyle{remark}\newtheorem{remark}[problem]{Remark}
\theoremstyle{definition}
\theoremstyle{plain}\newtheorem{theorem}[problem]{Theorem}
\theoremstyle{plain}
\theoremstyle{plain}\newtheorem{lemma}[problem]{Lemma}
\theoremstyle{plain}\newtheorem{proposition}[problem]{Proposition}
\theoremstyle{plain}
\theoremstyle{plain}\newtheorem{corollary}[problem]{Corollary}
\theoremstyle{plain}
\theoremstyle{plain}
\theoremstyle{plain}

\begin{document}

\title[Inference in convex models]{Inference for local parameters in convexity constrained models}
\thanks{The research of H.~Deng is partially supported by DMS-1451817 and CCF-1934924. The research of Q.~Han is partially supported by DMS-1916221. The research of B.~Sen is partially supported by DMS-1712822. }

\author[H. Deng]{Hang Deng}

\address[H. Deng]{
	Department of Statistics, Rutgers University, Piscataway, NJ 08854, USA.
}
\email{hdeng@stat.rutgers.edu}

\author[Q. Han]{Qiyang Han}

\address[Q. Han]{
	Department of Statistics, Rutgers University, Piscataway, NJ 08854, USA.
}
\email{qh85@stat.rutgers.edu}

\author[B. Sen]{Bodhisattva Sen}

\address[B. Sen]{
	Department of Statistics, Columbia University, New York, NY 10027, USA.
}
\email{bodhi@stat.columbia.edu}

\date{\today}

\keywords{limit distribution theory, confidence interval, convex regression, log-concave density estimation, $s$-concave density estimation, deconvolution, shape constraints}
\subjclass[2000]{60F17, 62E17}
\maketitle

\begin{abstract}
We consider the problem of inference for local parameters of a convex regression function $f_0: [0,1] \to \R$ based on observations from a standard nonparametric regression model, using the convex least squares estimator (LSE) $\widehat{f}_n$. For $x_0 \in (0,1)$, the local parameters include the pointwise function value $f_0(x_0)$, the pointwise derivative $f_0'(x_0)$, and the anti-mode (that is, the smallest minimizer) of $f_0$. It is well-known that the limiting distribution of the estimation error $(\widehat{f}_n(x_0) - f_0(x_0), \widehat{f}_n'(x_0) - f_0'(x_0) )$ depends on the unknown second derivative $f_0''(x_0)$, and is therefore not directly applicable for inference. To circumvent this impasse, we show that the following locally normalized errors (LNEs) enjoy pivotal limiting behavior: Let $[\widehat{u}(x_0), \widehat{v}(x_0)]$ be the maximal interval containing $x_0$ where $\widehat{f}_n$ is linear. Then, under standard conditions, 
\begin{align*}
\begin{pmatrix}
	\sqrt{n(\widehat{v}(x_0)-\widehat{u}(x_0))}(\widehat{f}_n(x_0)-f_0(x_0))  \\
	\sqrt{n(\widehat{v}(x_0)-\widehat{u}(x_0))^3}(\widehat{f}_n'(x_0)-f_0'(x_0))
\end{pmatrix}
\rightsquigarrow \sigma \cdot
\begin{pmatrix}
 	\mathbb{L}^{(0)}_2 \\
 	\mathbb{L}^{(1)}_2
 \end{pmatrix},
\end{align*}
where $n$ is the sample size, $\sigma$ is the standard deviation of the errors, and $\mathbb{L}^{(0)}_2, \mathbb{L}^{(1)}_2$ are universal random variables. This asymptotically pivotal LNE theory instantly yields a simple tuning-free procedure for constructing confidence intervals for $f_0(x_0)$ and $f_0'(x_0)$. We also construct an asymptotically pivotal LNE for the anti-mode of $f_0$, and its limiting distribution does not even depend on $\sigma$. These asymptotically pivotal LNE theories are further extended to other convexity/concavity constrained models for which a limit distribution theory is available for problem-specific estimators. Concrete models include: (i) Log-concave density estimation, (ii) $s$-concave density estimation, (iii) convex nonincreasing density estimation, (iv) concave bathtub-shaped hazard function estimation, and (v) concave distribution function estimation from corrupted data. The proposed confidence intervals for all these models are proved to have asymptotically exact coverage and optimal length, and require no further information than the estimator itself. We provide extensive simulation results that validate our theoretical results. 
\end{abstract}



\section{Introduction}

\subsection{Overview}
Consider the standard nonparametric regression model:
\begin{align}\label{model:regression}
Y_i = f_0(X_i)+\xi_i,\quad 1\leq i\leq n,
\end{align}
where $f_0: [0,1]\to \R$ is an unknown convex function, $X_1,\ldots,X_n$ are fixed or random design points, and $\xi_i$'s are i.i.d.~mean $0$ (unobserved) errors with variance $\sigma^2 >0$. We are interested in inference for local parameters of this model, including the function value $f_0(x_0)$ and its derivative $f_0'(x_0)$ at an interior point $x_0 \in (0,1)$, and the anti-mode of $f_0$, that is, the smallest minimizer of $f_0$.

The convex/concave regression model has been studied for more than 60 years in statistics. It was first proposed by \cite{hildreth1954point} to solve real problems particularly in economics where, for example, demand and supply relationship is often assumed to satisfy the concavity constraint; also see~\cite{varian1984nonparametric, matzkin1991semiparametric, convexexample}. Driven by its broad applications, considerable progress has been made in convex regression in the last few decades. Most of these works are almost exclusively focused on the convex least squares estimator (LSE) $\widehat{f}_n$ which is defined as the convex function that minimizes the mean squared error:
\begin{align*}
\widehat{f}_n\; \in\; \underset{f:\,\mathrm{convex}}{\mathrm{arg\,min}} \,\frac{1}{n}\sum_{i=1}^n \big(Y_i - f(X_i) \big)^2.
\end{align*}
Although not unique, the convex LSE $\widehat{f}_n$ has unique specification at the design points, that is, $( \widehat{f}_n(X_1), \ldots, \widehat{f}_n(X_n))^{\top}$ is unique. If we linearly interpolate this unique specification, the resulting piecewise linear function with kinks at design points is also unique and we treat this $\widehat{f}_n$ as the (unique) convex LSE without loss of generality. Consistency of the convex LSE $\widehat{f}_n$ is proved in \cite{hanson1976consistency}. \cite{mammen1991nonparametric} derives the pointwise convergence rate and \cite{dumbgen2004consistency} gives the uniform convergence rate of $\widehat{f}_n$. In \cite{groeneboom2001canonical,groeneboom2001estimation}, the authors derive the local asymptotic distribution theory for the LSE $\widehat{f}_n$. For global risk and the adaptation behavior of the convex LSE, results can be found in \cite{chatterjee2015risk,guntuboyina2013global,bellec2018sharp}. The most relevant result to our objectives in this paper is the limit distribution theory by \cite{groeneboom2001estimation}, which states that under certain conditions on the noise $\{\xi_i\}$ and design points $\{X_i\}$, when $f_0$ is twice continuously differentiable in a neighborhood of $x_0$ with $f_0''(x_0) > 0$,
\begin{align}\label{intro:limit}
\begin{pmatrix}
	\big(4! \sigma/f_0''(x_0) \big)^{1/5} \cdot n^{2/5} \big( \widehat{f}_n(x_0)-f_0(x_0) \big) \\
	\big(4! \sigma/f_0''(x_0) \big)^{3/5} \cdot n^{1/5} \big( \widehat{f}_n'(x_0)-f_0'(x_0) \big)
\end{pmatrix}
\rightsquigarrow
\sigma \cdot \begin{pmatrix}
	\mathbb{H}_{2}^{(2)}(0) \\
	\mathbb{H}_{2}^{(3)}(0)
\end{pmatrix},
\end{align}
where $\mathbb{H}_{2}^{(2)}(0)$ and $\mathbb{H}_{2}^{(3)}(0)$ are defined by a pivotal process with no dependence on $f_0$, $n$, or $\sigma$ (see Theorem~\ref{thm:limiting_dist_convex} for the details). Here $\rightsquigarrow$ denotes weak convergence. This theory is extended by \cite{chen2014convex,ghosal2017univariate} to include mean functions that are ``flatter'' at $x_0$, that is, $f_0''(x_0) =0$.  

As nice as the pointwise limit distribution theory \eqref{intro:limit} is for the convex LSE, there is, so far, no theoretically valid inference method that exploits its merits. The main difficulty in using (\ref{intro:limit}) for inference rests in its dependence on the unknown parameter $f_0''(x_0)$; even though it is fairly easy to find a consistent estimator for the noise level $\sigma^2$.
It is tempting to look for a sample proxy of $f_0''(x_0)$ by considering, e.g., kernel smoothing methods to estimate $f_0''(x_0)$; or we might consider
bootstrap methods such as the $m$-out-of-$n$ bootstrap and bootstrap with smoothing \cite{sen2010inconsistency,seijo2011change} so that such a sample proxy can be bypassed. However, these inference approaches require careful tuning (bandwidth for smoothing and $m$ for $m$-out-of-$n$ bootstrap) that can be delicate and hard to evaluate, making them not very appealing in shape restricted problems. 

\begin{figure}
\centering
\includegraphics[width=12cm]{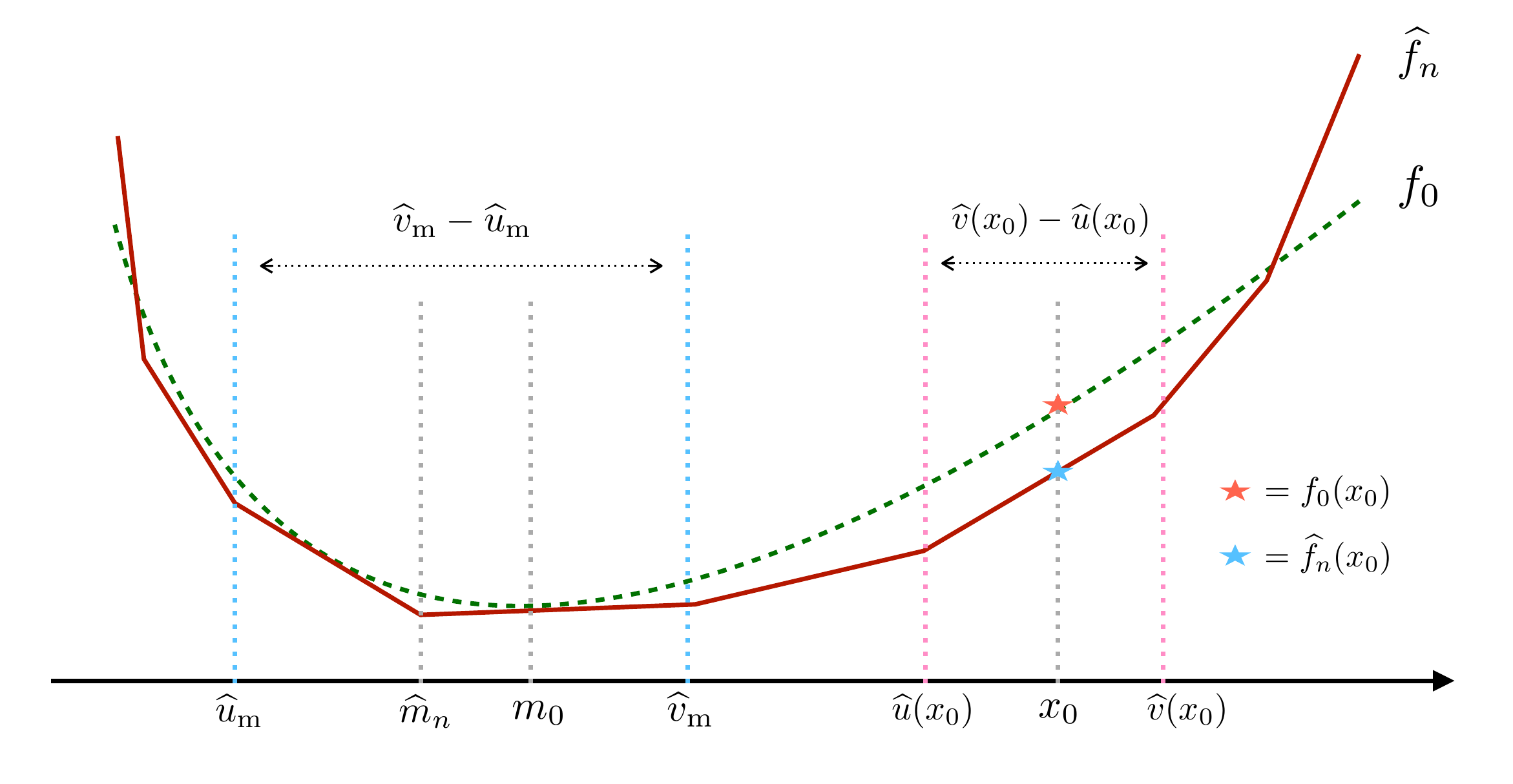}
\caption{Figure illustration of the quantities in (\ref{intro:limit_convex}) and (\ref{intro:limit_mode}).}
\label{fig:inf}
\end{figure}

It turns out that inference can be carried out, in this problem, in a surprisingly straightforward way and the sample proxy of $f_0''(x_0)$ is directly `accessible' from the convex LSE $\widehat{f}_n$,  although the second derivative of $\widehat{f}_n$ is almost everywhere zero (as $\widehat{f}_n$ is piecewise linear). The key observation is that the bias-variance trade-off should happen on each linear piece of the convex LSE $\widehat{f}_n$, since otherwise the linear pieces would adjust their lengths to further reduce the mean squared error of $\widehat{f}_n$. Let $[\widehat{u}(x_0), \widehat{v}(x_0)]$ be the maximal interval containing $x_0$ where $\widehat{f}_n$ is linear (see Figure~\ref{fig:inf}). As 
\begin{align}\label{intro:second_derivative_proxy}
\hbox{(bias)} \quad & f_0''(x_0)(\widehat{v}(x_0) - \widehat{u}(x_0))^2 \asymp \frac{\sigma}{\sqrt{n(\widehat{v}(x_0) - \widehat{u}(x_0))}} \quad \hbox{(s.d.)}
\cr
& \Rightarrow \;\; f_0''(x_0) \asymp \sigma/\sqrt{n(\widehat{v}(x_0) - \widehat{u}(x_0))^5},
\end{align}
it is reasonable to expect that, by plugging \eqref{intro:second_derivative_proxy} into \eqref{intro:limit}, the resulting quantities will be asymptotically pivotal. In fact, following this intuition, we rigorously establish a pivotal limit distribution theory (see Theorem \ref{thm:pivotal_limit_fcn}): Under the same conditions for \eqref{intro:limit},
\begin{align}\label{intro:limit_convex}
\begin{pmatrix}
	\sqrt{n(\widehat{v}(x_0)-\widehat{u}(x_0))}(\widehat{f}_n(x_0)-f_0(x_0))  \\
	\sqrt{n(\widehat{v}(x_0)-\widehat{u}(x_0))^3}(\widehat{f}_n'(x_0)-f_0'(x_0))
\end{pmatrix}
\rightsquigarrow \sigma \cdot
\begin{pmatrix}
 	\mathbb{L}^{(0)}_2 \\
 	\mathbb{L}^{(1)}_2
 \end{pmatrix},
\end{align}
where $\mathbb{L}^{(0)}_2, \mathbb{L}^{(1)}_2$ are universal random variables, whose distributions do not depend on $f_0$, $n$, or $\sigma$. We also show that $\mathbb{L}^{(0)}_2, \mathbb{L}^{(1)}_2$ have exponentially decaying tails (see Corollary \ref{cor:unif_tail_pivot}), a result that we obtain from new exponential tail estimates for the random variables $\mathbb{H}_2^{(2)}(0)$ and $\mathbb{H}_2^{(3)}(0)$ appearing in the limit theory (\ref{intro:limit}) (see Theorem \ref{thm:tail_uniform}). The latter result answers affirmatively a question concerning the existence of moments of  $\mathbb{H}_2^{(2)}(0)$ posed in \cite{groeneboom2001canonical}. Furthermore, the above pivotal limit distribution theory (\ref{intro:limit_convex}) can be generalized to the scenario when $f_0''(x_0)=0$ in similar spirit to \cite{chen2014convex,ghosal2017univariate}; see Theorem \ref{thm:pivotal_limit_fcn} for more details.

It is important to note that the distribution of $(\mathbb{L}_2^{(0)}, \mathbb{L}_2^{(1)})$ in (\ref{intro:limit_convex}) is different from $(\mathbb{H}_2^{(2)}(0), \mathbb{H}_2^{(3)}(0))$ in (\ref{intro:limit}), as the sample proxy $\sigma/\sqrt{n(\widehat{v}(x_0) - \widehat{u}(x_0))^5}$ in (\ref{intro:second_derivative_proxy}) is actually \emph{not} a consistent estimator of $f_0''(x_0)$; but rather it has the same order of magnitude as $f_0''(x_0)$. As we may treat $\sqrt{n (\widehat{v}(x_0) - \widehat{u}(x_0))}$ and $\sqrt{n (\widehat{v}(x_0) - \widehat{u}(x_0))^3}$ in (\ref{intro:limit_convex}) as local normalizing factors for the magnitude of the standard deviation of $\widehat{f}_n(x_0) - f_0(x_0)$ and $\widehat{f}_n'(x_0) - f_0'(x_0)$ respectively, we call the normalized errors in \eqref{intro:limit_convex} and other errors of this type the \emph{locally normalized errors} (LNEs).

The asymptotically pivotal LNE theory in (\ref{intro:limit_convex}) can be used for inference immediately. In testing the hypothesis $H_0: f_0(x_0) = \mu_0$ versus $H_1: f_0(x_0) \neq \mu_0$ for a fixed $\mu_0$, the rejection region at significance level $1- \delta$ is 
\begin{align*}
\Big\{ \widehat{f}_n(x_0): \big| \sqrt{n(\widehat{v}(x_0)-\widehat{u}(x_0))}(\widehat{f}_n(x_0)- \mu_0) \big| \le \widehat{\sigma}\cdot  c_{\delta}^{(0)} \Big\},
\end{align*}
and the $1- \delta$ confidence interval (CI) for $f_0(x_0)$ is
\begin{align*}
\Big[ \widehat{f}_n(x_0) - \widehat{\sigma}\cdot c_{\delta}^{(0)}\big/ \sqrt{n(\widehat{v}(x_0)-\widehat{u}(x_0))}, \widehat{f}_n(x_0) +\widehat{\sigma}\cdot  c_{\delta}^{(0)}\big/ \sqrt{n(\widehat{v}(x_0)-\widehat{u}(x_0))} \Big],
\end{align*}
where $c_{\delta}^{(0)}$ is the $(1- \delta)$-quantile of $|\mathbb{L}_2^{(0)}|$, and $\widehat{\sigma}$ is a consistent estimator of $\sigma$.

Another important problem in convex regression is the inference for the anti-mode, defined as the smallest minimizer of $f_0$. It turns out that the above approach of constructing an asymptotically pivotal LNE is still applicable for this location parameter. We establish a pivotal limit distribution theory for the anti-mode as follows: Let $m_0$ and $\widehat{m}_n$ be the anti-mode of $f_0$ and $\widehat{f}_n$ respectively. Under regularity conditions on the noise variables and design points, it holds, when $f_0$ is twice continuously differentiable in a neighborhood of $m_0$ with $f_0''(m_0) >0$, that (see Theorem \ref{thm:pivotal_limit_mode})
\begin{align}\label{intro:limit_mode}
\frac{1}{\widehat{v}_{\mathrm{m}}- \widehat{u}_{\mathrm{m}}}\big(\widehat{m}_n- m_0\big)\rightsquigarrow \mathbb{M}_2,
\end{align}
where $\widehat{u}_{\mathrm{m}}$ and $\widehat{v}_{\mathrm{m}}$ are the nearest kink points of $\widehat{f}_n$ to the left and right of $\widehat{m}_n$ (see Figure \ref{fig:inf}), and $\mathbb{M}_2$ has a pivotal distribution. What is even more striking in (\ref{intro:limit_mode}) than the pivotal limit distribution theory \eqref{intro:limit_convex} is that the LNE for the anti-mode is scale-free and therefore it is not necessary to estimate $\sigma$. 

The approach of the asymptotically pivotal LNE theory in (\ref{intro:limit_convex})-(\ref{intro:limit_mode}) has much broader applications beyond the regression setting in (\ref{model:regression}). In Section \ref{section:other_convex}, we extend this approach to many other nonparametric models under convexity/concavity constraints where a limit distribution theory similar to (\ref{intro:limit}) is available. These models include:
\begin{enumerate}[label = (\roman*),leftmargin = 0.4in]
	\item log-concave density estimation \cite{balabdaoui2009limit,saumard2014logconcavity,samworth2018recent},
	\item $s$-concave density estimation \cite{dharmadhikari1988unimodality,koenker2010quasi,han2015approximation},
		\item convex nonincreasing density estimation \cite{groeneboom2001estimation},
	\item convex bathtub-shaped hazard function estimation \cite{jankowski2009nonparametric}, and
	\item concave distribution function estimation from corrupted data~\cite{jongbloed2009estimating}. 
\end{enumerate} 

In the popular log-concave density estimation model, we construct asymptotically pivotal LNEs for the value  and the derivative at a point and the mode of the underlying log-concave density using the standard log-concave maximum likelihood estimator (MLE) that has been studied intensively in the literature, see e.g.,~\cite{walther2002detecting,cule2010maximum,cule2010theoretical,dumbgen2009maximum,dumbgen2011approximation,pal2007estimating,seregin2010nonparametric,kim2016global,kim2016adaptation,feng2018adaptation,doss2013global,barber2020local,han2019}. In other models, asymptotically pivotal LNE theories analogous to (\ref{intro:limit_convex}) and (\ref{intro:limit_mode}) (whenever available) are also established for natural tuning-free estimators with a limit distribution theory of the type (\ref{intro:limit}).

To the best of our knowledge, inference procedures with theoretical guarantees in the above models are limited to the problem of inference for the mode of log-concave densities, for which \cite{doss2016inference} developed the likelihood ratio test (LRT). We discuss this LRT based method in detail in Section~\ref{subsection:log_concave} and provide a numerical performance comparison with the proposed CIs in Section~\ref{subsection:DW_comparison}.
	
To put our results in a broader context, the idea of constructing an asymptotically pivotal LNE for inference was first employed in isotonic regression where  $f_0: [0,1]\to \R$, in model \eqref{model:regression},  is assumed to be a nondecreasing function. \cite{deng2020confidence} establishes the following local limit theory for an asymptotically pivotal LNE based on the isotonic LSE $\widehat{f}_n^{(\mathrm{iso})}$: 
\begin{align}\label{intro:limit_monotone}
\sqrt{n\big(\widehat{v}^{(\mathrm{iso})} (x_0)-\widehat{u}^{(\mathrm{iso})} (x_0) \big)} \big(\widehat{f}_n^{(\mathrm{iso})}(x_0)-f_0(x_0)\big)\rightsquigarrow \sigma \cdot  \mathbb{L}^{(\mathrm{iso})}_1,
\end{align}
where $[\widehat{u}^{(\mathrm{iso})}(x_0), \widehat{v}^{(\mathrm{iso})}(x_0)]$  is the maximal interval containing $x_0$ where $\widehat{f}_n^{(\mathrm{iso})}$ remains constant, and $\mathbb{L}^{(\mathrm{iso})}_1$ has a pivotal distribution. Compared to (\ref{intro:limit_monotone}), the asymptotically pivotal LNE theory (\ref{intro:limit_convex})-(\ref{intro:limit_mode}) demonstrates the additional advantage of convexity/concavity constraints in providing simultaneous inference for all local parameters $f_0(x_0)$, $f_0'(x_0)$, $m_0$. This is possible as the convexity/concavity constraints induce a natural second-order curvature condition under which sufficient information is available for all these local parameters, whereas it is not possible to infer more than $f_0(x_0)$ from the first-order monotonicity constraint as in (\ref{intro:limit_monotone}).  
	
Technically, the pivotal LNE theory (\ref{intro:limit_convex})-(\ref{intro:limit_mode}) for models under convexity/concavity constraints is more challenging to establish than (\ref{intro:limit_monotone}) for at least two different reasons. Firstly, unlike the block estimators with max-min and min-max formulas in isotonic regression, the convex LSE has no explicit formula. Technical complications due to the lack of such explicit formulas are well documented in convexity constrained problems \cite{groeneboom2001canonical,groeneboom2001estimation,doss2016inference}. In our problem, the implicit functionals that represent $\widehat{u}(x_0),\widehat{v}(x_0)$ in terms of the underlying process (with piecewise linear convex realizations) are in general not continuous with respect to the topology induced by the mode of convergence of the underlying process to its limit. The essential difficulty then is to argue that the underlying process must converge to the limit in the `continuity set' of this implicit functional in the prescribed topology. Secondly, (\ref{intro:limit_mode}) is different from (\ref{intro:limit_convex}) in that the location of the anti-mode of $\widehat{f}_n$ is random in (\ref{intro:limit_mode}), while $\widehat{f}_n(x_0)$ and $\widehat{f}_n'(x_0)$ have fixed location $x_0$ in (\ref{intro:limit_convex}).  This means that the localization arguments used for proving (\ref{intro:limit_mode}) must be carried out at a random center, and therefore must be performed in a nonstandard `uniform' fashion. These difficulties lead us to adopt a technical approach entirely different from \cite{deng2020confidence} to prove (\ref{intro:limit_convex})-(\ref{intro:limit_mode}).

The rest of the paper is organized as follows. We study the local inference mainly through \eqref{intro:limit_convex} and \eqref{intro:limit_mode} for convex regression in Section \ref{section:pivot_limit}. In Section \ref{section:other_convex}, we build a framework for constructing the LNEs for general models under convexity/concavity constraints and apply it to the models mentioned above. In Section \ref{section:tail_estimate}, we present a uniform tail estimate for the related limit processes that is both useful, for the results in Section \ref{section:pivot_limit}, and of independent interest. We carry out extensive simulations in Section \ref{section:simulation} to support our theoretical results in Sections~\ref{section:pivot_limit} and~\ref{section:other_convex}. All technical proofs are deferred to the Appendix.

\subsection{Notation}
For simplicity of presentation, we write the CI $[\widehat{\theta} - c_0, \widehat{\theta} + c_0]$ which is symmetric around $\widehat{\theta}$ as $\mathcal{I} = [ \widehat{\theta} \pm c_0]$. The anti-mode, or the smallest minimizer, of a convex function $f$ is denoted by $[f]_{\mathrm{m}} = [f]_{\mathrm{m}^+}$, and the mode, or the smallest maximizer of a concave function $g$ is denoted by $[g]_{\mathrm{m}^-}$ which equals $[-g]_{\mathrm{m}}$; see (\ref{def:mode}) for a formal definition.  
Let $f_0^{(k)}(\cdot)$ with $k=1,2,\ldots$, denote the $k$-th derivative of $f_0(\cdot)$. We may also use $f_0^{(0)}(x_0) \equiv f_0(x_0)$ and $f_0^{(1)}(x_0) \equiv f_0'(x_0)$ interchangeably. For two real numbers $a,b$, $a\vee b\equiv \max\{a,b\}$, $a\wedge b\equiv\min\{a,b\}$, and $a_+ \equiv a\vee 0$, $a_- \equiv (-a)\vee 0$. The indicator function $\bm{1}_{ A}(x) =  \bm{1}_{ \{x \in A\} }$ outputs $1$ if $x \in A$ and $0$ otherwise. We use $C_{x}$ or $K_{x}$ to denote a generic constant that depends only on $x$, whose numeric value may change from line to line unless otherwise specified. $a\lesssim_{x} b$ and $a\gtrsim_x b$ mean $a\leq C_x b$ and $a\geq C_x b$ respectively, and $a\asymp_x b$ means $a\lesssim_{x} b$ and $a\gtrsim_x b$ ($a\lesssim b$ means $a\leq Cb$ for some absolute constant $C$). $\mathcal{O}_{\mathbf{P}}$ and $\mathfrak{o}_{\mathbf{P}}$ denote the usual big and small O notation in probability. $\rightsquigarrow$ is reserved for weak convergence for general metric-space valued random variables. In this paper we will consider weak convergence of stochastic processes in the topology induced by  uniform convergence on compacta (that is, compact sets). A function $f$ is locally $C^{\alpha}$ at $x_0$ if it has a continuous $\alpha$-th derivative in a neighborhood of $x_0$. Lastly, $C([a, b])$ is the class of real-valued continuous functions defined on $[a,b] \subset \R$.

\section{Asymptotically pivotal LNE theory: convex regression}\label{section:pivot_limit}

\subsection{Review of the limit distribution theory}

First we state the assumptions.

\begin{assumption}\label{assump:local_smooth}
	Suppose that $f_0:[0,1]\to \R$ is a convex function and there exists some $\alpha \in \N$ such that $f_0$ is locally $C^\alpha$ at $x_0 \in (0,1)$ with $ f_0^{(\beta)}(x_0)=0$, $\beta=2,\ldots,\alpha-1$, and $ f_0^{(\alpha)}(x_0)\neq 0$. 
\end{assumption}
A simple Taylor's expansion of degree $\alpha - 2$ of $f_0^{(2)}(\cdot)$ at $x_0$ yields that $\alpha$ must be even and $f_0^{(\alpha)}(x_0) >0$ (cf.~ \cite[pp.~1305]{balabdaoui2009limit}). The canonical and most interesting case is $\alpha=2$.

\begin{assumption}\label{assump:design}
	Suppose the design points $\{X_i\}$ are either: (i) equally spaced fixed points on $[0,1]$, or (ii) i.i.d.~from the uniform distribution on $[0,1]$.
\end{assumption}

The equally spaced fixed design assumption can be relaxed to nearly equally spaced fixed design in the sense that the following two conditions are satisfied: (1) $\forall i <n, \frac{1}{C n}\leq X_{(i+1)}-X_{(i)}\leq \frac{C}{n}$ holds for some universal constant $C>0$, where $\{X_{(i)}\}$ are the order statistics of $\{X_i\}$, and (2) for $\mathbb{F}_n(x) = n^{-1} \sum_{i=1}^n \bm{1}_{\{X_i \le x\} }$, there exists some $\delta >0$ such that $
\sup_{x: |x - x_0| \le \delta} \big| \mathbb{F}_n(x) - x \big| = \mathfrak{o}(n^{-1/(2\alpha+1)})$. We assume a uniform distribution for the random design setting for simplicity of exposition. Our theory and the construction of CIs in this section can be easily modified to incorporate general design distributions; see Remark \ref{rmk:randon_design_general_dist}.

\begin{assumption}\label{assump:error}
	Suppose the errors $\{\xi_i\}$ are i.i.d.~mean-zero with variance $\sigma^2$ and sub-gaussian, that is, $\E \exp(t \xi_1^2)<\infty$ for $t$ in a neighborhood of $0$, and are independent of $\{X_i\}$ in the case of a random design. 
\end{assumption}

Here we have not tried to pin down the best possible moment condition on the errors. In fact, a sub-gaussian tail condition is assumed in \cite{mammen1991nonparametric,groeneboom2001estimation} in the nearly equally spaced fixed design setting, and a weaker sub-exponential tail condition is assumed in \cite{ghosal2017univariate} in the random design setting, for the limit distribution theory (see Theorem \ref{thm:limiting_dist_convex}) to hold.  For simplicity of presentation, we use a unified and stronger sub-gaussian condition. However, the reader should keep in mind that our main pivotal limit distribution theory (see Theorem \ref{thm:pivotal_limit_fcn}) below will work under the same conditions that validate the proof of Theorem \ref{thm:limiting_dist_convex} below.

Now we state the limit distribution theory for the convex LSE $\widehat{f}_n$ due to \cite{groeneboom2001estimation,ghosal2017univariate}.

\begin{theorem}\label{thm:limiting_dist_convex}
	Suppose Assumptions \ref{assump:local_smooth}-\ref{assump:error} hold. Then, 
	\begin{align*}
	\begin{pmatrix}
	(n/\sigma^2)^{\alpha/(2\alpha+1)}\big(\widehat{f}_n(x_0)-f_0(x_0)\big)\\
	(n/\sigma^2)^{(\alpha-1)/(2\alpha+1)}\big(\widehat{f}_n'(x_0)-f_0'(x_0)\big)
	\end{pmatrix}
    \rightsquigarrow
    \begin{pmatrix}
    d_{\alpha}^{(0)}(f_0,x_0)\cdot \mathbb{H}_{\alpha}^{(2)}(0)\\
    d_{\alpha}^{(1)}(f_0,x_0) \cdot \mathbb{H}_{\alpha}^{(3)}(0)
    \end{pmatrix}
    .
	\end{align*}
	Here
	\begin{align*}
	d_{\alpha}^{(0)}(f_0,x_0) &\equiv \bigg( \frac{f_0^{(\alpha)}(x_0)}{(\alpha+2)!}\bigg)^{1/(2\alpha+1)}, \quad  d_{\alpha}^{(1)}(f_0,x_0)  \equiv \bigg( \frac{f_0^{(\alpha)}(x_0)}{(\alpha+2)!}\bigg )^{3/(2\alpha+1)}, 
	\end{align*}
	and
	$\mathbb{H}_{\alpha}$ is an a.s.~uniquely well-defined random continuous function satisfying the following conditions:
	\begin{enumerate}
		\item For all $t \in \R$,
		\begin{align*}
		\mathbb{H}_{\alpha}(t)\geq \mathbb{Y}_{\alpha}(t)\equiv \int_{0}^{t} \mathbb{B}(s)\,\d{s} +t^{\alpha+2},
		\end{align*}
		where $\mathbb{B}$ is the standard two-sided Brownian motion starting from $0$.
		\item $\mathbb{H}_{\alpha}$ has a convex second derivative $\mathbb{H}_\alpha^{(2)}$.
		\item $\mathbb{H}_{\alpha}$ satisfies
		\begin{align*}
		\int_{-\infty}^{\infty} \big(\mathbb{H}_{\alpha}(t)-\mathbb{Y}_{\alpha}(t)\big)\,\d{\mathbb{H}_{\alpha}^{(3)}(t)} = 0.
		\end{align*}
	\end{enumerate}
	
\end{theorem}

\begin{remark}
	In words, $\mathbb{H}_{\alpha}$ is a.s.~determined as a piecewise cubic function that majorizes $\mathbb{Y}_{\alpha}$ with equality (touch points) taken at jumps of the piecewise constant nondecreasing function $\mathbb{H}_{\alpha}^{(3)}$. The process $\mathbb{H}_{\alpha}$ is called the ``invelope'' function of $\mathbb{Y}_{\alpha}$.
\end{remark}

\subsection{Asymptotically pivotal LNE theory I: Pointwise inference for the function and its derivative}

In this subsection, we consider the inference problem for the parameters $f_0(x_0)$ and $f_0'(x_0)$. We propose the following construction of CIs: let $[\widehat{u}(x_0),\widehat{v}(x_0)]$ be the ``maximal interval'' containing $x_0$ on which $\widehat{f}_n$ is linear, and
\begin{align}\label{def:CI_fcn_0_1}
\mathcal{I}_n^{(0)}(c_\delta^{(0)}) &\equiv \bigg[ \widehat{f}_n(x_0) \pm \frac{c_\delta^{(0)} \cdot \widehat{\sigma}}{ \sqrt{ n(\widehat{v}(x_0)-\widehat{u}(x_0) )} } \bigg],\\
\mathcal{I}_n^{(1)}(c_\delta^{(1)}) &\equiv \bigg[ \widehat{f}_n'(x_0) \pm \frac{c_\delta^{(1)} \cdot \widehat{\sigma}}{ \sqrt{ n(\widehat{v}(x_0)-\widehat{u}(x_0))^3} } \bigg], \nonumber
\end{align}
where $c_{\delta}^{(i)}$ $(i=0,1)$ are universal critical values determined only by the confidence level $1-\delta$, and will be detailed below (see Theorem~\ref{thm:CI_fcn}).
Here $\widehat{\sigma}$ is the square root of a consistent estimator of $\sigma^2$.

\begin{remark}
	To prevent potential ambiguity in the definition of $\widehat{u}(x_0)$ and $\widehat{v}(x_0)$ for finite samples, we require $[\widehat{u}(x_0), \widehat{v}(x_0)]$ to be the ``maximal interval'' which means: (i) the only interval containing $x_0$ if $x_0$ is not a kink of $\widehat{f}_n$, and (ii) the longer one (either one for equal length) if $x_0$ is a kink (so $x_0$ belongs to two intervals). This definition is primarily for practical concerns, as in theory any fixed point $x_0$ is a kink of $\widehat{f}_n$ with  vanishing probability in the large sample limit.
\end{remark}

Our proposal (\ref{def:CI_fcn_0_1}) for the CIs of $f_0(x_0)$ and $f_0'(x_0)$ is based on the following asymptotically pivotal LNE theory; see Appendix \ref{pf:pivotal_limit_fcn} for its proof.
\begin{theorem}\label{thm:pivotal_limit_fcn}
	Suppose Assumptions \ref{assump:local_smooth}-\ref{assump:error} hold. Then
	\begin{align*}
	\begin{pmatrix}
	\sqrt{n(\widehat{v}(x_0)-\widehat{u}(x_0))}\big(\widehat{f}_n(x_0)-f_0(x_0)\big)\\
	\sqrt{n(\widehat{v}(x_0)-\widehat{u}(x_0))^3}\big(\widehat{f}_n'(x_0)-f_0'(x_0)\big)
	\end{pmatrix}
	\rightsquigarrow \sigma\cdot  
	\begin{pmatrix}
	\mathbb{L}^{(0)}_\alpha\\
	 \mathbb{L}^{(1)}_\alpha
	\end{pmatrix}
	.
	\end{align*}
	Here $\mathbb{L}^{(0)}_{\alpha}$ and $\mathbb{L}^{(1)}_\alpha$ are a.s.~finite random variables defined by
	\begin{align*}
	\mathbb{L}^{(0)}_{\alpha} &\equiv \sqrt{h^\ast_{\alpha;-}+h^\ast_{\alpha;+}} \cdot \mathbb{H}_{\alpha}^{(2)}(0),\\
	\mathbb{L}^{(1)}_{\alpha} & \equiv \sqrt{\big(h^\ast_{\alpha;-}+h^\ast_{\alpha;+}\big)^3} \cdot \mathbb{H}_{\alpha}^{(3)}(0),
	\end{align*}
	where $h^\ast_{\alpha;-}$ (resp.~$h^\ast_{\alpha;+}$) is the absolute value of the location of the first touch point of the pair $(\mathbb{H}_{\alpha}, \mathbb{Y}_{\alpha})$ to the left (resp.~right) of $0$.
\end{theorem}

The proof of above theorem, at a high level, proceeds via a careful application of the continuous mapping theorem, by combining the proof of Theorem \ref{thm:limiting_dist_convex} and a suitable characterization of $\widehat{u}(x_0)$ and $\widehat{v}(x_0)$. Intuitively, one may wish to do so by considering $\widehat{u}(x_0)$ and $\widehat{v}(x_0)$ as two functionals $\mathcal{H}_\pm$ of the underlying process $(\mathbb{H}_n^{\mathrm{loc}})^{(2)}$, the finite sample version of $\mathbb{H}_{\alpha}^{(2)}$ defined in Theorem \ref{thm:limiting_dist_convex}, whose realizations are piecewise linear convex functions (see Appendix \ref{pf:preliminary} for a precise definition). However, it turns out that $\mathcal{H}_\pm$ are not continuous with respect to the topology induced by uniform convergence on compacta in which $(\mathbb{H}_n^{\mathrm{loc}})^{(2)}$ converges weakly to $\mathbb{H}_{\alpha}^{(2)}$; see (\ref{ineq:pivotal_lim_fcn_0}) for a counterexample. To overcome this difficulty, we employ a dual characterization of $\widehat{u}(x_0),\widehat{v}(x_0)$ using both $(\mathbb{H}_n^{\mathrm{loc}})^{(2)}$ and $\mathbb{H}_n^{\mathrm{loc}}$ (see (\ref{ineq:pivotal_lim_fcn_1})-(\ref{ineq:pivotal_lim_fcn_2})) that 
maintains suitable topological openness and closedness properties. In essence, the additional information on $\mathbb{H}_n^{\mathrm{loc}}$ shows that the convergence of the underlying process $(\mathbb{H}_n^{\mathrm{loc}})^{(2)}$ to its limit must occur inside the `continuity set' of the functionals $\mathcal{H}_\pm$ in the prescribed topology, and therefore $\widehat{u}(x_0)$ and $\widehat{v}(x_0)$, after proper scaling, converge in distribution to their white noise analogues. The universality of the limit then follows from Brownian scaling arguments; details can be found in Appendix \ref{pf:pivotal_limit_fcn}. 

In Figure \ref{fig:ecdf_L} below, we plot the approximate cumulative distribution functions of $\mathbb{L}_2^{(0)}$ and $\mathbb{L}_2^{(1)}$ based on simulation methods discussed in detail in Section \ref{section:simulation}. By time reflection $t\mapsto -t$ of the pair $(\mathbb{H}_2(t),\mathbb{Y}_2(t))$ in Theorem \ref{thm:limiting_dist_convex} and the symmetry of the two-sided Brownian motion about $0$, it is easy to see that $(-1)^\ell\mathbb{H}_2^{(\ell)}(-t)=_d \mathbb{H}_2^{(\ell)}(t)$ for $0\leq \ell\leq 3$ and $t\geq0$, so $\mathbb{H}_2^{(\ell)}(0)$ is symmetric for $\ell \in \{1,3\}$. Hence $\mathbb{L}^{(1)}_2$ is symmetric. Figure \ref{fig:ecdf_L0} also shows overwhelming numerical evidence in support of the symmetry of $\mathbb{L}^{(0)}_2$. It is an interesting open question to formally prove the conjectured symmetry of $\mathbb{L}_2^{(0)}$. Note that the symmetry of $\mathbb{L}^{(1)}_2$ and the conjectured symmetry of $\mathbb{L}^{(0)}_2$ lead to symmetric CIs proposed in (\ref{def:CI_fcn_0_1}).

\begin{figure}[!hbt]
	\centering
	\subfigure[$\mathbb{L}_2^{(0)}$]{
		\label{fig:ecdf_L0} 
		\includegraphics[width=0.48\textwidth]{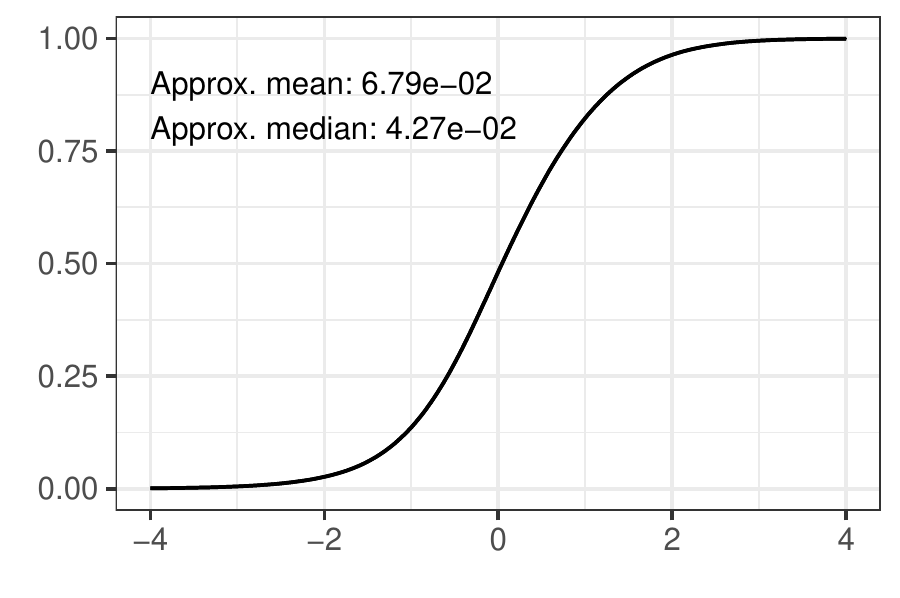}}
	\hspace{0\textwidth}
	\subfigure[$\mathbb{L}_2^{(1)}$]{
		\label{fig:ecdf_L1} 
		\includegraphics[width=0.48\textwidth]{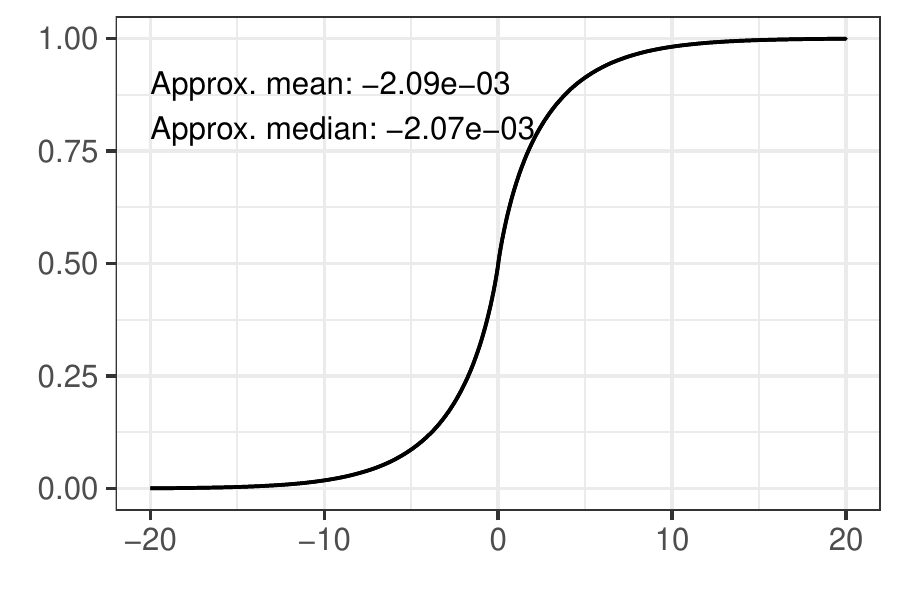}}
	\caption{Empirical distribution functions approximating the distributions of $\mathbb{L}_2^{(0)}$ and $\mathbb{L}_2^{(1)}$.}
	\label{fig:ecdf_L}
\end{figure}

\begin{remark}
	We compare Theorem \ref{thm:pivotal_limit_fcn} with the asymptotically pivotal LNE theory for isotonic regression developed in \cite{deng2020confidence}. Let $f_0$ be a univariate nondecreasing regression function in the regression model (\ref{model:regression}). Then the isotonic LSE $\widehat{f}_n^{(\textrm{iso})}$ is a piecewise constant nondecreasing function. Suppose Assumptions \ref{assump:local_smooth}-\ref{assump:error} hold (but assuming $f_0$ is nondecreasing in Assumption \ref{assump:local_smooth}), then $\widehat{f}_n^{(\textrm{iso})}$ satisfies 
	\begin{align}\label{ineq:iso_limit}
	(n/\sigma^2)^{1/(2+\alpha^{-1})}\big(\widehat{f}_n^{(\mathrm{iso})}(x_0)-f_0(x_0)\big)\rightsquigarrow \bigg(\frac{f_0^{(\alpha)}(x_0)}{(\alpha+1)!}\bigg)^{1/(2\alpha+1)}\cdot \mathbb{D}_{\alpha},
	\end{align}
	where $\mathbb{D}_{\alpha}$ is the slope at zero of the greatest convex minorant of $t \mapsto \mathbb{B}(t)+t^{\alpha+1}$; see~\cite{brunk1970estimation,wright1981asymptotic,han2019limit,han2019berry}. Let $\widehat{u}^{(\mathrm{iso})} (x_0)$ and $\widehat{v}^{(\mathrm{iso})} (x_0)$ be the left and right end-points of the constant piece of the isotonic LSE $\widehat{f}_n^{(\textrm{iso})}$ that contains $x_0$. Then under the same conditions as for the above limit theory (\ref{ineq:iso_limit}), \cite{deng2020confidence} proved the following asymptotically pivotal LNE theory:
	\begin{align}\label{ineq:iso_pivot_limit}
	\sqrt{n\big(\widehat{v}^{(\mathrm{iso})} (x_0)-\widehat{u}^{(\mathrm{iso})} (x_0) \big)} \big(\widehat{f}_n^{(\mathrm{iso})}(x_0)-f_0(x_0)\big)\rightsquigarrow \sigma\cdot \mathbb{L}^{(\mathrm{iso})}_\alpha,
	\end{align}
	where $\mathbb{L}^{(\mathrm{iso})}_\alpha$ does not depend on $f_0$. Theorem \ref{thm:pivotal_limit_fcn} can therefore be viewed as a `second-order analogue' of the limit theory (\ref{ineq:iso_pivot_limit}) in the context of convex regression, but with several notable differences:
	\begin{itemize}
		\item In the monotone setting, the local smoothness index $\alpha$ must be an odd integer, while in the convex setting, $\alpha$ must be an even integer. Hence the canonical assumption in the monotone setting is a non-vanishing first derivative, while in the convex setting the assumption is a non-vanishing second derivative.
		\item The assumption on the second derivative of $f_0$ and the information in $\widehat{u}(x_0),\widehat{v}(x_0)$ in the setting of convex regression is strong enough for a \emph{joint asymptotically pivotal LNE theory} for both $f_0(x_0)$ and $f_0'(x_0)$. As we will see below, it is also possible to derive asymptotically pivotal LNE theory for other local parameters, such as the anti-mode of the convex regression function, under similar local smoothness assumptions.  
		\item At a technical level, the isotonic estimate $\widehat{f}_n^{(\mathrm{iso})}(x_0)$ is the local average of the observations over the interval $[\widehat{u}^{(\mathrm{iso})}(x_0),\widehat{v}^{(\mathrm{iso})}(x_0)]$, while in the setting of convex regression, $\widehat{f}_n(x_0)$ is typically \emph{not} the local linear regression fit of the observations over the interval $[\widehat{u}(x_0),\widehat{v}(x_0)]$. This makes the technical analysis in Theorem \ref{thm:pivotal_limit_fcn} more involved and implicit compared to (\ref{ineq:iso_pivot_limit}). 
	\end{itemize}
\end{remark}

One particularly important and the canonical case is $\alpha=2$, where the CIs in (\ref{def:CI_fcn_0_1}) have asymptotically exact coverage and shrink at the optimal rate, as detailed below. See Appendix \ref{pf:CI_fcn} for a proof of the following result.
\begin{theorem}\label{thm:CI_fcn}
	Suppose Assumptions \ref{assump:local_smooth}-\ref{assump:error} hold with $\alpha =2$. Let $c_\delta^{(0)},c_\delta^{(1)}$ be chosen such that
	\begin{align}\label{def:cv_L}
	\Prob\big(\abs{\mathbb{L}_2^{(i)} }>c_\delta^{(i)}\big)=\delta,\quad i=0,1.
	\end{align}
	Then for any consistent variance estimator $\widehat{\sigma}$, the CIs in (\ref{def:CI_fcn_0_1}) satisfy 
	\begin{align*}
	\lim_{n \to \infty} \Prob_{f_0}\big(f_0(x_0) \in \mathcal{I}_n^{(0)}(c_\delta^{(0)})\big) = \lim_{n \to \infty} \Prob_{f_0}\big(f_0'(x_0) \in \mathcal{I}_n^{(1)}(c_\delta^{(1)})\big) = 1- \delta,
	\end{align*}
	and for any $\epsilon>0$, 
	\begin{align*}
	&\liminf_{n \to \infty} \bigg\{\Prob_{f_0}\bigg(\bigabs{\mathcal{I}_n^{(0)}(c_\delta^{(0)}) }<2c_\delta^{(0)} \mathfrak{g}_\epsilon^{(0)}\cdot (\sigma^2/n)^{2/5} d_{2}^{(0)}(f_0,x_0) \bigg)\\
	& \qquad \qquad \bigwedge \Prob_{f_0}\bigg(\bigabs{\mathcal{I}_n^{(1)}(c_\delta^{(1)}) }<2c_\delta^{(1)} \mathfrak{g}_\epsilon^{(1)}\cdot (\sigma^2/n)^{1/5} d_{2}^{(1)} (f_0,x_0) \bigg) \bigg\}\geq 1-\epsilon. 
	\end{align*}
	Here $\mathfrak{g}_\epsilon^{(i)} (i=0,1)$'s are constants that depend only on $\epsilon$. 
\end{theorem}

\begin{remark}\label{remark:optimal_rate}
The lengths of the proposed CIs shrink at the optimal rates in the sense that they adapt to the oracle rates which are locally asymptotically minimax optimal as shown in~\cite[Theorem 5.1]{groeneboom2001estimation}. In the oracle case where $f_0''(x_0)$ and $\sigma$ are both known, Theorem \ref{thm:limiting_dist_convex} implies an oracle CI for $f_0^{(i)}(x_0) \,(i=0,1)$ as 
\begin{align*}
\Big[ \widehat{f}^{(i)}_n(x_0) \pm (\sigma^2/n)^{(2-i)/5} d_{2}^{(i)}(f_0, x_0)  c_{\delta}\big( |\mathbb{H}_2^{(i + 2)}| \big)\Big],
\end{align*}
where $c_{\delta}\big( |\mathbb{H}_2^{(i + 2)}| \big)$ is the $(1- \delta)$-quantile of $|\mathbb{H}_2^{(i + 2)}|$. The length of this oracle CI shrinks at the rate $(\sigma^2/n)^{(2-i)/5} d_{2}^{(i)}(f_0, x_0)$, which is now shown by Theorem \ref{thm:CI_fcn} to be achievable using the proposed CI $\mathcal{I}_n^{(i)}(c_\delta^{(i)})$ in \eqref{def:CI_fcn_0_1}. 
\end{remark}

Let us now consider the case when $\alpha \neq 2$. Let $c_\delta^{(0)},c_\delta^{(1)}$ be chosen such that
	\begin{align}\label{cond:unif_tail}
	\sup_\alpha \Big\{\Prob\big(\abs{\mathbb{L}_{\alpha}^{(0)} }>c_\delta^{(0)}\big)\vee \Prob\big(\abs{\mathbb{L}_{\alpha}^{(1)} }>c_\delta^{(1)}\big) \Big\} \leq \delta.
	\end{align}
Then we may construct adaptive CIs for both $f_0(x_0)$ and $f_0'(x_0)$. We formalize this result in the following theorem; the proof is essentially the same as that of Theorem \ref{thm:CI_fcn} and is thus omitted.

\begin{theorem}\label{thm:adaptive_CI}
	Suppose Assumptions \ref{assump:local_smooth}-\ref{assump:error} hold. Let $c_\delta^{(0)},c_\delta^{(1)}$ be chosen according to (\ref{cond:unif_tail}). Then 
	\begin{align*}
	\liminf_{n \to \infty} \Big\{ \Prob_{f_0}\big(f_0(x_0) \in \mathcal{I}_n^{(0)}(c_\delta^{(0)})\big)\wedge \Prob_{f_0}\big(f_0'(x_0) \in \mathcal{I}_n^{(1)}(c_\delta^{(1)})\big) \Big\} \geq 1-\delta,
	\end{align*}
	and for any $\epsilon>0$, 
	\begin{align*}
	&\liminf_{n \to \infty} \bigg\{ \Prob_{f_0} \Big( \bigabs{\mathcal{I}_n^{(0)} (c_\delta^{(0)}) }<2 c_\delta^{(0)} \mathfrak{g}_{\epsilon,\alpha}^{(0)}\cdot (\sigma^2/n)^{\alpha/(2\alpha+1)} d_{\alpha}^{(0)} (f_0,x_0) \Big)\\
	& \qquad \qquad \bigwedge \Prob_{f_0} \Big( \bigabs{\mathcal{I}_n^{(1)} (c_\delta^{(1)}) }<2c_\delta^{(1)} \mathfrak{g}_{\epsilon,\alpha}^{(1)}\cdot (\sigma^2/n)^{(\alpha-1)/(2\alpha+1)} d_{\alpha}^{(1)} (f_0,x_0)  \Big) \bigg\} \geq 1-\epsilon. 
	\end{align*}
	Here $\mathfrak{g}_{\epsilon,\alpha}^{(i)}$'s (for $i=0,1$) are constants that depend only on $\epsilon,\alpha$, and $d_\alpha^{(i)}(f_0,x_0)$'s are defined in Theorem \ref{thm:limiting_dist_convex}.
\end{theorem}

The existence of critical values $c_\delta^{(i)}(i=0,1)$ satisfying (\ref{cond:unif_tail}) is verified in Corollary \ref{cor:unif_tail_pivot} ahead, so indeed adaptive CIs for both $f_0(x_0)$ and $f_0'(x_0)$ can be constructed by calibrating the critical values alone. 

\subsection{Asymptotically pivotal LNE theory  II: Inference for the anti-mode}

The above idea of constructing CIs for $f_0(x_0)$ and $f_0'(x_0)$ can be taken further to other `local parameters' for which a limit distribution theory is available. In this subsection we consider the inference problem for the anti-mode of the convex regression function $f_0$. More precisely, we define the \emph{anti-mode} of a convex function $f$ on $[0,1]$ as its smallest minimizer
\begin{align}\label{def:mode}
[f]_{ \mathrm{m} }  = [f]_{ \mathrm{m}^{+} } 
\equiv \min \big\{t: f(t) = \min_{u \in [0,1]} f(u) \big\}.
\end{align}
For a concave function $g$, the \emph{mode} is defined as its smallest maximizer $[g]_{\mathrm{m}^{-}} \equiv [-g]_{\mathrm{m}}$. We continue to use this notion of the mode for densities not necessarily convex or concave.

Let $m_0 \equiv [f_0]_{\mathrm{m}} \in (0,1)$ be the anti-mode of $f_0$ and $\widehat{m}_n \equiv [\widehat{f}_n]_{\mathrm{m}}$ be the anti-mode of the convex LSE $\widehat{f}_n$. Note that $\widehat{m}_n$ is a kink point of $\widehat{f}_n$. Let $ \widehat{u}_{\mathrm{m}}$ (resp.~$\widehat{v}_{\mathrm{m}}$) be the first kink of $\widehat{f}_n$ to the left (resp.~right) of $\widehat{m}_n$. We propose the following CI for $ m_0$:
\begin{align}\label{def:CI_mode}
\mathcal{I}_n^{\mathrm{m}}(c_\delta^{\mathrm{m}})\equiv \Big[ \widehat{m}_n \pm c_\delta^{\mathrm{m}} \big(\widehat{v}_{\mathrm{m}}- \widehat{u}_{\mathrm{m}}\big)\Big] \cap [0,1].
\end{align}
Here $c_\delta^{\mathrm{m}}$ is a universal critical value determined only by the confidence level $1-\delta$,  to be described below (see Theorem~\ref{thm:CI_mode}). For finite samples, when $\widehat{m}_n$ has no kink to its left (resp.~right), we simply let $\widehat{u}_{\mathrm{m}} = \widehat{m}_n$ (resp.~$\widehat{v}_{\mathrm{m}} = \widehat{m}_n$). It does not affect the limit theory as either case happens with vanishing probability for $m_0 \in (0,1)$. Note that $\widehat{v}_{\mathrm{m}} - \widehat{u}_{\mathrm{m}} > 0$ always holds unless $n=1$.

The above proposal \eqref{def:CI_mode} for a CI of $ m_0$ is based on the following asymptotically pivotal LNE theory (see  Appendix~\ref{pf:pivotal_limit_mode} for a proof of the following result). We will focus on the canonical case $\alpha=2$ for simplicity of exposition. 
\begin{theorem}\label{thm:pivotal_limit_mode}
	Suppose $f_0$ is locally $C^2$ at $ m_0 \in (0,1)$ with $f_0''( m_0)>0$, and that Assumptions \ref{assump:design}-\ref{assump:error} hold. Then
	\begin{align}\label{limit_mode_1}
	(n/\sigma^2)^{1/5}\big(\widehat{m}_n - m_0 \big)\rightsquigarrow d_2^{\mathrm{m}}(f_0)\cdot 
	\big[ \mathbb{H}_{2}^{(2)} \big]_{ \mathrm{m} } ,
	\end{align}
	where $d_2^{\mathrm{m}}(f_0) = \big(4!/f_0''( m_0)\big)^{2/5}$. Furthermore,
	\begin{align}\label{limit_mode_2}
	\frac{1}{\widehat{v}_{\mathrm{m}}- \widehat{u}_{\mathrm{m}}}\big(\widehat{m}_n- m_0\big)\rightsquigarrow \mathbb{M}_2.
	\end{align}
	Here $\mathbb{M}_2$ is an a.s.~finite random variable defined by
	\begin{align*}
	\mathbb{M}_2\equiv \frac{ \big[ \mathbb{H}_{2}^{(2)} \big]_{ \mathrm{m} } }{ h^\ast_{2, \mathrm{m} ;-}+ h^\ast_{2, \mathrm{m} ;+} },
	\end{align*}
	where $h^\ast_{2, \mathrm{m} ;-}$ (resp.~$h^\ast_{2, \mathrm{m} ;+} $) is the first kink of the random convex function $\mathbb{H}_{2}^{(2)}$ (defined in Theorem \ref{thm:limiting_dist_convex}) to the left (resp.~right) of its anti-mode $\big[ \mathbb{H}_{2}^{(2)} \big]_{ \mathrm{m} } $.
\end{theorem}

As we will mention later in Section \ref{subsection:log_concave}, \cite{balabdaoui2009limit} proved a limit distribution theory for the mode of the MLE of log-concave densities that is parallel to \eqref{limit_mode_1}. Although our proof strategy is similar to that in~\cite{balabdaoui2009limit}, the limit distribution theory \eqref{limit_mode_1} is new in convex regression.

The proof of the more significant result \eqref{limit_mode_2} is more difficult than the proofs of 
Theorem \ref{thm:pivotal_limit_fcn} and \eqref{limit_mode_1}. As $\widehat{u}_{\mathrm{m}}$ and $\widehat{v}_{\mathrm{m}}$ have to be characterized by processes with center $\widehat{m}_n$ that is \emph{random}, the continuous mapping argument in the proof of Theorem \ref{thm:pivotal_limit_fcn} and the argmax continuous mapping argument in the proof of \eqref{limit_mode_1} (originally developed in \cite{balabdaoui2009limit}) cannot be applied, at least directly. As a result, the weak convergence on compacta must be argued for the randomly centered processes. Details of the resulting technical complications and the proof can be found in Appendix \ref{pf:pivotal_limit_mode}.

In Figure \ref{fig:ecdf_M2} below, we plot the approximate cumulative distribution function of $\mathbb{M}_2$ based on simulation methods discussed in detail in Section \ref{section:simulation}. The distribution of $\mathbb{M}_2$ is symmetric due to the symmetry of the two-sided Brownian motion about $0$, which is strongly supported by Figure \ref{fig:ecdf_M2}.

\setlength{\belowcaptionskip}{-10pt}
	\begin{figure}[!hbt]
	\centering
	\includegraphics[width=0.6\textwidth]{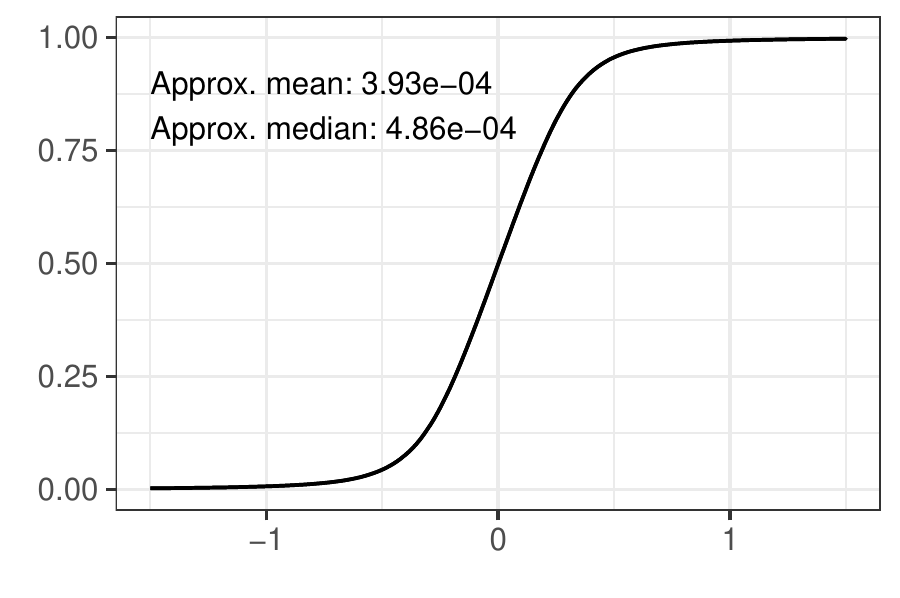}
	\caption{Empirical distribution function approximating the distribution of $\mathbb{M}_2$.}
	\label{fig:ecdf_M2}
\end{figure}

\begin{remark}
	As discussed in the Introduction, the second-order curvature of $\mathbb{H}_2$ contains sufficient information about $f_0(x_0), f_0'(x_0)$ and $m_0$. The joint distributional convergence of the LNEs for these local parameters can be established by a combination of the proofs of Theorems \ref{thm:pivotal_limit_fcn} and \ref{thm:pivotal_limit_mode} with minor changes.
\end{remark}

One striking difference of the CI (\ref{def:CI_mode}) compared to (\ref{def:CI_fcn_0_1}) is the \emph{complete elimination of the need to estimate the variance $\sigma^2$}. This is clearly reflected in the pivotal limiting distribution for $ m_0$ in the above theorem. The intuition is that both the quantities $\widehat{m}_n- m_0$ and $\widehat{v}_{\mathrm{m}}- \widehat{u}_{\mathrm{m}}$ have roughly the same order of magnitudes, so their ratio becomes pivotal in the limit.

As a straightforward consequence of Theorem \ref{thm:pivotal_limit_mode} (proved in Appendix \ref{pf:CI_mode}), the CI (\ref{def:CI_mode}) has asymptotically exact coverage and shrinks at the optimal length.

\begin{theorem}\label{thm:CI_mode}
	Let $c_\delta^{\mathrm{m}}$ be chosen such that
	\begin{align}\label{def:cv_M}
	\Prob\big(\abs{\mathbb{M}_2}>c_\delta^{\mathrm{m}}\big) = \delta. 
	\end{align}
	Then the CI in (\ref{def:CI_mode}) satisfies
	\begin{align*}
	\lim_{n \to \infty} \Prob_{ m_0}\big( m_0 \in \mathcal{I}_n^{\mathrm{m}}(c_\delta^{\mathrm{m}})\big) = 1- \delta,
	\end{align*}
	and for any $\epsilon>0$,
	\begin{align*}
	\liminf_{n \to \infty} \Prob_{ m_0}\bigg(\bigabs{\mathcal{I}_n^{\mathrm{m}} (c_\delta^{\mathrm{m}}) }<2c_\delta^{\mathrm{m}} \mathfrak{g}_\epsilon^{\mathrm{m}}\cdot (\sigma^2/n)^{1/5} d_2^{\mathrm{m}}(f_0)\bigg)\geq 1-\epsilon.
	\end{align*}
	Here $\mathfrak{g}_\epsilon^{\mathrm{m}}$ is a constant depending only on $\epsilon$, and $d_2^{\mathrm{m}}(f_0)$ is defined in Theorem \ref{thm:pivotal_limit_mode}.
\end{theorem}

\section{Inference in other convex/concave models}\label{section:other_convex}

In this section, we consider the inference problem for local parameters in other convexity/concavity constrainted models beyond the regression setting in Section \ref{section:pivot_limit}. The specific models we treat are:

\begin{enumerate}[label=(\roman*), leftmargin = 0.4in]
	\item log-concave density estimation \cite{balabdaoui2009limit},
	\item $s$-concave density estimation \cite{han2015approximation},
	\item convex nonincreasing density estimation \cite{groeneboom2001estimation},
	\item convex bathtub-shaped hazard function estimation  \cite{jankowski2009nonparametric}, and
	\item concave distribution function estimation from corrupted data \cite{jongbloed2009estimating}. 
\end{enumerate}

In each of the above settings, there is a natural estimator (not necessarily the LSE/MLE) exhibiting a non-standard limiting distribution characterized as in Theorem \ref{thm:limiting_dist_convex}. We will construct CIs for local parameters such as the value/derivative of the convexity/concavity constrained function at a fixed point, or the mode of a concave-transformed density. The constructions are largely inspired by the corresponding asymptotically pivotal LNE theories in the regression setting developed in Section \ref{section:pivot_limit}, and the resulting asymptotically pivotal LNE theories in these models follow a similar pattern to Theorems \ref{thm:pivotal_limit_fcn} and \ref{thm:pivotal_limit_mode} in convex regression. However, minor/major modifications are required for different models.

\subsection{Underlying machinery}

Suppose a piecewise linear estimator $\widehat{g}_n$ for a convex (resp.~concave) function $g_0$, where $g_0$ is locally $C^2$ at $x_0$ with $g_0''(x_0)>0$ (resp.~$g_0''(x_0)<0$), satisfies the following non-standard limit distribution theory with $(a,b) \in \R^2_{>0}$:
\begin{align}\label{ineq:underlying_1}
\begin{pmatrix}
n^{2/5}\big(\widehat{g}_n(x_0)-g_0(x_0)\big)\\
n^{1/5}\big(\widehat{g}_n'(x_0)-g_0'(x_0)\big)
\end{pmatrix}
\rightsquigarrow \pm
\begin{pmatrix}
H_{a,b}^{(2)}(0) \\
H_{a,b}^{(3)}(0)
\end{pmatrix}
.
\end{align}
Here we take $+$ in the convex case and $-$ in the concave case, and $H_{a,b}$ is a.s.~uniquely determined as a piecewise cubic function that majorizes a drifted integrated Brownian motion
\begin{align}\label{ineq:generic_gaussian_white_noise}
Y_{a,b}(t) \equiv a \int_{0}^{t} \mathbb{B}(s)\,\d{s} + b t^4,
\end{align}
with equality taken at jumps of the piecewise constant nondecreasing function $H_{a,b}^{(3)}$. Let $h^\ast_{a,b;-}$ (resp.~$h^\ast_{a,b;+}$) be the absolute value of the location of the first touch point of the pair $(H_{a,b},Y_{a,b})$ to the left (resp.~right) of $0$.

Although two nuisance parameters $a,b$ are present in the Gaussian white noise model (\ref{ineq:generic_gaussian_white_noise}), the really difficult nuisance parameter to estimate is $b$, which is typically related to the second derivative of the underlying unknown convex/concave function. This parameter cannot be estimated directly from a piecewise linear estimator $\widehat{g}_n$ as its second derivative is a.e.~$0$, and hence its elimination constitutes the main hurdle in the construction of a valid CI.

Inspired by the idea in Section \ref{section:pivot_limit} in the regression setting, let $[\widehat{u}(x_0),\widehat{v}(x_0)]$ be the maximal interval containing $x_0$ on which $\widehat{g}_n$ is linear. By a continuous mapping type argument, we may show that
\begin{align}\label{ineq:generic_piv_limit}
\begin{pmatrix}
\sqrt{n(\widehat{v}(x_0)-\widehat{u}(x_0))}\big(\widehat{g}_n(x_0)-g_0(x_0)\big)\\
\sqrt{n(\widehat{v}(x_0)-\widehat{u}(x_0))^3}\big(\widehat{g}_n'(x_0)-g_0'(x_0)\big)
\end{pmatrix}
\rightsquigarrow
\pm 
\begin{pmatrix}
\sqrt{h^\ast_{a,b;+}+h^\ast_{a,b;-} }\cdot H_{a,b}^{(2)}(0) \\
\sqrt{ (h^\ast_{a,b;+}+h^\ast_{a,b;-})^3 }\cdot  H_{a,b}^{(3)}(0)
\end{pmatrix}
.
\end{align}
Let $\mathbb{H}\equiv \mathbb{H}_2, \mathbb{Y}\equiv \mathbb{Y}_2, h^\ast_{\pm}\equiv h^\ast_{2,\pm}$ be defined as in Theorems \ref{thm:limiting_dist_convex} and \ref{thm:pivotal_limit_fcn} with $\alpha=2$. Let $\gamma_0,\gamma_1$ be such that
\begin{align*}
\gamma_0\gamma_1^{3/2} =a,\quad \gamma_0\gamma_1^4 =b.
\end{align*}
Then a standard Brownian scaling shows that
\begin{align*}
\gamma_0\mathbb{Y}(\gamma_1 t) = Y_{a,b}(t),
\end{align*}
and hence
\begin{align*}
H_{a,b}^{(2)}(t)=\gamma_0\gamma_1^2 \mathbb{H}^{(2)} (\gamma_1 t), \quad H_{a,b}^{(3)}(t)=\gamma_0\gamma_1^3 \mathbb{H}^{(3)}(\gamma_1 t),\quad h^\ast_{\pm}= \gamma_1 h^\ast_{a,b;\pm}.
\end{align*}
Now the limit distributions in (\ref{ineq:generic_piv_limit}) become
\begin{align*}
\sqrt{h^\ast_{a,b;+}+h^\ast_{a,b;-} }\cdot H_{a,b}^{(2)}(0) &=_d a\cdot \sqrt{ h^\ast_{+} + h^\ast_{-} }\cdot \mathbb{H}^{(2)}(0)\equiv a\cdot \mathbb{L}^{(0)}\\
\sqrt{ (h^\ast_{a,b;+}+h^\ast_{a,b;-} )^3}\cdot H_{a,b}^{(3)}(0) &=_d a\cdot \sqrt{( h^\ast_{+} + h^\ast_{-})^3 }\cdot \mathbb{H}^{(3)}(0) \equiv  a\cdot \mathbb{L}^{(1)},
\end{align*}
where $\mathbb{L}^{(\cdot)}$'s are by definition universal random variables. Hence, with any consistent estimator $\widehat{a}_n$ of $a$, we may construct CIs for $g_0(x_0),g_0'(x_0)$ as
\begin{align}\label{def:CI_generic_0_1}
\mathcal{I}_{n, \ast}^{(0)} (c_\delta^{(0)}) &\equiv \bigg[ \widehat{g}_n(x_0) \pm \frac{\widehat{a}_n\cdot c_\delta^{(0)}}{ \sqrt{ n(\widehat{v}(x_0)-\widehat{u}(x_0) )} } \bigg], \\
\mathcal{I}_{n, \ast}^{(1)}(c_\delta^{(1)}) &\equiv \bigg[ \widehat{g}_n'(x_0) \pm \frac{\widehat{a}_n\cdot c_\delta^{(1)}}{ \sqrt{ n(\widehat{v}(x_0)-\widehat{u}(x_0))^3} }\bigg]. \nonumber
\end{align}
These CIs have asymptotically exact coverage, and can be shown to shrink at optimal length, provided the critical values $c_\delta^{(i)}$ are chosen to be the corresponding quantiles for the universal random variables $\mathbb{L}^{(i)}$, for $i=0,1$. 

For mode estimation, let $m_0 \equiv [g_0]_{\mathrm{m}}$ (resp.~$m_0 \equiv [g_0]_{\mathrm{m}^-}$) and $\widehat{m}_n\equiv [\widehat{g}_n]_{\mathrm{m}}$ (resp.~$\widehat{m}_n\equiv [\widehat{g}_n]_{\mathrm{m}^-}$) be the anti-mode (resp.~mode) of the estimator $\widehat{g}_n$, where $g_0$ is convex (resp.~concave) and satisfies $g_0''(m_0)>0$ (resp.~$g_0''(m_0)<0$). Suppose $\widehat{m}_n$ satisfies the `argmin' (resp.~`argmax') version of (\ref{ineq:underlying_1}), that is, 
\begin{align}\label{ineq:underlying_2}
n^{1/5}(\widehat{m}_n-m_0)\rightsquigarrow \big[H_{a,b}^{(2)}\big]_{\mathrm{m}}.
\end{align}
Let $ \widehat{u}_{\mathrm{m}}$ (resp.~$\widehat{v}_{\mathrm{m}}$) be the first kink of $\widehat{g}_n$ to the left (resp.~right) of $\widehat{m}_n$. Then a continuous mapping type argument leads to
\begin{align}
\frac{1}{\widehat{v}_{\mathrm{m}}-\widehat{u}_{\mathrm{m}} }(\widehat{m}_n-m_0)\rightsquigarrow
\frac{ \big[ H_{a,b}^{(2)} \big]_{ \mathrm{m} } }{ h^\ast_{a,b, \mathrm{m} ;-}+ h^\ast_{a,b, \mathrm{m} ;+} }
\end{align}
where $h^\ast_{a,b, \mathrm{m} ;-}$ (resp.~$h^\ast_{a,b, \mathrm{m} ;+} $) is the first kink of $H_{a,b}^{(2)}$ to the left (resp.~right) of $  \big[ H_{a,b}^{(2)} \big]_{ \mathrm{m} } $. Using a similar scaling argument as above, one may show that the right hand side of the above display is pivotal, that is,
\begin{align*}
\frac{ \big[ H_{a,b}^{(2)} \big]_{ \mathrm{m} } }{ h^\ast_{a,b, \mathrm{m} ;-}+ h^\ast_{a,b, \mathrm{m} ;+} } = \mathbb{M},
\end{align*}
for some universal random variable $\mathbb{M}$. Hence we may construct a CI for $m_0$ as
\begin{align}\label{def:CI_generic_0_2}
\mathcal{I}_{n,\ast}^{\mathrm{m}}(c_\delta^{\mathrm{m}})\equiv \Big[ \widehat{m}_n \pm c_\delta^{\mathrm{m}} \big(\widehat{v}_{\mathrm{m}}- \widehat{u}_{\mathrm{m}}\big)\Big],
\end{align}
provided the critical value $c_\delta^{\mathrm{m}}$ is chosen to be the corresponding quantile for the universal random variable $\mathbb{M}$.

	\begin{remark}\label{rmk:randon_design_general_dist}
		In the regression setting with a random design, Theorems \ref{thm:pivotal_limit_fcn} and \ref{thm:pivotal_limit_mode} in Section \ref{section:pivot_limit} are stated under the uniform distribution on $[0,1]$. We may use this general machinery to easily extend our conclusions to a general design distribution $P$ on $[0,1]$, that is, $X_i \overset{\mathrm{i.i.d.}}{\sim} P$ for all $1 \le i \le n$. Let the Lebesgue density $\pi$ of $P$ be locally continuous at $x_0 \in (0,1)$ with $\pi(x_0)>0$. Suppose that $f_0$ is locally $C^2$ at $x_0$ with $f_0''(x_0)>0$. After some calculations, we obtain the `driving process':
		\begin{align*}
		\mathbb{Y}(t;f_0)\equiv \frac{\sigma}{\sqrt{\pi(x_0)}}\int_{0}^{t} \mathbb{B}(s)\,\d{s}+\frac{f_0''(x_0)}{4!}t^4.
		\end{align*}
		Hence the LSE $\widehat{f}_n$ satisfies (\ref{ineq:generic_piv_limit}) with $a= \sigma/ \sqrt{\pi(x_0)}$ and $b=f_0''(x_0)/4!$. A consistent estimator for the nuisance parameter $a$ can be taken as
		\begin{align*}
		\widehat{a}_n\equiv \textstyle \widehat{\sigma} \Big( \sum_i \bm{1}_{\{\widehat{u}(x_0)\leq X_i\leq \widehat{v}(x_0)\} } \Big/ \big\{ n (\widehat{v}(x_0)-\widehat{u}(x_0))\big\} \Big)^{-1/2},
		\end{align*}
		where $\widehat{\sigma}^2$ is a consistent estimator for $\sigma^2$. We may modify the CIs for the parameters $f_0(x_0),f_0'(x_0)$ in (\ref{def:CI_fcn_0_1}) by replacing $\widehat{\sigma}$ therein with $\widehat{a}_n$. As the generic CI in (\ref{def:CI_generic_0_2}) is free of the scale parameters $a$ and $b$, we may continue to use the same CI for the anti-mode $m_0$ as defined in (\ref{def:CI_mode}) in the regression setting with a general design distribution.
	\end{remark}

	In the next few subsections we work out this machinery in concrete models mentioned at the beginning of this section.

\subsection{Log-concave density estimation}\label{subsection:log_concave}

Suppose that we observe i.i.d.~data $X_1,\ldots,X_n$ from a log-concave density $f_0 \equiv \exp(\varphi_0)$ where $\varphi_0$ is a proper concave function on $\R$. Let $\widehat{f}_n = \exp(\widehat{\varphi}_n)$ be the log-concave MLE based on $X_1,\ldots,X_n$, that is,
\begin{align}\label{def:log_conc}
\widehat{\varphi}_n &\equiv \underset{{\varphi: \mathrm{\,concave }, \int_{\R} e^\varphi =1} }{\mathrm{arg\, max}}  
\int_{-\infty}^{\infty} \varphi(x)\,\d{\mathbb{F}_n(x)} \\
& = \underset{ \varphi: \rm{\,concave } }{\mathrm{arg\, max}}  
\bigg\{\int_{-\infty}^{\infty} \varphi(x)\,\d{\mathbb{F}_n(x)} - \int_{-\infty}^{\infty} e^{\varphi(x)}\,\d{x}\bigg\}.\nonumber
\end{align}
Here $\mathbb{F}_n$ is the empirical distribution function of the sample $X_1,\ldots,X_n$. It can be shown that $\widehat{\varphi}_n$ is a piecewise linear concave function with possible kinks at the data points. 

The class of log-concave densities is statistically appealing due to its several nice closure properties with respect to marginalization, conditioning and convolution operations (see e.g.,~\cite{saumard2014logconcavity}). The estimation of log-concave densities can be carried out using the method of maximum likelihood, and has been investigated by many authors; see~\cite{walther2002detecting,cule2010maximum,cule2010theoretical,dumbgen2009maximum,dumbgen2011approximation,pal2007estimating,seregin2010nonparametric,kim2016global,kim2016adaptation,feng2018adaptation,doss2013global,barber2020local,han2019}, just to name a few. The log-concave shape constraint also has applications in other settings; see, e.g.,  \cite{muller2009smooth,samworth2012independent,chen2013smoothed,balabdaoui2018inference}. We refer the reader to \cite{saumard2014logconcavity,samworth2018recent} for comprehensive reviews.

We first consider inference for the parameters $f_0(x_0)$ and $f_0'(x_0)$. Let $[\widehat{u}(x_0),\widehat{v}(x_0)]$ be the maximal interval containing $x_0$ on which $\widehat{\varphi}_n$ is linear, and
\begin{align}\label{def:CI_log_concave_0_1}
\mathcal{I}_{n, \mathrm{lc}}^{(0)} (c_\delta^{(0)}) &\equiv \bigg[ \widehat{f}_n(x_0) \pm \frac{\sqrt{\widehat{f}_n(x_0)}\cdot c_\delta^{(0)}}{ \sqrt{ n(\widehat{v}(x_0)-\widehat{u}(x_0) )} } \bigg] \cap [0, \infty),\\
\mathcal{I}_{n, \mathrm{lc}}^{(1)} (c_\delta^{(1)}) &\equiv \bigg[ \widehat{f}_n'(x_0) \pm \frac{\sqrt{\widehat{f}_n(x_0)}\cdot c_\delta^{(1)}}{ \sqrt{ n(\widehat{v}(x_0)-\widehat{u}(x_0))^3} } \bigg]. \nonumber
\end{align}
The above CIs are based on the following result, proved in Appendix \ref{pf:pivotal_limit_log_concave}.
\begin{theorem}\label{thm:pivotal_limit_log_concave}
	Suppose $f_0$ is a log-concave density with $f_0=e^{\varphi_0}$ for some concave function $\varphi_0$, $f_0(x_0)>0$ and $\varphi_0$ is locally $C^2$ at $x_0$ with $\varphi_0''(x_0)<0$.
	\begin{enumerate}
		\item With $\mathbb{L}^{(i)}_2(i=0,1)$ defined in Theorem \ref{thm:pivotal_limit_fcn},
		\begin{align*}
		\begin{pmatrix}
		\sqrt{n(\widehat{v}(x_0)-\widehat{u}(x_0))}\big(\widehat{f}_n(x_0)-f_0(x_0)\big)\\
		\sqrt{n(\widehat{v}(x_0)-\widehat{u}(x_0))^3}\big(\widehat{f}_n'(x_0)-f_0'(x_0)\big)
		\end{pmatrix}
		\rightsquigarrow - \sqrt{f_0(x_0)}\cdot 
		\begin{pmatrix}
		\mathbb{L}^{(0)}_2\\
		\mathbb{L}^{(1)}_2
		\end{pmatrix}.
		\end{align*}
		\item Let $c_\delta^{(0)},c_\delta^{(1)}$ be chosen such that $
		\Prob\big(\abs{\mathbb{L}_2^{(i)} }>c_\delta^{(i)}\big)=\delta$ for $i=0,1$, then the CIs in (\ref{def:CI_log_concave_0_1}) satisfy 
		\begin{align*}
		&\lim_{n \to \infty} \Prob_{f_0}\big(f_0(x_0) \in \mathcal{I}_{n, \mathrm{lc}}^{(0)} (c_\delta^{(0)}) \big) = \lim_{n \to \infty} \Prob_{f_0}\big(f_0'(x_0) \in \mathcal{I}_{n, \mathrm{lc}}^{(1)} (c_\delta^{(1)})\big) = 1- \delta.
		\end{align*}
		\item For any $\epsilon>0$, 
		\begin{align*}
		&\liminf_{n \to \infty} \bigg\{\Prob_{f_0}\Big(\bigabs{\mathcal{I}_{n, \mathrm{lc}}^{(0)} (c_\delta^{(0)})  }<2c_\delta^{(0)} \mathfrak{g}_\epsilon^{(0)}\cdot n^{-2/5} d_{2, \mathrm{lc}}^{(0)} (f_0,x_0)\Big)\\
		& \qquad \qquad \bigwedge \Prob_{f_0}\Big(\bigabs{\mathcal{I}_{n, \mathrm{lc}}^{(1)} (c_\delta^{(1)}) }<2c_\delta^{(1)} \mathfrak{g}_\epsilon^{(1)}\cdot n^{-1/5} d_{2, \mathrm{lc}}^{(1)} (f_0,x_0)\Big)\bigg\}\geq 1-\epsilon. 
		\end{align*}
		Here $\mathfrak{g}_\epsilon^{(i)} (i=0,1)$'s are constants that depend only on $\epsilon$, and
		\begin{align*}
		d_{2, \mathrm{lc}}^{(0)} (f_0,x_0) = \bigg(\frac{f_0(x_0)^3 \abs{\varphi_0''(x_0)} }{4!}\bigg)^{1/5},\quad d_{2, \mathrm{lc}}^{(1)}(f_0,x_0) =\bigg(\frac{f_0(x_0)^4 \abs{\varphi_0''(x_0)}^3 }{(4!)^3}\bigg)^{1/5}.
		\end{align*}
	\end{enumerate}
\end{theorem}

Clearly, the above asymptotically pivotal LNE theory shows that the CIs in (\ref{def:CI_log_concave_0_1}) have asymptotically exact coverage. \cite{balabdaoui2009limit} establish the pointwise limit distribution theory, as in \eqref{ineq:underlying_1} for $\widehat{\varphi}_n$ with $a=1/\sqrt{f_0(x_0)}, b= -\varphi_0''(x_0)/4!$ and then, by the delta method, the limit distribution theory for the log-concave MLE $\widehat{f}_n$, that is, 
\begin{align}\label{limit_log_concave}
\begin{pmatrix}
n^{2/5}\big(\widehat{f}_n(x_0)- f_0(x_0)\big)\\
n^{1/5}\big(\widehat{f}_n'(x_0)- f_0'(x_0)\big)
\end{pmatrix}
\rightsquigarrow 
\begin{pmatrix}
d_{2, \mathrm{lc}}^{(0)} (f_0,x_0) \cdot \mathbb{H}_2^{(2)}(0) \\
d_{2, \mathrm{lc}}^{(1)} (f_0,x_0) \cdot \mathbb{H}_2^{(3)}(0)
\end{pmatrix}.
\end{align}
By Theorem~\ref{thm:pivotal_limit_log_concave}-(3) and the above display we see that the CIs in \eqref{def:CI_log_concave_0_1} shrink at optimal length (as in Remark \ref{remark:optimal_rate}). 

Note that in the current setting $f_0$ by itself is not convex/concave, so the proofs need to be carried out at the underlying convex/concave level. 

Next we consider inference for the mode of the log-concave density $f_0$.~\cite{balabdaoui2009limit} obtained the pointwise limit distribution theory (\ref{ineq:underlying_2}) for the plug-in mode estimator $\widehat{m}_n\equiv \big[ \widehat{\varphi}_n \big]_{ \mathrm{m}^{-}}$ with $a=1/\sqrt{f_0(m_0)}$ and $b=-\varphi_0''(m_0)/4!$. We construct below a CI for $m_0$ as in (\ref{def:CI_generic_0_2}).

Note $\widehat{m}_n$ is a kink point of $\widehat{\varphi}_n$. Let $ \widehat{u}_{\mathrm{m}}$ (resp.~$\widehat{v}_{\mathrm{m}}$) be the first kink of $\widehat{\varphi}_n$ to the left (resp.~right) of $\widehat{m}_n$. We propose the following CI:
\begin{align}\label{def:CI_mode_log_concave}
\mathcal{I}_{n, \mathrm{lc}}^{\mathrm{m}} (c_\delta^{\mathrm{m}})\equiv \Big[\widehat{m}_n \pm c_\delta^{\mathrm{m}} \big(\widehat{v}_{\mathrm{m}}- \widehat{u}_{\mathrm{m}}\big) \Big].
\end{align}
The validity of the above CI is based on the following result, proved in Appendix \ref{pf:pivotal_limit_log_concave}.

\begin{theorem}\label{thm:pivotal_limit_mode_log_concave}
	Suppose $f_0$ is a log-concave density with $f_0=e^{\varphi_0}$ for some concave function $\varphi_0$, and $f_0$ is locally $C^2$ at $ m_0$ with $f_0''( m_0)<0$, where $ m_0	 \equiv [\varphi_0]_{ \mathrm{m}^{-}} $ is the mode of $f_0$.
	\begin{enumerate}
		\item With $\mathbb{M}_2$ defined in Theorem \ref{thm:pivotal_limit_mode},
		\begin{align*}
		\frac{1}{\widehat{v}_{\mathrm{m}}- \widehat{u}_{\mathrm{m}}}\big(\widehat{m}_n- m_0\big)\rightsquigarrow \mathbb{M}_2.
		\end{align*}
		\item Let $c_\delta^{\mathrm{m}}$ be chosen such that $
		\Prob\big(\abs{\mathbb{M}_2}>c_\delta^{\mathrm{m}}\big) = \delta$, then the CI in (\ref{def:CI_mode_log_concave}) satisfies
		\begin{align*}
		\lim_{n \to \infty} \Prob_{ m_0}\big( m_0 \in \mathcal{I}_{n, \mathrm{lc}}^{\mathrm{m}} (c_\delta^{\mathrm{m}})\big) = 1- \delta.
		\end{align*}
		\item For any $\epsilon>0$ and $d_{2, \mathrm{lc}}^{\mathrm{m}}(f_0) = \big\{ (4!)^2 f_0( m_0) \big/ (f_0''( m_0))^2\big\}^{1/5}$,
		\begin{align*}
		\liminf_{n \to \infty} \Prob_{ m_0}\bigg(\bigabs{\mathcal{I}_{n, \mathrm{lc}}^{\mathrm{m}}  (c_\delta^{\mathrm{m}}) }<2c_\delta^{\mathrm{m}} \mathfrak{g}_\epsilon^{\mathrm{m}}\cdot n^{-1/5} d_{2, \mathrm{lc}}^{\mathrm{m}}(f_0)\bigg)\geq 1-\epsilon.
		\end{align*}
		Here $\mathfrak{g}_\epsilon^{\mathrm{m}}$ is a constant depending only on $\epsilon$.
	\end{enumerate}
\end{theorem}
As in Theorem \ref{thm:pivotal_limit_log_concave}, the above asymptotically pivotal LNE theory shows that the CI in (\ref{def:CI_mode_log_concave}) has asymptotically exact coverage. Comparing the above result with the limit distribution theory for the plug-in mode estimator $\widehat{m}_n$ established in \cite{balabdaoui2009limit} (as in~\eqref{ineq:underlying_2}),
\begin{align}\label{limit_log_concave_mode}
n^{1/5}(\widehat{m}_n-m_0)\rightsquigarrow d_{2, \mathrm{lc}}^{\mathrm{m}}(f_0) \cdot [\mathbb{H}^{(2)}_2]_{\mathrm{m}},
\end{align}
we see that the proposed CI shrinks at optimal length.

Doss and Wellner~\cite{doss2016inference} developed a different procedure for inference of the mode $ m_0$ based on the LRT. More specifically, consider the following hypothesis testing problem:
\begin{align*}
H_0: [\varphi_0]_{ \mathrm{m}^{-}} =  m_0\qquad \textrm{versus}\qquad  H_1: [\varphi_0]_{ \mathrm{m}^{-}} \neq  m_0.
\end{align*}
Let $\widehat{f}_{n,0}$ be the mode-constrained log-concave MLE, that is, $\widehat{f}_{n,0}= e^{\widehat{\varphi}_{n,0}}$, where
\begin{align}\label{def:log_conc_constrained}
\widehat{\varphi}_{n,0} &\equiv 
\underset{ \substack{\varphi: \mathrm{ concave }, \varphi( m_0)\geq \varphi(x), x \in \R } }{\mathrm{arg\,max} }
\bigg\{\int_{-\infty}^{\infty} \varphi(x)\,\d{\mathbb{F}_n(x)} - \int_{-\infty}^{\infty} e^{\varphi(x)}\,\d{x}\bigg\}.
\end{align}
The LRT statistic is now defined as
\begin{align}\label{def:DW_LRT}
2\log \lambda_n( m_0)\equiv 2n\Prob_n\big(\log \widehat{f}_n - \log \widehat{f}_{n,0}\big) = 2n\Prob_n\big(\widehat{\varphi}_n-\widehat{\varphi}_{n,0}\big),
\end{align}
where $\Prob_n = n^{-1}\sum_{i=1}^n \delta_{X_i}$ is the empirical measure based on i.i.d.~observations $X_1,\ldots,X_n$.~\cite{doss2016inference} proved the following result: Under the same conditions as in Theorem \ref{thm:pivotal_limit_mode_log_concave},
\begin{align}\label{ineq:doss_wellner_piv_limit}
2 \log \lambda_n( m_0) \rightsquigarrow \mathbb{K},
\end{align}
where $\mathbb{K}$ has a universal limiting distribution. A CI for $ m_0$ can then be obtained by inverting the above LRT statistic: Let
\begin{align}\label{def:CI_DW_mode}
\mathcal{I}_{n, \mathrm{lc}}^{(m),\textrm{DW}} (d_\delta)\equiv \{ m_0: 2 \log \lambda_n( m_0)\leq d_\delta\},
\end{align}
where $d_\delta$ is chosen such that $\Prob(\mathbb{K}>d_\delta)=\delta$. Then 
\begin{align*}
\lim_{n \to \infty} \Prob_{ m_0}\big( m_0 \in \mathcal{I}_{n, \mathrm{lc}}^{(m),\textrm{DW}} (d_\delta) \big) =  \Prob\big(\mathbb{K}\leq d_\delta\big) = 1-\delta.
\end{align*}
It is easy to see that the implementation of (\ref{def:CI_DW_mode}) requires the computation of many mode-constrained log-concave MLEs, whereas our proposed CI (\ref{def:CI_mode_log_concave}) only requires the computation of the log-concave MLE once. On the technical side, the proof of (\ref{ineq:doss_wellner_piv_limit}) in \cite{doss2016inference} is substantially more difficult and involved compared to the corresponding results in the problem of inference in monotone models \cite{banerjee2001likelihood,banerjee2007likelihood,groeneboom2015nonparametric}, as the difference of the unconstrained and constrained log-concave MLEs $\widehat{f}_n$ and $\widehat{f}_{n,0}$ outside of a $\mathcal{O}_{\mathbf{P}}(n^{-1/5})$ local neighborhood of $ m_0$ is much harder to control. However, as shown in our proposal (\ref{def:CI_mode_log_concave}) and the resulting asymptotically pivotal LNE theory in Theorem \ref{thm:pivotal_limit_mode_log_concave}, it suffices to take advantage of a data-driven $\mathcal{O}_{\mathbf{P}}(n^{-1/5})$ local neighborhood of $ m_0$ using the information in $ \widehat{u}_{\mathrm{m}},\widehat{v}_{\mathrm{m}}$. For a detailed numerical comparison between Doss-Wellner CI (\ref{def:CI_DW_mode}) and our proposal (\ref{def:CI_mode_log_concave}), we refer the reader to Section \ref{subsection:DW_comparison}.

\subsection{$s$-concave density estimation}
Define for $\theta \in (0,1)$,
\begin{align*}
M_s(a,b;\theta)\equiv
\begin{cases}
\big((1-\theta)a^s+\theta b^s\big)^{1/s}, & s\neq 0, a,b > 0,\\
0, & s <0, ab = 0,\\
a^{1-\theta}b^\theta, &s=0,\\
a\wedge b, &s=-\infty.
\end{cases}
\end{align*}
A density $p$ on $\R$ is called $s$-concave, that is, $p \in \mathcal{P}_s$, if and only if for all $x_0,x_1\in \R$ and $\theta \in (0,1)$, $p\big((1-\theta)x_0+\theta x_1\big)\geq M_s(p(x_0),p(x_1);\theta)$. It is easy to see that the density $p$ has  the form $p=\varphi_+^{1/s}$ for some concave function $\varphi$ if $s>0$, $p=\exp(\varphi)$ for some concave $\varphi$ if $s=0$, and $p=\varphi_+^{1/s}$ for some convex $\varphi$ if $s<0$. The function classes $\mathcal{P}_s$ are nested in $s$ in that for every $r>0>s$, we have $\mathcal{P}_r\subset \mathcal{P}_0\subset \mathcal{P}_s\subset \mathcal{P}_{-\infty}$.

The class of $s$-concave densities generalizes that of log-concave densities to a large extent by allowing polynomial tails for the densities. The study of the MLE of $s$-concave densities was initiated in \cite{seregin2010nonparametric} and its global rates of convergence was investigated in \cite{doss2013global,han2019}. Here we will be interested in the regime $ -1<s<0$ and the R\'enyi divergence estimator introduced in \cite{koenker2010quasi} and further studied in \cite{han2015approximation}. Suppose $X_1,\ldots,X_n$ are i.i.d.~samples from a density $f_0 \in \mathcal{P}_s$. Let $\beta_s\equiv 1+1/s<0$ and
\begin{align}\label{def:s_conc}
\widehat{\varphi}_{n,s}\equiv 
\underset{\varphi\geq 0: \mathrm{\,convex }}{\mathrm{arg\, max}}
\bigg\{\int_{-\infty}^{\infty} \varphi(x)\,\d{\mathbb{F}_n(x)} + \frac{1}{\abs{\beta_s}} \int_{-\infty}^{\infty} (\varphi(x))^{\beta_s}\,\d{x}\bigg\}.
\end{align}
\cite{koenker2010quasi} and \cite{han2015approximation} showed that $\widehat{\varphi}_{n,s}$ exists and is unique with probability $1$. Let $\widehat{f}_{n,s}\equiv\widehat{\varphi}_{n,s}^{1/s}$. The connection between (\ref{def:s_conc}) and (\ref{def:log_conc}) can be seen clearly from the dual formulations; we refer the reader to \cite{han2015approximation} for more details.~\cite{han2015approximation}  obtained the limit  distribution theory (\ref{ineq:underlying_1}) for $\widehat{\varphi}_{n,s}$, (\ref{ineq:underlying_2}) for the plug-in mode estimator $\widehat{m}_{m,s}\equiv \big[ \widehat{\varphi}_{n,s} \big]_{ \mathrm{m}}$, with $a=1/\sqrt{f_0(x_0)}, b=r_s\varphi_0''(x_0)/(\varphi_s(x_0)4!)$ and $r_s \equiv -1/s>0$. The limit distribution theory for $\widehat{f}_{n,s}$ can then be obtained by the delta method.

Now we consider the inference problem. The proposal below is similar to (\ref{def:CI_log_concave_0_1}) and (\ref{def:CI_mode_log_concave}) in the setting of log-concave density estimation using the MLE. Let $[\widehat{u}(x_0),\widehat{v}(x_0)]$ be the maximal interval containing $x_0$ on which $\widehat{\varphi}_{n,s}$ is linear. Let $\widehat{u}_{\mathrm{m},s}$ (resp.~$\widehat{v}_{\mathrm{m},s}$) be the first kink of $\widehat{\varphi}_{n,s}$ to the left (resp.~right) of $\widehat{m}_{n,s}$. Consider the following CIs:  
\begin{align}\label{def:CI_s_concave_0_1}
\mathcal{I}_{n, \mathrm{sc}}^{(0)} (c_\delta^{(0)}) &\equiv \bigg[ \widehat{f}_{n,s}(x_0) \pm \frac{\sqrt{\widehat{f}_{n,s}(x_0)}\cdot c_\delta^{(0)}}{ \sqrt{ n(\widehat{v}(x_0)-\widehat{u}(x_0) )} } \bigg] \cap [0, \infty),
\\ \nonumber 
\mathcal{I}_{n, \mathrm{sc}}^{(1)} (c_\delta^{(1)}) &\equiv \bigg[ \widehat{f}_{n,s}'(x_0) \pm \frac{\sqrt{\widehat{f}_{n,s}(x_0)}\cdot c_\delta^{(1)}}{ \sqrt{ n(\widehat{v}(x_0)-\widehat{u}(x_0))^3} } \bigg], 
\\ \nonumber
\mathcal{I}_{n, \mathrm{sc}}^{\mathrm{m}} (c_\delta^{\mathrm{m}}) &\equiv \bigg[\widehat{m}_{n,s} \pm c_\delta^{\mathrm{m}} \big(\widehat{v}_{\mathrm{m},s}-\widehat{u}_{\mathrm{m},s}\big) \bigg]. 
\end{align}
The validity of the above proposed CIs is guaranteed by the following theorems; see Appendix \ref{pf:pivotal_limit_s_concave} for their proofs. 

\begin{theorem}\label{thm:pivotal_limit_s_concave}
	Let $s \in (-1,0)$. Suppose $f_0 \in \mathcal{P}_s$ with $f_0=\varphi_s^{1/s}$ for some convex function $\varphi_s$, $f_0(x_0)>0$ and $\varphi_s$ is locally $C^2$ at $x_0$ with $\varphi_s''(x_0)>0$. 
	\begin{enumerate}
		\item With $\mathbb{L}^{(i)}_2(i=0,1)$ defined in Theorem \ref{thm:pivotal_limit_fcn},
		\begin{align*}
		\begin{pmatrix}
		\sqrt{n(\widehat{v}(x_0)-\widehat{u}(x_0))}\big(\widehat{f}_{n,s}(x_0)-f_0(x_0)\big)\\
		\sqrt{n(\widehat{v}(x_0)-\widehat{u}(x_0))^3}\big(\widehat{f}_{n,s}'(x_0)-f_0'(x_0)\big)
		\end{pmatrix}
		\rightsquigarrow -\sqrt{f_0(x_0)}\cdot  
		\begin{pmatrix}
		\mathbb{L}^{(0)}_2\\
		 \mathbb{L}^{(1)}_2
		\end{pmatrix}.
		\end{align*}
		\item Let $c_\delta^{(0)},c_\delta^{(1)}$ be chosen such that $
		\Prob\big(\abs{\mathbb{L}_2^{(i)} }>c_\delta^{(i)}\big)=\delta$ for $i=0,1$, then the CIs in (\ref{def:CI_s_concave_0_1}) satisfy 
		\begin{align*}
		&\lim_{n \to \infty} \Prob_{f_0}\big(f_0(x_0) \in \mathcal{I}_{n, \mathrm{sc}}^{(0)} (c_\delta^{(0)}) \big) = \lim_{n \to \infty} \Prob_{f_0}\big(f_0'(x_0) \in \mathcal{I}_{n, \mathrm{sc}}^{(1)} (c_\delta^{(1)})\big) = 1- \delta.
		\end{align*}
		\item For any $\epsilon>0$, 
		\begin{align*}
		&\liminf_{n \to \infty} \bigg\{ \Prob_{f_0} \Big( \bigabs{\mathcal{I}_{n, \mathrm{sc}}^{(0)} (c_\delta^{(0)})  }<2c_\delta^{(0)} \mathfrak{g}_\epsilon^{(0)}\cdot n^{-2/5}  d_{2, \mathrm{sc}}^{(0)} (f_0,x_0) \Big)\\
		& \qquad \qquad \bigwedge \Prob_{f_0}\Big(\bigabs{\mathcal{I}_{n, \mathrm{sc}}^{(1)} (c_\delta^{(1)}) }<2c_\delta^{(1)} \mathfrak{g}_\epsilon^{(1)}\cdot n^{-1/5} d_{2, \mathrm{sc}}^{(1)} (f_0,x_0) \Big) \bigg\} \geq 1-\epsilon. 
		\end{align*}
		Here $\mathfrak{g}_\epsilon^{(i)} (i=0,1)$'s are constants that depend only on $\epsilon$, and with $r_s = -1/s$,
		\begin{align*}
		&  d_{2, \mathrm{sc}}^{(0)} (f_0,x_0) = \bigg(\frac{r_s f_0(x_0)^3 \abs{\varphi_s''(x_0)} }{\varphi_s(x_0) 4!}\bigg)^{1/5}, \;
 d_{2, \mathrm{sc}}^{(1)} (f_0,x_0) =\bigg(\frac{r_s^3 f_0(x_0)^4 \abs{\varphi_s''(x_0)}^3 }{(\varphi_s(x_0))^3(4!)^3}\bigg)^{1/5}.
		\end{align*}
	\end{enumerate}
\end{theorem}

\begin{theorem}\label{thm:pivotal_limit_mode_s_concave}
	Let $s \in (-1,0)$. Suppose $f_0 \in \mathcal{P}_s$ with $f_0=\varphi_s^{1/s}$ for some convex function $\varphi_s$, $f_0$ is locally $C^2$ at $m_0$ with $f_0''(m_0)<0$,  where $ m_0\equiv  [\varphi_s]_{ \mathrm{m} } $ is the mode of $f_0$.
	\begin{enumerate}
		\item With $\mathbb{M}_2$ defined in Theorem \ref{thm:pivotal_limit_mode},
		\begin{align*}
		\frac{1}{\widehat{v}_{\mathrm{m},s}-\widehat{u}_{\mathrm{m},s}}\big(\widehat{m}_{n,s}- m_0\big)\rightsquigarrow \mathbb{M}_2.
		\end{align*}
		\item Let $c_\delta^{\mathrm{m}}$ be chosen such that $
		\Prob\big(\abs{\mathbb{M}_2}>c_\delta^{\mathrm{m}}\big) = \delta$, then the CI in (\ref{def:CI_s_concave_0_1}) for the mode satisfies
		\begin{align*}
		\lim_{n \to \infty} \Prob_{ m_0}\big( m_0 \in \mathcal{I}_{n, \mathrm{sc}}^{\mathrm{m}} (c_\delta^{\mathrm{m}})\big) = 1- \delta.
		\end{align*}
		\item For any $\epsilon>0$,
		\begin{align*}
		\liminf_{n \to \infty} \Prob_{ m_0}\bigg(\bigabs{\mathcal{I}_{n, \mathrm{sc}}^{\mathrm{m}}  (c_\delta^{\mathrm{m}}) }<2c_\delta^{\mathrm{m}} \mathfrak{g}_\epsilon^{\mathrm{m}}\cdot n^{-1/5} d_{2, \mathrm{sc}}^{\mathrm{m}}(f_0)\bigg)\geq 1-\epsilon.
		\end{align*}
		Here $\mathfrak{g}_\epsilon^{\mathrm{m}}$ depends only on $\epsilon$, and $\displaystyle d_{2, \mathrm{sc}}^{\mathrm{m}}(f_0) = \bigg(\frac{(4!)^2 (\varphi_s( m_0))^2 }{r_s^2 f_0( m_0)\big(\varphi_s''( m_0)\big)^2}\bigg)^{1/5}$.
	\end{enumerate}
\end{theorem}
The above asymptotically pivotal LNE theories show that the CIs in (\ref{def:CI_s_concave_0_1}) have asymptotically exact coverage and shrink at the optimal length in the sense similar to Remark \ref{remark:optimal_rate}. 

Suppose the true density $f_0$ is log-concave ($0$-concave) and we use CIs in (\ref{def:CI_s_concave_0_1}) constructed using the divergence estimator $\widehat{f}_{n,s}$. Then these CIs still have asymptotically exact coverage. Now we consider how much price we need to pay for making inference on a true log-concave density by the divergence estimator over the larger class of $s$-concave densities. We formalize the result below which is proved in Appendix~\ref{pf:constat_log_s_concave}. Recall that $r_s=-1/s$ and the notation used in Theorems~\ref{thm:pivotal_limit_log_concave},~\ref{thm:pivotal_limit_mode_log_concave},~\ref{thm:pivotal_limit_s_concave} and~\ref{thm:pivotal_limit_mode_s_concave}.

\begin{proposition}\label{prop:constat_log_s_concave}
	Let $f_0$ be a log-concave density and $f_0=\exp(\varphi_0)$ for some concave function $\varphi_0$. Let $\varphi_s\equiv f_0^{-1/r_s} = \exp(-\varphi_0/r_s)$ be the underlying convex function when $f_0$ is viewed as an $s$-concave density. Then the following hold:
	\begin{enumerate}
		\item For $i=0,1$, $ d_{2, \mathrm{sc}}^{(i)} (f_0,x_0)>d_{2, \mathrm{lc}}^{(i)} (f_0,x_0)$ for all $s \in (-1,0)$, and the limit holds: $\lim_{s \uparrow 0}  d_{2, \mathrm{sc}}^{(i)} (f_0,x_0) = d_{2, \mathrm{lc}}^{(i)} (f_0,x_0)$.
		\item $d_{2, \mathrm{sc}}^{\mathrm{m}}(f_0)= d_{2, \mathrm{lc}}^{\mathrm{m}}(f_0)$ for all $s \in (-1,0)$.
	\end{enumerate}
\end{proposition}

From the above proposition, it is clear that a price will be paid in terms of the length of the CIs when using the divergence estimator $\widehat{f}_{n,s}$ for making inference for $f_0(x_0),f_0'(x_0)$ if the true density $f_0$ is log-concave. This price vanishes as $s \uparrow 0$. However, Proposition~\ref{prop:constat_log_s_concave}-(2) shows that, interestingly, no price will be paid when the task is to make inference about the mode of a log-concave density, even if one uses the divergence estimator that is designed for a strictly larger class of densities.

\subsection{Convex nonincreasing density estimation}

Suppose we observe $X_1,\ldots,X_n$ from a convex nonincreasing density $f_0$ on $[0,\infty)$. Let $\widehat{f}_n$ be the MLE based on $X_1,\ldots,X_n$, that is,
\begin{align*}
\widehat{f}_n&\equiv \underset{f: \textrm{ convex nonincreasing}, \int_{0}^{\infty} f =1}{\mathrm{arg\,max}} \,\int_{0}^{\infty} \log f(x)\,\d{\mathbb{F}_n(x)}\\
&= \underset{f: \textrm{ convex nonincreasing}}{\mathrm{arg\,max}} \, \bigg\{\int_{0}^{\infty} \log f(x)\,\d{\mathbb{F}_n(x)} - \int_{0}^\infty f(x)\,\d{x}\bigg\},
\end{align*}
or the LSE, that is,
\begin{align*}
\widehat{f}_n&\equiv \underset{f: \textrm{ convex nonincreasing}, \int_{0}^{\infty} f =1}{\mathrm{arg\,min}} \,\bigg\{\frac{1}{2} \int_{0}^{\infty} f^2(x)\,\d{x} - \int_{0}^{\infty} f(x) \,\d{\mathbb{F}_n(x)}\bigg\}
\cr
&= \underset{f: \textrm{convex nonincreasing}}{\mathrm{arg\,min}} \,\bigg\{\frac{1}{2} \int_{0}^{\infty} f^2(x)\,\d{x} - \int_{0}^{\infty} f(x) \,\d{\mathbb{F}_n(x)}\bigg\}.
\end{align*}
Recall that $\mathbb{F}_n$ is the empirical distribution function of the sample $X_1,\ldots,X_n$. \cite{groeneboom2001estimation} obtained the limit distribution theory (\ref{ineq:underlying_1}) for the above convex MLE and LSE $\widehat{f}_n$ with $a = \sqrt{f_0(x_0)}, b = f_0''(x_0)/4!$ under natural curvature conditions at $x_0 \in (0,\infty)$.

Now consider inference for the parameters $f_0(x_0),f_0'(x_0)$ using the CIs in (\ref{def:CI_generic_0_1}) with $\widehat{a}_n=\sqrt{\widehat{f}_n(x_0)}$. More specifically, let $[\widehat{u}(x_0),\widehat{v}(x_0)]$ be the maximal interval containing $x_0$ on which $\widehat{f}_n$ is linear, and
\begin{align}\label{def:CI_density_0_1}
\mathcal{I}_{n, \mathrm{d}}^{(0)} (c_\delta^{(0)}) &\equiv \bigg[ \widehat{f}_n(x_0) \pm \frac{\sqrt{\widehat{f}_n(x_0)}\cdot c_\delta^{(0)}}{ \sqrt{ n(\widehat{v}(x_0)-\widehat{u}(x_0) )} } \bigg] \cap [0, \infty),\\
\mathcal{I}_{n, \mathrm{d}}^{(1)}(c_\delta^{(1)}) &\equiv \bigg[\widehat{f}_n'(x_0) \pm \frac{\sqrt{\widehat{f}_n(x_0)}\cdot c_\delta^{(1)}}{ \sqrt{ n(\widehat{v}(x_0)-\widehat{u}(x_0))^3} }\bigg]. \nonumber
\end{align}
The validity of the above CIs is based on the following result, proved in Appendix \ref{pf:pivotal_limit_density}.
\begin{theorem}\label{thm:pivotal_limit_density}
	Suppose that $f_0$ is convex nonincreasing and $f_0$ is locally $C^2$ at $x_0 \in (0,\infty)$ with $f_0''(x_0)>0$. 
	\begin{enumerate}
		\item With $\mathbb{L}^{(i)}_2(i=0,1)$ defined in Theorem \ref{thm:pivotal_limit_fcn},
		\begin{align*}
		\begin{pmatrix}
		\sqrt{n(\widehat{v}(x_0)-\widehat{u}(x_0))}\big(\widehat{f}_n(x_0)-f_0(x_0)\big)\\
		\sqrt{n(\widehat{v}(x_0)-\widehat{u}(x_0))^3}\big(\widehat{f}_n'(x_0)-f_0'(x_0)\big)
		\end{pmatrix}
		\rightsquigarrow \sqrt{f_0(x_0)}\cdot  
		\begin{pmatrix}
		\mathbb{L}^{(0)}_2\\
		\mathbb{L}^{(1)}_2
		\end{pmatrix}
		.
		\end{align*}
		\item Let $c_\delta^{(0)},c_\delta^{(1)}$ be chosen such that $
		\Prob\big(\abs{\mathbb{L}_2^{(i)} }>c_\delta^{(i)}\big)=\delta$ for $i=0,1$, then the CIs in (\ref{def:CI_density_0_1}) satisfy 
		\begin{align*}
		&\lim_{n \to \infty} \Prob_{f_0}\big(f_0(x_0) \in \mathcal{I}_{n, \mathrm{d}}^{(0)} (c_\delta^{(0)}) \big) = \lim_{n \to \infty} \Prob_{f_0}\big(f_0'(x_0) \in \mathcal{I}_{n, \mathrm{d}}^{(1)} (c_\delta^{(1)})\big) = 1- \delta.
		\end{align*}
		\item For any $\epsilon>0$, 
		\begin{align*}
		&\liminf_{n \to \infty} \bigg\{\Prob_{f_0} \Big( \bigabs{\mathcal{I}_{n, \mathrm{d}}^{(0)} (c_\delta^{(0)})  }<2c_\delta^{(0)} \mathfrak{g}_\epsilon^{(0)}\cdot n^{-2/5} d_{2, \mathrm{d}}^{(0)} (f_0,x_0) \Big)\\
		& \qquad \qquad \bigwedge \Prob_{f_0} \Big( \bigabs{\mathcal{I}_{n, \mathrm{d}}^{(1)} (c_\delta^{(1)}) }<2c_\delta^{(1)} \mathfrak{g}_\epsilon^{(1)}\cdot n^{-1/5} d_{2, \mathrm{d}}^{(1)} (f_0,x_0)\Big) \bigg\}\geq 1-\epsilon. 
		\end{align*}
		Here $\mathfrak{g}_\epsilon^{(i)} (i=0,1)$'s are constants that depend only on $\epsilon$, and 
		\begin{align*}
		&d_{2, \mathrm{d}}^{(0)} (f_0,x_0) = \bigg(\frac{f_0(x_0)^2 f_0''(x_0) }{4!}\bigg)^{1/5},\quad  d_{2, \mathrm{d}}^{(1)} (f_0,x_0) =\bigg(\frac{f_0(x_0) f_0''(x_0)^3 }{(4!)^3}\bigg)^{1/5}.
		\end{align*}
	\end{enumerate}
\end{theorem}
The above asymptotically pivotal LNE theory shows that the CIs in (\ref{def:CI_density_0_1}) have asymptotically exact coverage and shrink at optimal length.

\subsection{Convex bathtub-shaped hazard function estimation}

Suppose we observe i.i.d.~samples $X_1,\ldots, X_n$ from a density $f_0$ on $[0,\infty)$ with convex hazard rate $h_0\equiv f_0/(1-F_0)$ where $F_0$ is the cumulative distribution function of $f_0$. Let $X_{(1)},\ldots, X_{(n)}$ be the order statistics of $X_1,\ldots,X_n$.  Following \cite{jankowski2009nonparametric}, let $\widehat{h}_n:[0,X_{(n)})\to \R_{\geq 0}$ be the maximizer of
\begin{align*}
h \mapsto \prod_{i=1}^{n-1} h(X_{(i)}) e^{-H(X_{(i)};h)}\cdot e^{-H(X_{(n)};h)},
\end{align*}
where $h$ ranges over all nonnegative convex functions on $[0,X_{(n)})$ and $H(t;h)\equiv \int_0^t h(s)\,\d{s}$, and then extend $\widehat{h}_n$ on the whole real line by setting $\widehat{h}_n(x)\equiv \infty$ for $x\geq X_{(n)}$. \cite{jankowski2009nonparametric} obtained the limit distribution theory (\ref{ineq:underlying_1}) for $\widehat{h}_n$ with $a = \sqrt{h_0(x_0)/(1-F_0(x_0))}, b = h_0''(x_0)/4!$ under natural curvature conditions at $x_0 \in (0,\infty)$.

Now we construct CIs for the parameters $h_0(x_0),h_0'(x_0)$ as in (\ref{def:CI_generic_0_1}) with $\widehat{a}_n=\sqrt{\widehat{h}_n(x_0)/(1-\mathbb{F}_n(x_0))}$. Let $[\widehat{u}(x_0),\widehat{v}(x_0)]$ be the maximal interval containing $x_0$ on which $\widehat{h}_n$ is linear, and
\begin{align}\label{def:CI_hazard_0_1}
\mathcal{I}_{n, \mathrm{h}}^{(0)} (c_\delta^{(0)}) &\equiv \bigg[ \widehat{h}_n(x_0) \pm \frac{ c_\delta^{(0)} \cdot \sqrt{\widehat{h}_n(x_0)} }{ \sqrt{1-\mathbb{F}_n(x_0)} \sqrt{ n(\widehat{v}(x_0)-\widehat{u}(x_0) )} } \bigg] \cap [0, \infty),
\\ \nonumber
\mathcal{I}_{n, \mathrm{h}}^{(1)} (c_\delta^{(1)}) &\equiv \bigg[ \widehat{h}_n'(x_0) \pm \frac{c_\delta^{(1)} \cdot \sqrt{\widehat{h}_n(x_0)} }{ \sqrt{1-\mathbb{F}_n(x_0)}  \sqrt{ n(\widehat{v}(x_0)-\widehat{u}(x_0))^3} } \bigg].
\end{align}
The validity of the above CIs is based on the following result, proved in Appendix \ref{pf:pivotal_limit_hazard}.
\begin{theorem}\label{thm:pivotal_limit_hazard}
	Suppose the hazard rate $h_0=f_0/(1-F_0)$ is convex, and $x_0>0$ is a point such that $h_0$ is locally $C^2$ at $x_0$ with $h_0(x_0)>0, h_0''(x_0)>0$. 
	\begin{enumerate}
		\item With $\mathbb{L}^{(i)}_2(i=0,1)$ defined in Theorem \ref{thm:pivotal_limit_fcn}, 
		\begin{align*}
		\begin{pmatrix}
		\sqrt{n(\widehat{v}(x_0)-\widehat{u}(x_0))}\big(\widehat{h}_n(x_0)-h_0(x_0)\big)\\
		\sqrt{n(\widehat{v}(x_0)-\widehat{u}(x_0))^3}\big(\widehat{h}_n'(x_0)-h_0'(x_0)\big)
		\end{pmatrix}
		\rightsquigarrow \sqrt{ \frac{h_0(x_0)}{1-F_0(x_0)} }\cdot 
		\begin{pmatrix}
		\mathbb{L}^{(0)}_2\\
		\mathbb{L}^{(1)}_2
		\end{pmatrix}.
		\end{align*}
		\item Let $c_\delta^{(0)},c_\delta^{(1)}$ be chosen such that $
		\Prob\big(\abs{\mathbb{L}_2^{(i)} }>c_\delta^{(i)}\big)=\delta$ for $i=0,1$, then the CIs in (\ref{def:CI_density_0_1}) satisfy 
		\begin{align*}
		&\lim_{n \to \infty} \Prob_{h_0}\big(h_0(x_0) \in \mathcal{I}_{n, \mathrm{h}}^{(0)} (c_\delta^{(0)}) \big) = \lim_{n \to \infty} \Prob_{h_0}\big(h_0'(x_0) \in \mathcal{I}_{n, \mathrm{h}}^{(1)} (c_\delta^{(1)})\big) = 1- \delta.
		\end{align*}
		\item For any $\epsilon>0$, 
		\begin{align*}
		&\liminf_{n \to \infty} \bigg\{\Prob_{f_0}\Big(\bigabs{\mathcal{I}_{n, \mathrm{h}}^{(0)} (c_\delta^{(0)})  }<2c_\delta^{(0)} \mathfrak{g}_\epsilon^{(0)}\cdot n^{-2/5}  d_{2, \mathrm{h}}^{(0)} (h_0,x_0)\Big)\\
		& \qquad \qquad \bigwedge \Prob_{f_0}\Big(\bigabs{\mathcal{I}_{n, \mathrm{h}}^{(1)} (c_\delta^{(1)}) }<2c_\delta^{(1)} \mathfrak{g}_\epsilon^{(1)}\cdot n^{-1/5} d_{2, \mathrm{h}}^{(1)} (h_0,x_0)\Big)\bigg\}\geq 1-\epsilon. 
		\end{align*}
		Here $\mathfrak{g}_\epsilon^{(i)} (i=0,1)$'s are constants that depend only on $\epsilon$, and
		\begin{align*}
		& d_{2, \mathrm{h}}^{(0)} (h_0,x_0) = \bigg(\frac{h_0(x_0)^2 h_0''(x_0) }{(1-F_0(x_0))^2 4!}\bigg)^{1/5},\quad  d_{2, \mathrm{h}}^{(1)} (h_0,x_0) =\bigg(\frac{h_0(x_0) h_0''(x_0)^3 }{(1-F_0(x_0))(4!)^3}\bigg)^{1/5}.
		\end{align*}
	\end{enumerate}
\end{theorem}

The above asymptotically pivotal LNE theory shows that the CIs in (\ref{def:CI_hazard_0_1}) have asymptotically exact coverage and shrink at the optimal length. 

\subsection{Concave distribution function estimation from corrupted data} 

We consider estimation of a concave distribution function as studied in \cite{jongbloed2009estimating} and use their notation. Let $X_1,\ldots,X_n$ be i.i.d.~random variables from an unknown concave distribution function $F_0$ on $[0,\infty)$, and $\epsilon_1,\ldots,\epsilon_n$ be i.i.d.~random variables, independent of the $X_i$'s, with known probability density function $k:[0,\infty)\to [0,\infty)$ that is bounded and nonincreasing. The goal is to estimate the distribution function $F_0$ based on i.i.d.~corrupted observations $Z_i=X_i+\epsilon_i$ with density $g_0\equiv k*\d{F_0} = \int k(\cdot-y)\,\d{F_0}(y)$, and distribution function $G_0$. This is essentially a deconvolution problem.

We will estimate $F_0$ by the LSE defined in \cite{jongbloed2009estimating} as follows. By \cite[Lemma 2.4]{jongbloed2009estimating}, there exists some $p(\cdot)$ that is nondecreasing, equals $0$ on $(-\infty,0)$ and $p(0+)=1/k(0+)$, such that $p*k(x) =  x\bm{1}_{[0,\infty)}(x)$. Explicit forms of $p$ can be found in \cite[Lemma 2.4 and Remark 2.5]{jongbloed2009estimating}. Let $U(x)\equiv x-(p*g_0)(x) = x-(p*k)*\d{F_0} (x) = x - \int_0^x F_0(t)\,\d{t}$. The survival function is $s_0(x)\equiv 1-F_0(x) = U'(x)$. Let $U_n(x)\equiv x-(p*\d{\G_n})(x)$ be the empirical estimate of $U(x)$, where $\G_n$ is the empirical measure of $Z_1,\ldots,Z_n$. The LSE of $s_0$ is now defined as
\begin{align*}
\widehat{s}_n\equiv\underset{s \in \mathcal{S}}{\mathrm{arg\,min}} \bigg\{\frac{1}{2}\int_0^\infty s^2(x)\,\d{x}-\int_0^\infty s(x)\,\d{U_n(x)} \bigg\},
\end{align*}
where the minimum is taken over the class $\mathcal{S}$ containing all $s \in L^2([0,\infty))$ such that $s$ is nonnegative, convex, nonincreasing and $s(0)\in (0,1]$. As shown in \cite[Theorem 2.8]{jongbloed2009estimating}, the set $\mathcal{S}$ in the above minimization can be further reduced to the set $\mathcal{S}_n$ containing all piecewise linear convex nonincreasing functions $s$ with kinks only at $\{Z_1,\ldots,Z_n\}$ and $s(0)=1,s(Z_{(n)})=0$. Computation of $\widehat{s}_n$ is based on a variant of the support reduction algorithm (see \cite{groeneboom2008support}) detailed in the Appendix of \cite{jongbloed2009estimating}.

\cite{jongbloed2009estimating} obtained the limit distribution theory (\ref{ineq:underlying_1}) for the LSE $\widehat{s}_n$ with $a= \sqrt{g_0(x_0)}/k(0), b=s_0''(x_0)/4!$ under natural curvature conditions at $x_0 \in (0,\infty)$.

Now consider inference for the parameters $s_0(x_0),s_0'(x_0)$ using the CIs in (\ref{def:CI_generic_0_1}) with $\widehat{a}_n=\sqrt{\widehat{g}_n(x_0)}/k(0)$. Here $\widehat{g}_n(x_0)\equiv \sum_{i} \bm{1}_{ \{ \widehat{u}(x_0)\leq Z_i\leq \widehat{v}(x_0) \} }/\{ n(\widehat{v}(x_0)-\widehat{u}(x_0))\}$, with  $[\widehat{u}(x_0),\widehat{v}(x_0)]$ being the maximal interval containing $x_0$ on which $\widehat{s}_n$ is linear. Let
\begin{align}\label{def:CI_deconvolution_0_1}
\mathcal{I}_{n, \mathrm{dc}}^{(0)} (c_\delta^{(0)}) &\equiv \bigg[ \widehat{s}_n(x_0) \pm \frac{\sqrt{\widehat{g}_n(x_0)}\cdot c_\delta^{(0)}}{k(0) \sqrt{ n(\widehat{v}(x_0)-\widehat{u}(x_0) )} } \bigg] \cap [0, 1],\\
\mathcal{I}_{n, \mathrm{dc}}^{(1)}(c_\delta^{(1)}) &\equiv \bigg[\widehat{s}_n'(x_0) \pm \frac{\sqrt{\widehat{g}_n(x_0)}\cdot c_\delta^{(1)}}{k(0) \sqrt{ n(\widehat{v}(x_0)-\widehat{u}(x_0))^3} }\bigg]. \nonumber
\end{align}
The above CIs have asymptotically exact coverage and optimal length, as shown below; the proof can be found in Appendix \ref{pf:pivotal_limit_deconvolution}.
\begin{theorem}\label{thm:pivotal_limit_deconvolution}
	Suppose that $s_0=1-F_0$ is convex nonincreasing and $s_0$ is locally $C^2$ at $x_0 \in (0,\infty)$ with $s_0''(x_0)>0$, and $k$ is smooth in the sense that $k(x)$ can be written as $\int_{x}^{\infty} \kappa(y)\, \d{y}$ for a Lipschitz continuous nonnegative function $\kappa$ on $(0, \infty)$.
	\begin{enumerate}
		\item With $\mathbb{L}^{(i)}_2(i=0,1)$ defined in Theorem \ref{thm:pivotal_limit_fcn},
		\begin{align*}
		\begin{pmatrix}
		\sqrt{n(\widehat{v}(x_0)-\widehat{u}(x_0))}\big(\widehat{s}_n(x_0)-s_0(x_0)\big)\\
		\sqrt{n(\widehat{v}(x_0)-\widehat{u}(x_0))^3}\big(\widehat{s}_n'(x_0)-s_0'(x_0)\big)
		\end{pmatrix}
		\rightsquigarrow \frac{\sqrt{g_0(x_0)}}{k(0)}\cdot  
		\begin{pmatrix}
		\mathbb{L}^{(0)}_2\\
		\mathbb{L}^{(1)}_2
		\end{pmatrix}
		.
		\end{align*}
		\item Let $c_\delta^{(0)},c_\delta^{(1)}$ be chosen such that $
		\Prob\big(\abs{\mathbb{L}_2^{(i)} }>c_\delta^{(i)}\big)=\delta$ for $i=0,1$, then the CIs in (\ref{def:CI_deconvolution_0_1}) satisfy 
		\begin{align*}
		&\lim_{n \to \infty} \Prob_{s_0}\big(s_0(x_0) \in \mathcal{I}_{n, \mathrm{dc}}^{(0)} (c_\delta^{(0)}) \big) = \lim_{n \to \infty} \Prob_{s_0}\big(s_0'(x_0) \in \mathcal{I}_{n, \mathrm{dc}}^{(1)} (c_\delta^{(1)})\big) = 1- \delta.
		\end{align*}
		\item For any $\epsilon>0$, 
		\begin{align*}
		&\liminf_{n \to \infty} \bigg\{ \Prob_{s_0} \Big(\bigabs{\mathcal{I}_{n, \mathrm{dc}}^{(0)} (c_\delta^{(0)})  }<2c_\delta^{(0)} \mathfrak{g}_\epsilon^{(0)}\cdot n^{-2/5} d_{2, \mathrm{dc}}^{(0)} (s_0,x_0) \Big)\\
		& \qquad \qquad \bigwedge \Prob_{s_0}\Big( \bigabs{\mathcal{I}_{n, \mathrm{dc}}^{(1)} (c_\delta^{(1)}) }<2c_\delta^{(1)} \mathfrak{g}_\epsilon^{(1)}\cdot n^{-1/5} d_{2, \mathrm{dc}}^{(1)} (s_0,x_0)\Big) \bigg\}\geq 1-\epsilon. 
		\end{align*}
		Here $\mathfrak{g}_\epsilon^{(i)} (i=0,1)$'s are constants that depend only on $\epsilon$, and 
		\begin{align*}
		&d_{2, \mathrm{dc}}^{(0)} (s_0,x_0) = \bigg(\frac{g_0(x_0)^2 s_0''(x_0) }{4! k(0)^4}\bigg)^{1/5},\quad  d_{2, \mathrm{dc}}^{(1)} (s_0,x_0) =\bigg(\frac{g_0(x_0) s_0''(x_0)^3 }{(4!)^3 k(0)^2}\bigg)^{1/5}.
		\end{align*}
	\end{enumerate}
\end{theorem}
It is worth noting that, as in all the other models studied in this paper, the conditions we assume in the above theorem are the same as those  used to derive the limit distribution theory (\ref{ineq:underlying_1}) of $(\widehat{s}_n(x_0), \widehat{s}'_n(x_0))$ in \cite{jongbloed2009estimating}, that is, we do not impose any extra conditions on the underlying model.

\section{A uniform tail estimate for the limit distributions}\label{section:tail_estimate}

We first present a result on an exponential tail estimate of the limit processes in Theorem \ref{thm:limiting_dist_convex} that holds uniformly in $\alpha$; see Appendix \ref{pf:tail_uniform} for its proof.

\begin{theorem}\label{thm:tail_uniform}
	There exist universal constants $L>0,b>0$ such that
	\begin{align}\label{ineq:tail_uniform}
	\sup_\alpha \Big\{\Prob\big(\abs{\mathbb{H}_{\alpha}^{(2)}(0)}>t\big)\vee \Prob\big(\abs{\mathbb{H}_{\alpha}^{(3)}(0)}>t\big) \vee \Prob\big(h^\ast_{\alpha;\pm} >t\big) \Big\}\leq L \exp(-t^b/L).
	\end{align}
	Here $h^\ast_{\alpha;-}$ (resp.~$h^\ast_{\alpha;+}$) is the absolute value of the location of the first touch point of the pair $(\mathbb{H}_{\alpha}, \mathbb{Y}_{\alpha})$ to the left (resp.~right) of $0$.
\end{theorem}

The above theorem resolves a question posed in \cite{groeneboom2001canonical} concerning the existence of moments of $\mathbb{H}_{2}^{(2)}(0)$ (see pp.~1648 therein). In fact, the theorem above shows that all moments of $\mathbb{H}_{\alpha}^{(2)}(0)$ and $\mathbb{H}_{\alpha}^{(3)}(0)$ can be controlled uniformly in $\alpha$.

\begin{remark}
	Although it is in principle possible to track down the constant value $b$ in (\ref{ineq:tail_uniform}) in the proof for the above theorem, this numerical value can be far from optimal. In the related problem of isotonic regression with LSE $\widehat{f}_n^{(\textrm{iso})}$, the limiting distribution $\mathbb{D}_{\alpha}$ in (\ref{ineq:iso_limit}) can be analytically characterized when $\alpha=1$. Let $\mathbb{Z}_1\equiv \mathbb{D}_1/2$ be the Chernoff distribution. Then by \cite{groeneboom1989brownian} (see also \cite{dumbgen2016law}), the density function $p_{\mathbb{Z}_1}$ of $\mathbb{Z}_1$ satisfies
	\begin{align*}
	p_{\mathbb{Z}_1}(t)\sim \frac{1}{2 \mathrm{Ai}'(a_1)} 4^{4/3} t \exp\bigg(-\frac{2}{3}t^3+3^{1/3}a_1 t\bigg),\quad t\uparrow \infty,
	\end{align*}
	and therefore the tail probability for $\mathbb{Z}_1$ satisfies
	\begin{align*}
	\Prob\big(\mathbb{Z}_1>t \big)\sim \frac{1}{2 \mathrm{Ai}'(a_1)} 4^{4/3}\frac{1}{t}\exp\bigg(-\frac{2}{3}t^3\bigg),\quad t\uparrow \infty.
	\end{align*}
	Here $a_1\approx -2.3381$ is the largest zero of the Airy function $\mathrm{Ai}$ and $\mathrm{Ai}'(a_1)\approx 0.7022$. The exponent $3$ here can also be seen by a law of iterated logarithm (LIL) established for the Grenander estimator in the decreasing density model in \cite{dumbgen2016law}, where techniques from local empirical processes (see e.g.,~\cite{deheuvels1994functional,einmahl1997gaussian}) rather than the above formulas are exploited. These techniques (from LIL) naturally hint that in the setting of estimation of convex functions, the limiting random variables may have tail bounds like
	\begin{align*}
	\Prob\big(\abs{\mathbb{H}_{\alpha}^{(2)}(0)}>t\big)\leq K_\alpha\exp(-t^{2+1/\alpha}/K_\alpha),
	\end{align*}
	for some constant $K_\alpha>0$ that may depend on $\alpha$. It is an interesting open question to prove (or disprove) the above conjectured optimal tail behavior. 
\end{remark}

\begin{remark}
	Results of similar spirit as in Theorem \ref{thm:tail_uniform} in the monotone setting are proved in \cite{han2019limit,deng2020confidence} through the representation of the limiting Chernoff-type distributions by explicit min-max formulas (see e.g.,~\cite{groeneboom2014nonparametric,han2019berry}). The proof of Theorem \ref{thm:tail_uniform} in the case for estimation of convex functions is significantly more challenging due to the lack of a closed-form
	expression for the process $\mathbb{H}_{\alpha}$. In fact, instead of directly working with the limiting process, we will derive the tail estimate through the weak limit of finite-sample tail behavior of the LSE in the convex density model with a class of carefully constructed true convex densities, so that the estimates can be obtained uniformly in $\alpha$. 
\end{remark}

As a direct consequence of Theorem \ref{thm:tail_uniform}, we have the following exponential tail for the limit distributions in Theorem \ref{thm:pivotal_limit_fcn}; see Appendix \ref{pf:unif_tail_pivot} for its proof.

\begin{corollary}\label{cor:unif_tail_pivot}
	There exist universal constants $L>0,b>0$ such that 
	\begin{align*}
	\sup_\alpha \Big\{ \Prob\big(\abs{\mathbb{L}_{\alpha}^{(0)}}>t\big)\vee \Prob\big(\abs{\mathbb{L}_{\alpha}^{(1)}}>t\big) \Big\}\leq L\exp(-t^b/L).
	\end{align*}
\end{corollary}

The above corollary verifies the existence of $c_\delta^{(i)} (i=0,1)$ in (\ref{cond:unif_tail}) and hence the existence of adaptive CIs in Theorem \ref{thm:adaptive_CI}.

\section{Simulation studies}\label{section:simulation}

\subsection{Simulated critical values}

We directly use Theorems \ref{thm:pivotal_limit_fcn} and \ref{thm:pivotal_limit_mode} with the true mean function $f_0(x) = 12 (x - 0.5)^2$ and $x_0 = 0.5$ to approximate the distributions of the pivotal random variables $\{ \mathbb{L}_2^{(0)}, \mathbb{L}_2^{(1)}, \mathbb{M}_2\}$. After that, we confirm the universality of these distributions by comparing them to their counterparts from two different $f_0$'s. The convex LSEs are computed using the support reduction algorithm \cite{groeneboom2008support} implemented in the \verb|R| function \verb|conreg| from package \verb|cobs|. 

We formally describe the simulation procedure as follows. Let $n = 10^5$ and design point $X_i = i/n$ for all $0 \le i \le n$. We generate data $\{(X_i, Y_i), 0 \le i \le n\}$ where $Y_i = f_0(X_i) + \xi_i = 12(X_i - 0.5)^2 + \xi_i$ and $\xi_i \overset{\mathrm{i.i.d.}}{\sim} \mathcal{N}(0, 1)$. We compute the convex LSE $\widehat{f}_n$ and then calculate the LNEs at $x_0 = 0.5$ and $m_0=[f_0]_{ \mathrm{m} }= 0.5$:
\begin{align*}
& T^{(0)}(x_0) \equiv \sqrt{n(\widehat{v}(x_0) - \widehat{u}(x_0))}\bigabs{\widehat{f}_n(x_0)-f_0(x_0)},
\cr
& T^{(1)}(x_0) \equiv \sqrt{n(\widehat{v}(x_0) - \widehat{u}(x_0))^3}\bigabs{\widehat{f}'_n(x_0)-f'_0(x_0)}, \hbox{ and }
\cr
& T^{\mathrm{m}} \equiv (\widehat{m}_n- m_0)/(\widehat{v}_{\mathrm{m}}- \widehat{u}_{\mathrm{m}})
\end{align*}
to obtain a sample of $\{ T^{(0)}(x_0), T^{(1)}(x_0), T^{\mathrm{m} } \}$. Repeating this procedure $B = 10^6$ times, we can generate one million samples of $\{ T^{(0)}(x_0), T^{(1)}(x_0), T^{\mathrm{m}} \}$. Their empirical distribution functions are then used to approximate the cumulative distribution functions of $\mathbb{L}_2^{(0)}$, $\mathbb{L}_2^{(1)}$ and $\mathbb{M}_2$, which are given in Figures \ref{fig:ecdf_L0}, \ref{fig:ecdf_L1}, and \ref{fig:ecdf_M2} respectively in Section \ref{section:pivot_limit}. 

We report in Table \ref{tab:cv} some important quantiles of these empirical distributions as the approximate corresponding critical values $c_{\delta}(\mathbb{T})$, defined by $\mathbb{P} \{ \mathbb{T} > c_{\delta}(\mathbb{T}) \} = \delta$ for $\mathbb{T} \in \{\mathbb{L}_2^{(0)}, \mathbb{L}_2^{(1)},\mathbb{M}_2\}$. 

\setlength{\belowcaptionskip}{-10pt}
\setlength{\tabcolsep}{3pt}
\renewcommand{\arraystretch}{1.5}
\begin{table}[htb]
\scriptsize{
	\begin{tabular}{|c||c c c c c c c c c |}
		\hline 
		$\delta$  & 0.990  & 0.975 & 0.950 & 0.900 & 0.500 & 0.100 & 0.050 & 0.025 & 0.010 \\
		\hline
		$c_{\delta}\big(\mathbb{L}_2^{(0)} \big)$  & -2.59 & -2.03 & -1.61 & -1.19 & 0.04 & 1.39 & 1.82 & 2.20 & 2.66   \\
		\hline
		$c_{\delta}\big(\mathbb{L}_2^{(1)} \big) $  & -11.87 & -9.00 & -6.78 & -4.55 & 0.00 & 4.54 & 6.77 & 9.00 & 11.91 
		\\
		\hline
		$c_{\delta}\big(\mathbb{M}_2 \big) $  & -0.86 & -0.61 & -0.48 & -0.35 & 0.00 & 0.35 & 0.47 & 0.61 & 0.86 \\
		\hline
	\end{tabular}
}
	\vspace{1ex}
	\caption{Approximate quantiles of $\mathbb{L}^{(0)}_2$, $\mathbb{L}^{(1)}_2$ and $\mathbb{M}_2$.}
	\label{tab:cv}
\end{table}

Recall that $\mathbb{L}_2^{(1)}$ and $\mathbb{M}_2$ are symmetric and notice that, by Figure \ref{fig:ecdf_L0} and Table \ref{tab:cv}, $\mathbb{L}_2^{(0)}$ is at least nearly symmetric. We give in Table \ref{tab:cv_abs} some absolute sample quantiles which approximate the corresponding critical values $c_{\delta}(\abs{\mathbb{T}})$ for $\mathbb{T} \in \{\mathbb{L}_2^{(0)}, \mathbb{L}_2^{(1)},\mathbb{M}_2\}$. They are used to construct the symmetric CIs (e.g., in \eqref{def:CI_fcn_0_1} and \eqref{def:CI_mode}).

\begin{table}[htb]
\scriptsize{\begin{tabular}{|c||c c c c c c |}
		\hline 
		$\delta$ & 0.50 & 0.20 & 0.10 & 0.05 & 0.02 & 0.01\\
		\hline
		$c_{\delta}\big(|\mathbb{L}_2^{(0)}| \big) $ & 0.65 & 1.30 & 1.73 & 2.13 & 2.63 & 2.99  \\
		\hline
		$c_{\delta}\big( |\mathbb{L}_2^{(1)}| \big) $ & 1.73 & 4.55 & 6.78 & 9.00 & 11.89 & 14.02
		\\
		\hline
		$c_{\delta} \big( |\mathbb{M}_2| \big) $ & 0.19 & 0.35 & 0.47 & 0.61 & 0.86 & 1.13 \\
		\hline
	\end{tabular}}
	\vspace{1ex}
	\caption{Approximate quantiles of $|\mathbb{L}^{(0)}_2|$, $|\mathbb{L}^{(1)}_2|$ and $|\mathbb{M}_2|$.}
	\label{tab:cv_abs}
\end{table}

In the second part of this subsection, we repeat the above procedure with different $f_0$'s and check if the resulting approximate distributions are almost the same. This helps to support the conclusions of Theorems \ref{thm:pivotal_limit_fcn} and \ref{thm:pivotal_limit_mode}. Consider $f_0(x) = 6(x - 0.2)^2$ and $f_0(x) = x + 2/(x+1)$ at $x_0 = 0.5$. The second derivatives of these two functions and $f_0(x) = 12(x - 0.5)^2$ are all different. We follow exactly the same procedure as before but only obtain $B = 10^4$ samples of the corresponding LNEs. Their empirical distribution functions are compared to those of the approximate $\mathbb{L}_2^{(0)}$, $\mathbb{L}_2^{(1)}$ and $\mathbb{M}_2$ in Figure \ref{fig:ecdf_comp}. 

In conclusion, we clearly observe that the empirical distributions from different $f_0$ are in general very close to one other. This indicates that the approximate critical values in Tables \ref{tab:cv} and \ref{tab:cv_abs} should be accurate enough for constructing the proposed CIs. 

\begin{figure}[!hbt]
	\centering
	\subfigure[$\mathbb{L}_2^{(0)}$]{
		\label{fig:log_con_comp:a} 
		\includegraphics[width=0.48\textwidth]{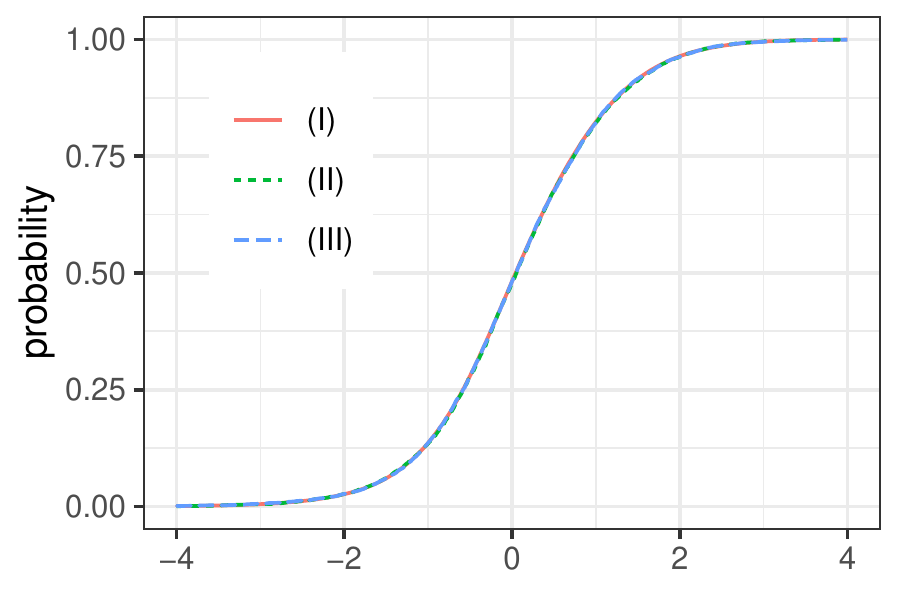}}
	\hspace{0\textwidth}
	\subfigure[$\mathbb{L}_2^{(1)}$]{
		\label{fig:log_con_comp:b} 
		\includegraphics[width=0.48\textwidth]{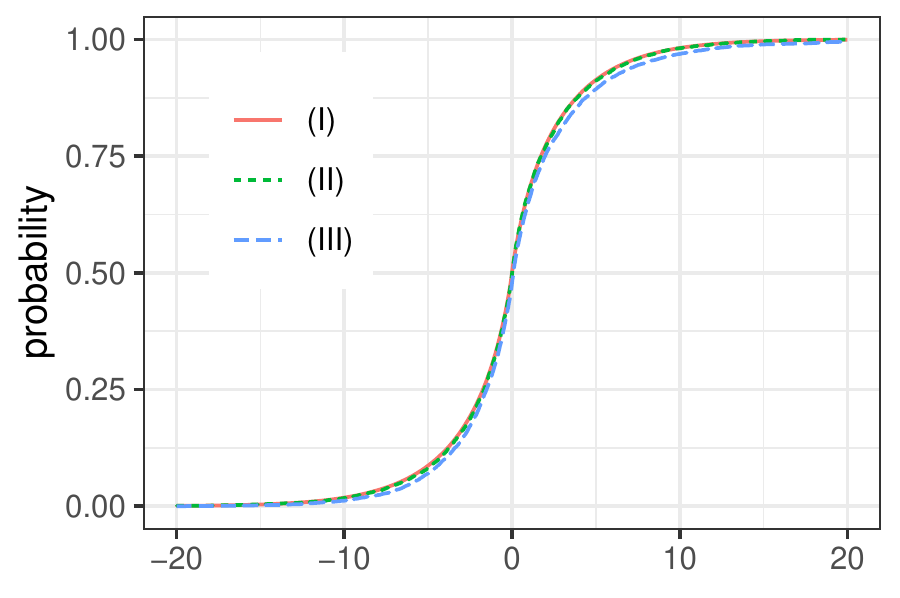}}
	\hspace{0\textwidth}
	\subfigure[$\mathbb{M}^{(0)}$]{
		\label{fig:log_con_comp:c} 
		\includegraphics[width=0.48\textwidth]{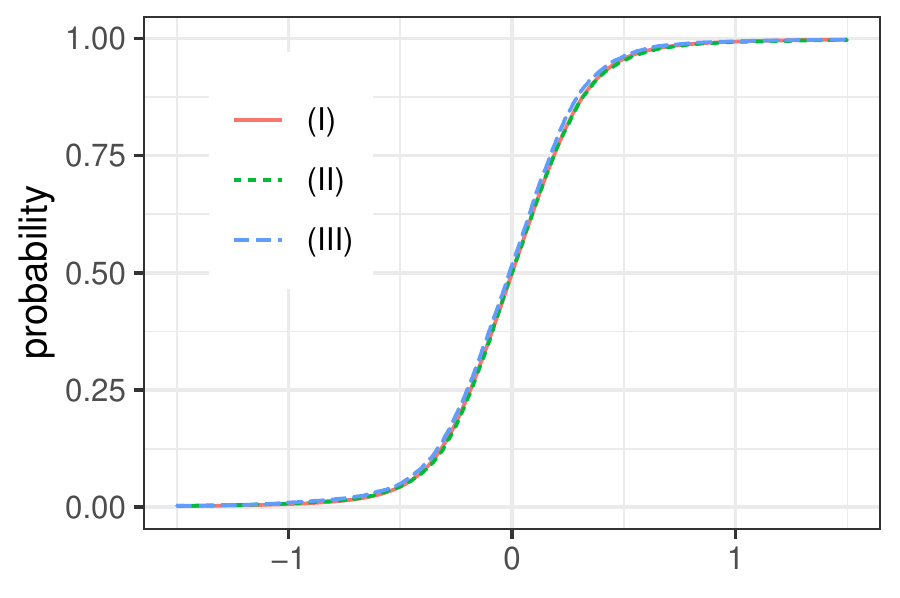}}
	\caption{Empirical distributions of the LNEs generated from: (I) $f_0(x) = 12(x-0.5)^2$, (II) $f_0(x) = 6(x - 0.2)^2$, and (III) $f_0(x) = x + 2/(x + 1)$ at $x_0 = 0.5$.}
	\label{fig:ecdf_comp}
\end{figure}

\begin{remark}
We may approximate the distributions of $\{\mathbb{L}_2^{(0)}, \mathbb{L}_2^{(1)}, \mathbb{M}_2\}$ by first generating samples of $\mathbb{H}_2^{(2)}$ and then computing these random variables from their definition. Let $\widehat{f}_{c}(x): [-c, c] \to \mathbb{R}$ be the solution to
\begin{align*}
\min \ & \phi_c(f) \equiv \frac{1}{2}\int_{-c}^c f^2(t) \d{t} - \int_{-c}^{c} f(t) \d{\big(\mathbb{B}(t) + 4t^3 \big)}
\\ \nonumber
\hbox{s.t.}& \hbox{ $f$ is convex and } f(\pm c) = 12 c^2.
\end{align*}
\cite{groeneboom2001canonical} proved that $\widehat{f}_{c}$ is unique and its linearly extended version converges almost surely to $\mathbb{H}_2^{(2)}$ in the topology of uniform convergence on compacta. They proposed the iterative cubic spline algorithm to compute $\widehat{f}_c$. However, a simulation study on $\mathbb{H}_2^{(2)}(0)$ by \cite{azadbakhsh2014computing} suggested that this algorithm does not perform very well; see Remark A.1 of \cite{azadbakhsh2014computing}. 

To effectively generate samples of $\mathbb{H}_2^{(2)}$ using \verb|R| package \verb|cobs|, \cite{azadbakhsh2014computing} removed the side constraints $f(\pm c) = 12c^2$ of the minimization problem and approximated integrals in $\phi_c(f)$ on a grid $\{X_i, 0 \le i \le N\}$ of $[-c, c]$. Let $n = \lceil N/(2c) \rceil $ be the number of points on each unit interval. Their approach is almost equivalent to convex regression with data $\{(X_i, Y_i), 0 \le i \le N\}$ where $Y_i = 12 X_i^2 + \sqrt{n}\xi_i$ and $\xi_i \overset{\mathrm{i.i.d.}}{\sim} \mathcal{N}(0,1)$ (note that $\mathbb{B}(t) + 4t^3$ can be approximated by the partial sum process of $\{Y_i\}$), which is then similar to the procedure we employ here. We actually implemented this procedure and the resulting approximate distributions show the difference is remarkably small (numerical results omitted here).
\end{remark}

\subsection{Numerical performance of the proposed confidence intervals}

We are now ready to illustrate the proposed procedures of constructing CIs and study their numerical performance. In this subsection, we focus on the convex regression model and the log-concave density estimation model. The following results mainly serve as numerical support of Theorems \ref{thm:CI_fcn} and \ref{thm:CI_mode} for convex regression and Theorems \ref{thm:pivotal_limit_log_concave} and \ref{thm:pivotal_limit_mode_log_concave} for log-concave density estimation, showing that: (i) the corresponding proposed CIs have asymptotically accurate coverage, and (ii) their lengths adapt to oracle rates (cf. Remark \ref{remark:optimal_rate}). To this end, their performance will be evaluated with different sample sizes. 

Finally, in order to compute the lengths of the oracle CIs, we shall simulate the quantiles of $\mathbb{H}_2^{(2)}(0)$, $\mathbb{H}_2^{(3)}(0)$ and $[\mathbb{H}_2^{(2)}]_{\mathrm{m}}$. They can be conveniently obtained as byproducts when we simulate the critical values of $\mathbb{L}_2^{(0)}$, $\mathbb{L}_2^{(1)}$ and $\mathbb{M}_2$; see Table \ref{tab:cv_abs_H} for the approximate quantiles.

\begin{table}[htb]
\scriptsize{\begin{tabular}{|c||c c c c c c |}
		\hline 
		$\delta$ & 0.50 & 0.20 & 0.10 & 0.05 & 0.02 & 0.01\\
		\hline
		$c_{\delta}\big(|\mathbb{H}_2^{(2)}| \big) $ & 0.89 & 1.68 & 2.16 & 2.58 & 3.08 & 3.44  \\
		\hline
		$c_{\delta}\big( |\mathbb{H}_2^{(3)}| \big) $ & 4.28 & 7.79 & 9.66 & 11.14 & 12.72 & 13.70
		\\
		\hline
		$c_{\delta} \big( | [\mathbb{H}_2^{(2)}]_{\mathrm{m}}| \big) $ & 0.18 & 0.32 & 0.40 & 0.46 & 0.53 & 0.57 \\
		\hline
	\end{tabular}}
	\vspace{1ex}
	\caption{Approximate critical values of $|\mathbb{H}_2^{(2)}|$, $|\mathbb{H}_2^{(3)}|$ and $\big| [\mathbb{H}_2^{(2)}]_{\mathrm{m}} \big|$.}
	\label{tab:cv_abs_H}
\end{table}

\subsubsection{Convex Regression}
Suppose we observe in convex regression data $\{(X_i, Y_i), 1\le i \le n\}$ of size $n$. The goal is to construct $95\%$ CIs for the function value $f_0(x_0)$, derivative value $f_0'(x_0)$ and anti-mode $m_0$. Let $X_i = i/n$ for all $0 \le i \le n$ and $\xi_i \overset{\mathrm{i.i.d.}}{\sim} \mathcal{N}(0, \sigma^2 = 1)$. The variance of noise $\sigma^2 = 1$ is assumed to be known; otherwise it can be very well approximated by, say, the difference-based estimators \cite{rice1984bandwidth,munk2005difference}. We consider $f_0(x) = 20 - 20\sqrt{1-(x-0.5)^2}$ and $x_0 = 0.5$. The anti-mode of this convex function is $m_0 = x_0 = 0.5$. 

For each data set $\{(X_i, Y_i), 1 \le i \le n\}$, we apply support reduction algorithm implemented in the \verb|R| function \verb|conreg| from package \verb|cobs| to compute the convex LSE and construct the $95\%$ CIs defined in \eqref{def:CI_fcn_0_1} and \eqref{def:CI_mode} with approximate critical values in Table \ref{tab:cv_abs}. Here $\delta = 0.05$, so that $c_{\delta}^{(i)}$ in (\ref{def:CI_fcn_0_1}) is taken to be $c_{.05}\big( |\mathbb{L}_2^{(i)}| \big)$ in Table \ref{tab:cv_abs}, for $i=0,1$, and $c_{\delta}^{\mathrm{m}}$ in (\ref{def:CI_mode}) equals $c_{.05}\big( |\mathbb{M}_2| \big)$ in Table \ref{tab:cv_abs}. With the proposed CIs constructed, we check if they cover the true values of local parameters and report their lengths. We approximate the coverage probabilities by repeating the above procedures $10^4$ times and calculating the relative frequencies of successful coverage. The plot of the estimated coverage probabilities are given in Figure \ref{fig:regression_comp:a}. Box plots of the lengths of these $10^4$ CIs for each of $\{f_0(x_0), f_0'(x_0), m_0\}$ are reported in Figures \ref{fig:regression_comp:b} -- \ref{fig:regression_comp:d}, along with the oracle CI lengths in red dashed lines. Note that by the limiting distribution theories for these local parameters in Theorem \ref{thm:limiting_dist_convex} and \eqref{limit_mode_1} in Theorem \ref{thm:pivotal_limit_mode}, the symmetric oracle CIs are 
\begin{align*}
&\Big[\widehat{f}_n(x_0) \pm (f_0^{(2)}(x_0)/24)^{1/5}(n/\sigma^2)^{-2/5} c_{\delta}( |\mathbb{H}_2^{(2)}(0)| ) \Big] \hbox{ for } f_0(x_0),\\
& \Big[\widehat{f}'_n(x_0) \pm (f_0^{(2)}(x_0)/24)^{3/5}(n/\sigma^2)^{-1/5} c_{\delta}( |\mathbb{H}_2^{(3)}(0)| ) \Big] \hbox{ for } f'_0(x_0), \hbox{ and}\\
& \Big[\widehat{m}_n \pm (24/f_0^{(2)}(m_0))^{2/5}(n/\sigma^2)^{-1/5} c_{\delta}( | [\mathbb{H}_2^{(2)}]_{\mathrm{m}}| ) \Big] \hbox{ for } m_0.
\end{align*}

\begin{figure}[!hbt]
	\centering
	\subfigure[(i) $f_0(x_0)$, (ii) $f_0'(x_0)$, and (iii) $m_0$]{
		\label{fig:regression_comp:a} 
		\includegraphics[width=0.48\textwidth]{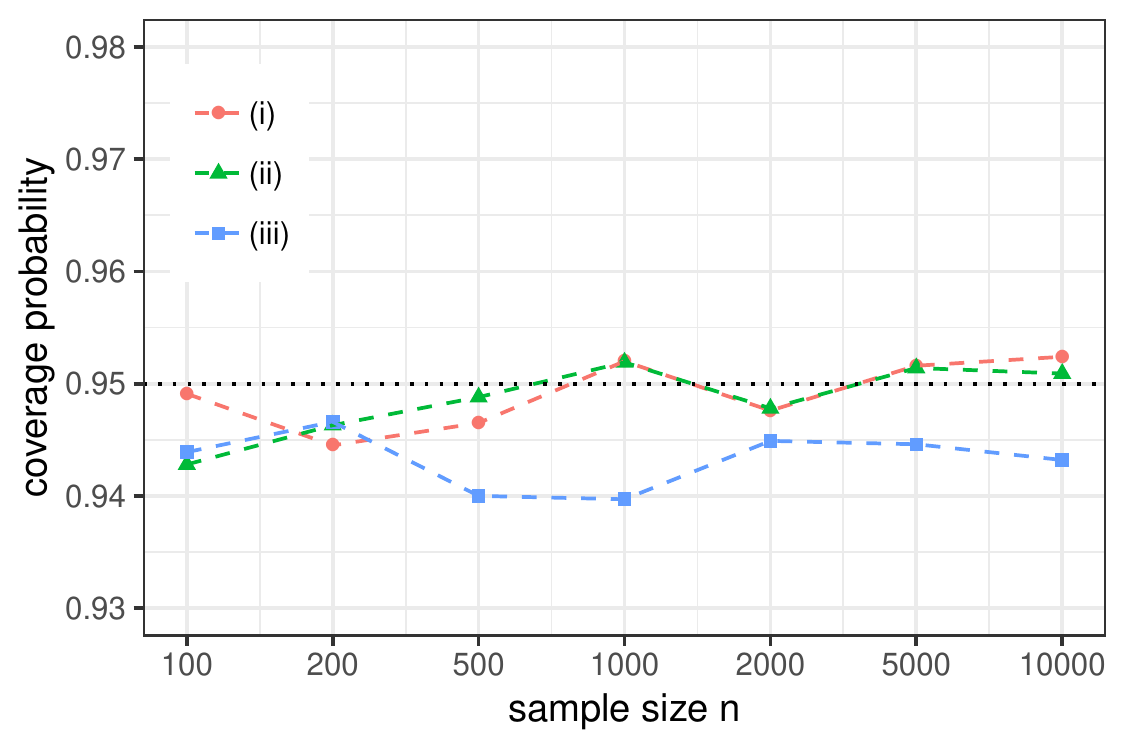}}
	\subfigure[$f_0(x_0)$]{
		\label{fig:regression_comp:b} 
		\includegraphics[width=0.48\textwidth]{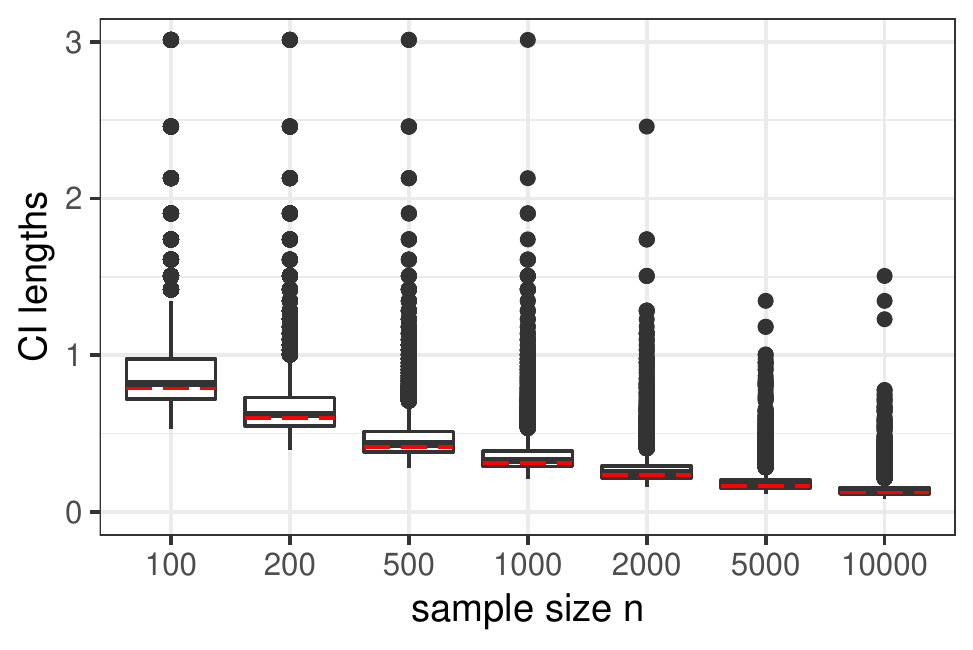}}
	\subfigure[$f_0'(x_0)$]{
		\label{fig:regression_comp:c} 
		\includegraphics[width=0.48\textwidth]{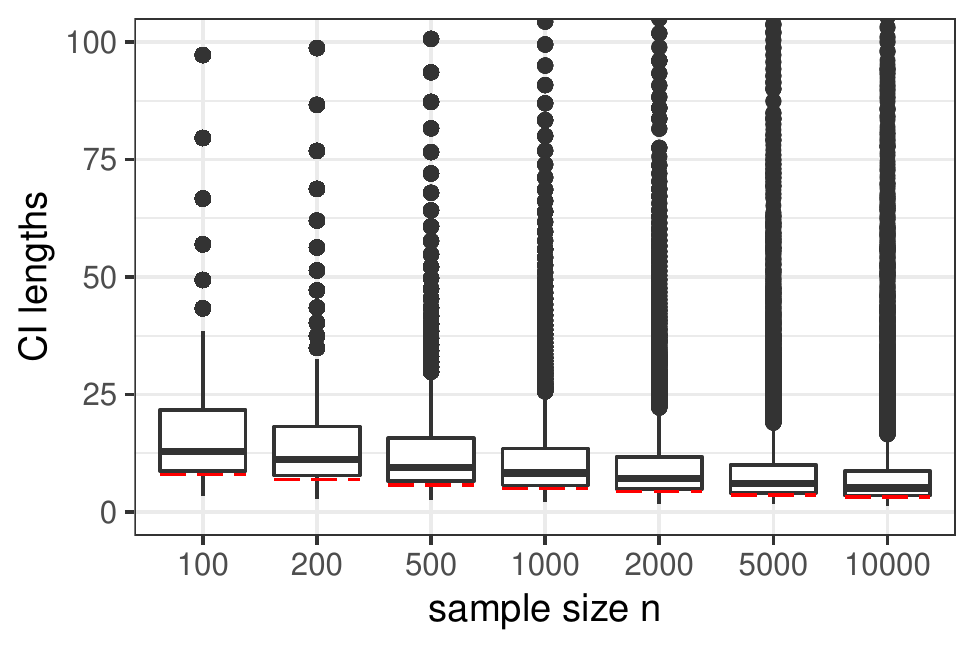}}
	\subfigure[$m_0$]{
		\label{fig:regression_comp:d} 
		\includegraphics[width=0.48\textwidth]{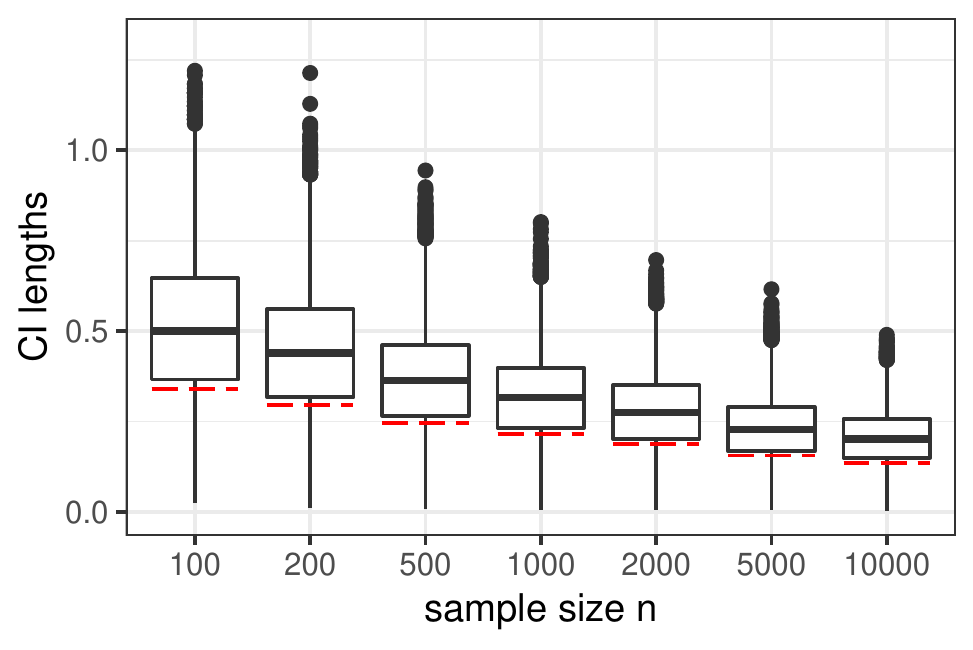}}

	\caption{Plot of the simulated coverage probabilities and box plots of the lengths of the proposed CIs for corresponding local parameters in convex regression. Here $f_0(x) = 20 - 20\sqrt{ 1-(x-0.5)^2}$, $x_0 = 0.5$ and anti-mode $m_0 = 0.5$. The red dashed lines in box plots (b)--(d) represent the lengths of the oracle CIs.}
	\label{fig:regression_comp}
\end{figure}

\begin{figure}[!hbt]
	\centering
	\subfigure[(i) $f_0(x_0)$, (ii) $f_0'(x_0)$, and (iii) $m_0$]{
		\label{fig:log_concave_comp:a} 
		\includegraphics[width=0.48\textwidth]{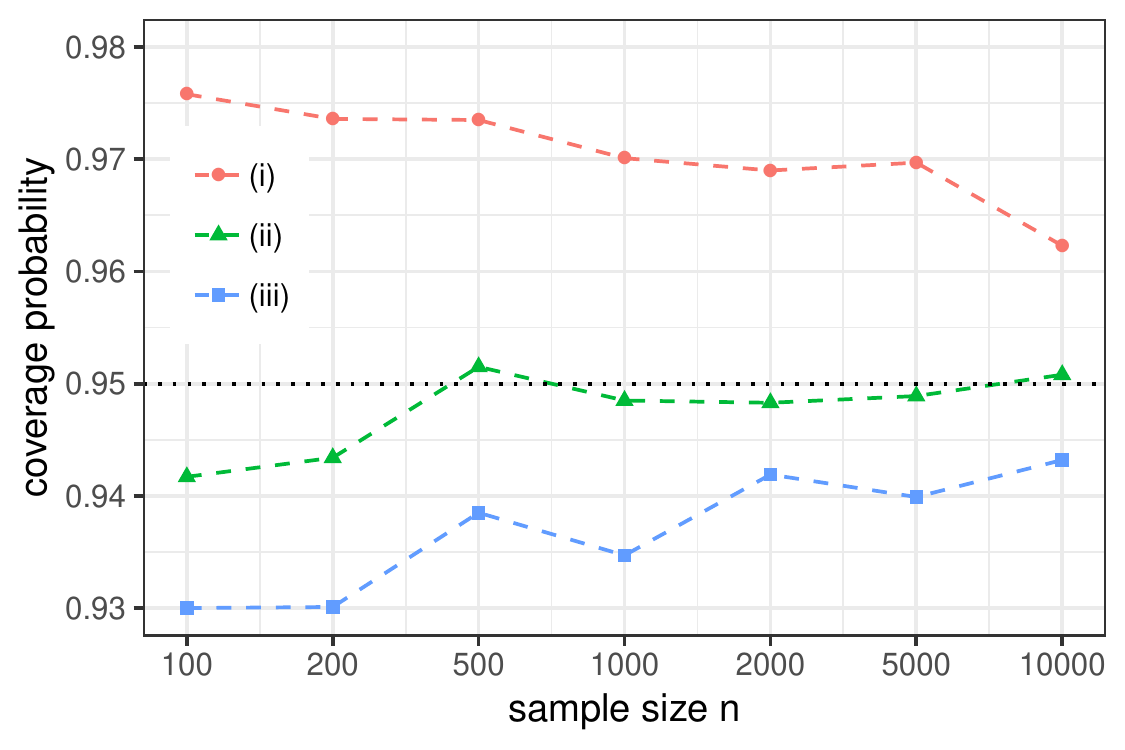}}
	\subfigure[density value $f_0(x_0)$]{
		\label{fig:log_concave_comp:b} 
		\includegraphics[width=0.48\textwidth]{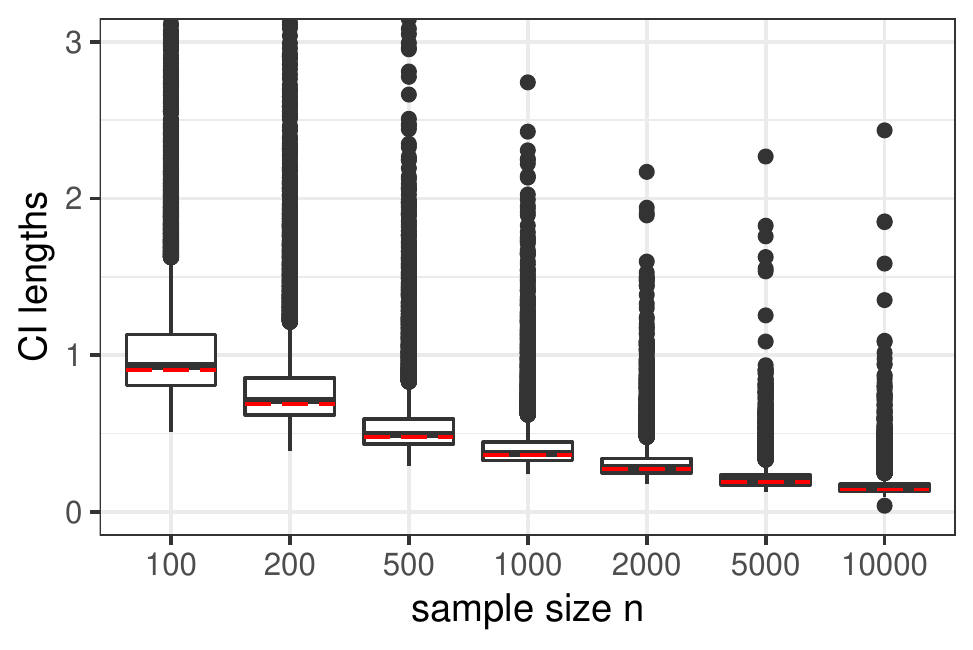}}
	\subfigure[density derivative $f_0'(x_0)$]{
		\label{fig:log_concave_comp:c} 
		\includegraphics[width=0.48\textwidth]{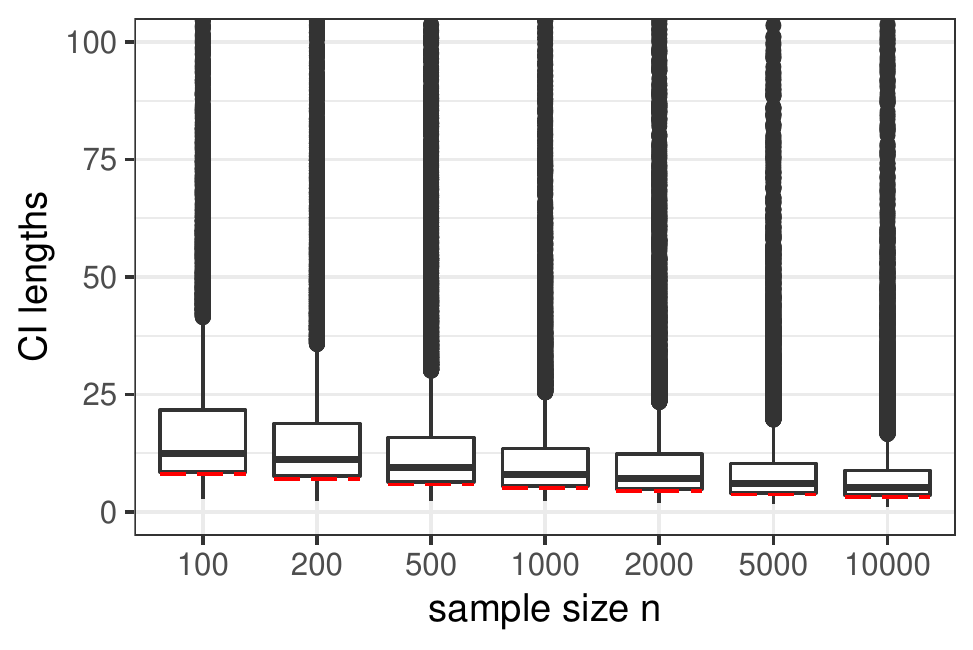}}
	\subfigure[mode $m_0$]{
		\label{fig:log_concave_comp:d} 
		\includegraphics[width=0.48\textwidth]{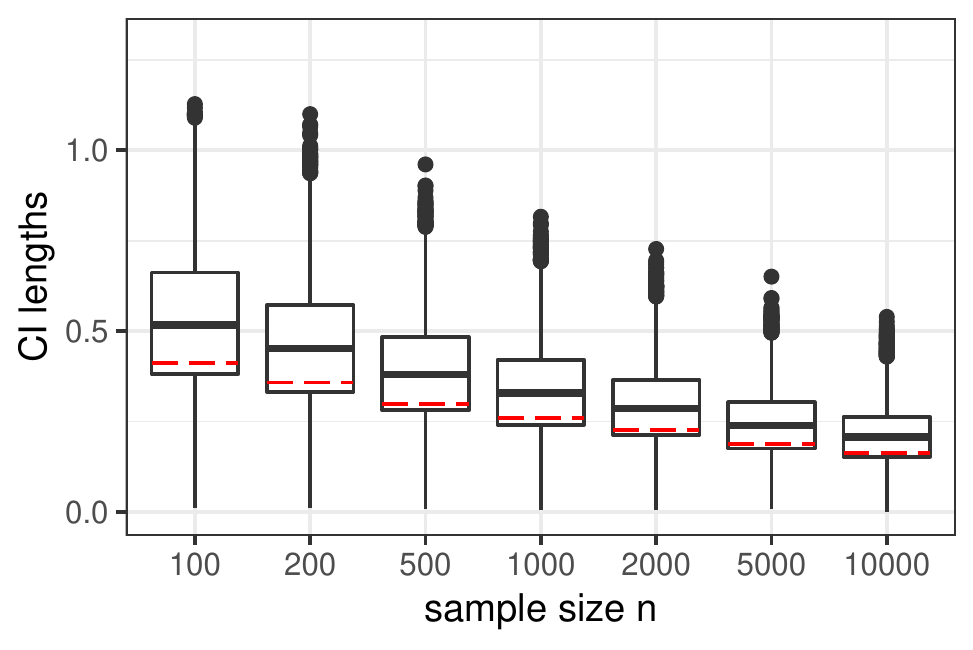}}
	\caption{Plot of the simulated coverage probabilities and box plots of the lengths of the proposed CIs for corresponding parameters in log-concave density estimation. Here $f_0$ is the density of $\mathrm{Beta}(2,3)$ distribution, $x_0 = 0.5$, and mode $m_0 = 1/3$. The red dashed lines in box plots (b)--(d) represent the lengths of the oracle CIs.}
	\label{fig:log_concave_comp}
\end{figure}

As we can see from Figure \ref{fig:regression_comp:a}, all CIs for the local parameters have rather accurate coverage and the convergence of coverage probabilities is approximately achieved for sample size as small as $n=100$. For $n$ greater than $200$, all coverage errors deviate from the nominal coverage by less than $0.005$. In terms of length, it is obvious from Figures \ref{fig:regression_comp:b}--\ref{fig:regression_comp:d} that the lengths of the proposed CIs shrink at the same rate with those of the oracle CIs. Note that when the local pieces $[\widehat{u}(x_0), \widehat{v}(x_0)]$ used to construct CIs for $f_0(x_0)$ and $f_0'(x_0)$ are small the proposed CIs may become quite wide; so we observe relatively more outliers on the CIs for $f_0(x_0)$ and $f_0'(x_0)$ than for $m_0$.

\subsubsection{Log-concave density estimation}
Suppose we observe i.i.d.~data $\{X_i, 1\le i \le n\}$ from a log-concave density. The goal is to construct $95\%$ CIs for the density value $f_0(x_0)$, density derivative value $f_0'(x_0)$ and mode $m_0$. We consider $X_i \overset{\mathrm{i.i.d.}}{\sim} \mathrm{Beta}(2,3)$ and $x_0 = 0.5$. Its density function $f_0(x) = 12 x(1-x)^2 = e^{\varphi_0(x)}$ where $\varphi_0(x) = \log(x) + 2\log(1-x) + \log(12)$ is concave, and thus $f_0$ is a log-concave density. The mode of $f_0$ is $m_0 = 1/3$. 

For each data set $\{X_i, 1\le i \le n\}$, we use the constrained Newton method implemented in \verb|R| package \verb|cnmlcd| \cite{liu2018fast} to compute the log-concave MLE $\widehat{f}_n$. This algorithm is much faster than the active set algorithm from \verb|R| package \verb|logcondens| \cite{dumbgen2009maximum}. We construct the proposed CIs defined in \eqref{def:CI_log_concave_0_1} and \eqref{def:CI_mode_log_concave} with approximate critical values in Table \ref{tab:cv_abs}, check if the CIs cover the truths, and report the CI lengths. We repeat this procedure $10^4$ times and approximate the coverage probabilities by the relative frequencies of successful coverage. The simulated coverage probabilities are reported in Figure \ref{fig:log_concave_comp:a}, and the lengths of these CIs are reported in Figures \ref{fig:log_concave_comp:b}-\ref{fig:log_concave_comp:d}. Note that by the limiting distribution theories of these local parameters in \eqref{limit_log_concave} and \eqref{limit_log_concave_mode}, their symmetric oracle CIs are: 
\begin{align*}
& \Big[\widehat{f}_n(x_0) \pm \big((f_0(x_0))^3|\varphi_0''(x_0)|/24 \big)^{1/5} n^{-2/5} c_{\delta}( |\mathbb{H}_2^{(2)}(0)| ) \Big] \hbox{ for } f_0(x_0),
\\ \nonumber
& \Big[\widehat{f}'_n(x_0) \pm \big((f_0(x_0))^4 |\varphi_0''(x_0)|^3 /24^3 \big)^{1/5} n^{-1/5} c_{\delta}( |\mathbb{H}_2^{(3)}(0)| ) \Big] \hbox{ for } f'_0(x_0), \hbox{ and }
\\ \nonumber
& \Big[\widehat{m}_n \pm \big(24^2f_0(m_0)/(f_0''(m_0))^2\big)^{1/5} n^{-1/5} c_{\delta}( | [\mathbb{H}_2^{(2)}]_{\mathrm{m}}| ) \Big] \hbox{ for } m_0. 
\end{align*}
The red dashed lines in \ref{fig:log_concave_comp:b}-\ref{fig:log_concave_comp:d} represent the lengths of these oracle CIs. We give a brief summary below:
\begin{itemize}
	\item Compared to convex regression in Figure \ref{fig:regression_comp:a}, the convergence of the coverage probabilities of the proposed CIs for density function value $f_0(x_0)$ seems much slower in Figure \ref{fig:log_concave_comp:a}. However, as $n$ increases, the coverage is still converging to $95\%$, which supports Theorem \ref{thm:pivotal_limit_log_concave}.

	\item Based on our extensive simulation results that are not given here due to space constraint, we have observed that the coverage probabilities of the CIs for $f_0'(x_0)$ converge more slowly than that for $f_0(x_0)$, with coverage error greater than $0.02$ even for sample size $n = 1000$. This and the above observation on the CIs for $f_0(x_0)$ in Figure \ref{fig:log_concave_comp:a} perhaps imply that a large sample size may be required to conduct accurate inference for the density value $f_0(x_0)$ or the density derivative value $f_0'(x_0)$.

	\item The coverage probability errors of the CIs for the mode steadily vanishes as $n$ increases, supporting Theorem \ref{thm:pivotal_limit_mode_log_concave}. More simulation results on the CI for the mode of log-concave densities under small sample sizes can be found in the next subsection when compared to the LRT based CIs.

	\item Similar to the case of convex regression, we observe that the lengths of the proposed CIs for local parameters in log-concave densities shrink at the same rate as the oracle ones. This supports the related statements in Theorems \ref{thm:pivotal_limit_log_concave} and \ref{thm:pivotal_limit_mode_log_concave}.
\end{itemize}

\subsection{Comparison with the LRT-based CIs for mode of log-concave densities}\label{subsection:DW_comparison}

Among all the models studied in this paper, it seems that only the mode of a log-concave likelihood density has a proven LRT limit theory \cite{doss2016inference}. We here compare the numerical performance of the proposed procedure for the mode, referred to as LNE CIs, and theirs, referred to as LRT CIs. \cite{doss2019concave} conjectured that the LRT based procedure also works for the function value $f_0(x_0)$ in convex regression, but a formal theory is yet to be developed. 

We first compare coverage probabilities of the LNE and LRT CIs under different confidence levels and based on different log-concave densities. Let sample size $n = 100$. For each i.i.d.~sample $\{X_1, \ldots, X_n\}$ drawn from a distribution with log-concave density, the LNE CI and the LRT CI for its mode are computed to check if the true mode is covered. We repeat this procedure $10^4$ times and approximate the coverage probabilities with the relative frequencies of successful coverage. 

We use the \verb|R| function \verb|LCLRCImode| from package \verb|logcondens.mode| \cite{doss2016inference} to compute the LRT CIs with confidence levels $50\%$, $80\%$, $90\%$, $95\%$, $98\%$, $99\%$. The corresponding approximate critical values of $\mathbb{K}$ in \eqref{ineq:doss_wellner_piv_limit} are also given in this package. Essentially, the \verb|R| function \verb|LCLRCImode| first applies active set method to find the log-concave MLE $\widehat{f}_n$ and uses bisection method to solve the inverse problem \eqref{def:CI_DW_mode}. This means it has to conduct an LRT at every iteration. In contrast, as long as we have the log-concave MLE $\widehat{f}_n$, the LNE CI can be constructed instantly using the formula \eqref{def:CI_mode_log_concave}. It is therefore much slower to compute the LRT CIs, which is the reason why we limit this comparison to sample size $n = 100$.

The simulated coverage probabilities of the LNE and LRT CIs for the mode of log-concave densities are reported in Table \ref{tab:log_concave_comp}, rounded to 2 decimal places. Overall, we find that the performance of these two types of CIs is comparable in terms of coverage probability.

\begin{table}[htb]
	\scriptsize{
	\begin{tabular}{c|c||c|c|c|c|c|c}
			\hline 
			Distribution & CI type & 50\% CI  & 80\% CI & 90\% CI & 95\% CI & 98\% CI & 99\% CI \\
			\hline
			\multirow{2}{*}{$\chi_4^2$} & LNE & 0.58 & 0.80 & 0.88 & 0.93 & 0.97 & 0.98  \\
			& LRT & 0.47 & 0.78 & 0.89 & 0.94 & 0.98 & 0.99 \\
			\hline
			\multirow{2}{*}{$\mathrm{Beta}(2,3)$} & LNE & 0.46 & 0.74 & 0.86 & 0.93 & 0.97 & 0.98  \\
			& LRT & 0.46 & 0.77 & 0.88 & 0.93 & 0.97 & 0.99  \\
			\hline
			\multirow{2}{*}{$\mathrm{Gamma}(1,3)$} & LNE & 0.57 & 0.80 & 0.89 & 0.94 & 0.98 & 0.99 \\
			& LRT & 0.48 & 0.79 & 0.89 & 0.94 & 0.98 & 0.99 \\
			\hline
			\multirow{2}{*}{$\mathrm{Weibull}(1,1.5)$} & LNE & 0.52 & 0.76 & 0.86 & 0.92 & 0.96 & 0.98 \\
			& LRT & 0.46 & 0.77 & 0.88 & 0.93 & 0.97 & 0.99 \\
			\hline
			\multirow{2}{*}{$\mathrm{Normal}(0,1)$} & LNE & 0.54 & 0.82 & 0.90 & 0.95 & 0.98 & 0.99  \\
			& LRT & 0.47 & 0.78 & 0.89 & 0.94 & 0.98 & 0.99 \\
			\hline
		\end{tabular}
	}
	\vspace{1ex}
	\caption{Simulated coverage probabilities of the CIs for the mode of log-concave densities based on $B=10^4$ samples. Here sample size $n=100$.}
	\label{tab:log_concave_comp}
\end{table}

\setlength{\belowcaptionskip}{-10pt}
	\begin{figure}[!hbt]
	\centering
	\includegraphics[width=\textwidth]{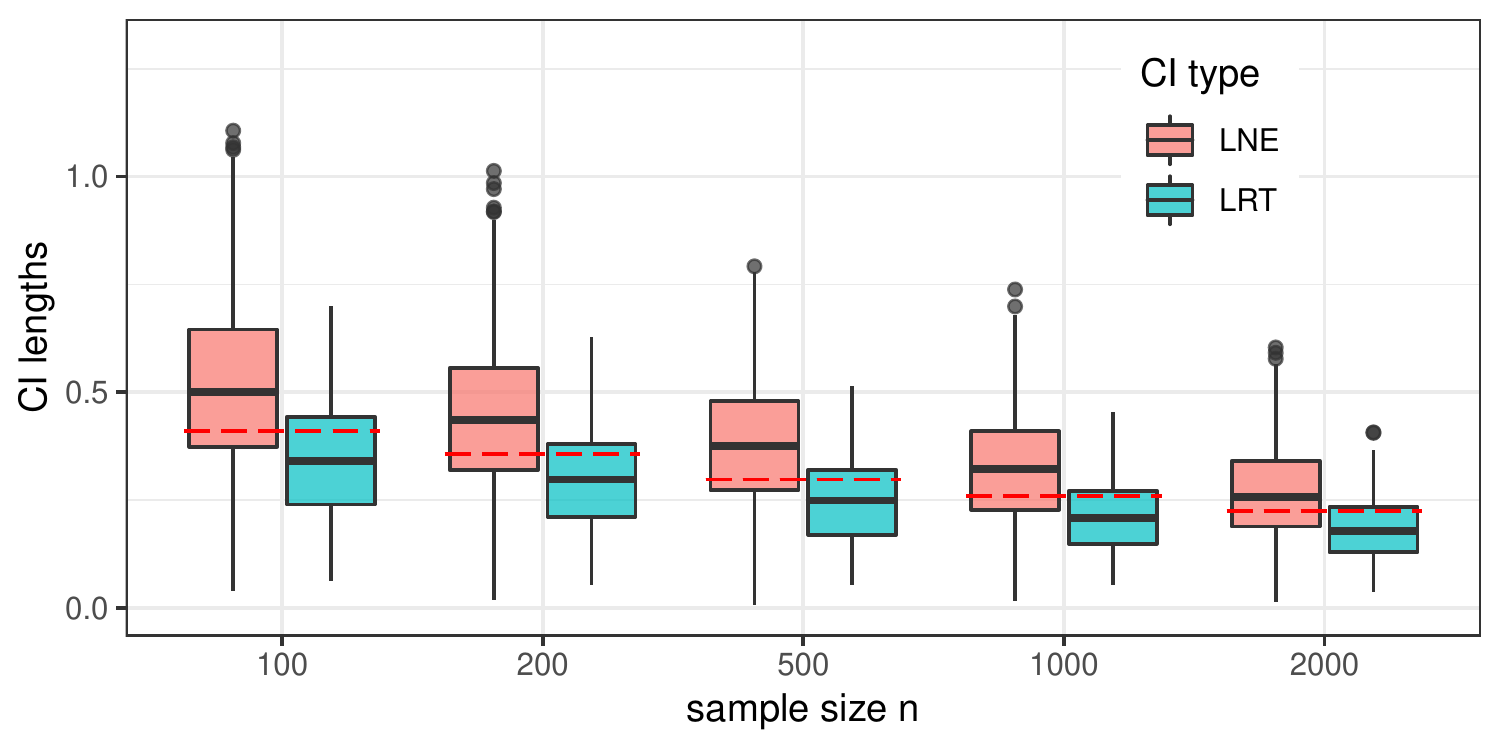}
	\caption{Box plots of the lengths of the LNE and LRT CIs for the mode of $\mathrm{Beta}(2,3)$ distribution. Red dashed lines represent the lengths of oracle CIs.}
	\label{fig:length_comp_log_concave}
\end{figure}

We next compare the lengths of these two types of CIs. In this comparison, we only repeat the procedure $B = 1000$ times and evaluate their performance for $n \in \{100, 200, 500, 1000, 2000\}$. It seems hard to run simulation with greater $B$ and $n$ due to the slow computation of LRT CIs. Here we consider $f_0$ to be the density of $\mathrm{Beta}(2,3)$. In Figure \ref{fig:length_comp_log_concave}, we give box plots of the lengths of the LNE and LRT CIs for the mode of $\mathrm{Beta}(2,3)$ distribution. The red dashed lines represent the lengths of the oracle CIs from limit distribution theory \eqref{limit_log_concave_mode} as discussed in Section 5.2.2. As we can see from Figure \ref{fig:length_comp_log_concave}, the LRT CIs are generally narrower than the LNE CIs and, interestingly, also narrower than the oracle CIs. The wider length of LNE CI is not a real surprise since only using $\widehat{u}_{\mathrm{m}}$ and $\widehat{v}_{\mathrm{m}}$ to construct CIs in finite samples is likely to bring in a fair amount of variation.

\appendix

\section{Proof of results in Section~\ref{section:pivot_limit}}\label{section:proof_main}

\subsection{Preliminaries}\label{pf:preliminary}
	As our proof of the results in Section \ref{section:pivot_limit} relies on the proof of Theorem \ref{thm:limiting_dist_convex} that is given in \cite{groeneboom2001estimation}, we shall give a proof sketch of Theorem \ref{thm:limiting_dist_convex} in this subsection. We will focus on the case $\alpha=2$ with a fixed design, and drop the dependence on $\alpha$ in the notation in the proof below. The localization arguments for $\alpha\neq 2$ in random design are carried out in \cite{ghosal2017univariate}.
	
	\begin{proof}[Proof sketch of Theorem \ref{thm:limiting_dist_convex}]
	Let
	\begin{align}\label{def:proc}
	\mathbb{S}_n(u)&\equiv \frac{1}{n}\sum_{i=1}^n Y_i  \bm{1}_{ \{ X_i\leq u \} } ,\quad \mathbb{Y}_n(t) \equiv \int_0^t \mathbb{S}_n(u)\,\d{u},
	\\ \nonumber
	\mathbb{R}_n(t;f)& \equiv \frac{1}{n}\sum_{i=1}^n f(X_i)  \bm{1}_{ \{ X_i\leq t \} } , \quad \mathbb{H}_n(t;f) \equiv \int_0^t \mathbb{R}_n(u;f)\,\d{u},
	\\ \nonumber
	\widetilde{\mathbb{R}}_n(t;f)& \equiv \int_0^t f(u)\,\d{u},\quad \widetilde{\mathbb{H}}_n(t;f) \equiv \int_0^t \widetilde{\mathbb{R}}_n(u;f)\,\d{u}.
	\end{align}
	Characterization of the LSE (see~\cite[Lemma 2.6]{groeneboom2001estimation}) shows that a piecewise linear convex function $f$ is the LSE if and only if $\mathbb{H}_n^{(3)}(;f)$ is a piecewise non-decreasing constant function and $\mathbb{H}_n(;f)$ majorizes $\mathbb{Y}_n$: $\mathbb{H}_n(t;f)\geq \mathbb{Y}_n(t)$ with equality taken at jumps of $\mathbb{H}_n^{(3)}(;f)$. 
	For the LSE $\widehat{f}_n$, we write for notational simplicity
	\begin{align*}
	&\mathbb{R}_n(t)\equiv \mathbb{R}_n(t;\widehat{f}_n),\quad \mathbb{H}_n(t)\equiv \mathbb{H}_n(t;\widehat{f}_n),\\
	&  \widetilde{\mathbb{R}}_n(t)\equiv  \widetilde{\mathbb{R}}_n(t;\widehat{f}_n),\quad \widetilde{\mathbb{H}}_n(t)\equiv \widetilde{\mathbb{H}}_n(t;\widehat{f}_n).
	\end{align*}
	
	The limit distribution theory is based on the localization of this characterization. In essence, we wish to define local counterparts $\mathbb{Y}_n^{\mathrm{loc}}, \mathbb{H}_n^{\mathrm{loc}}$ of $\mathbb{Y}_n, \mathbb{H}_n$ in such a way that (i) $\mathbb{H}_n^{\mathrm{loc}}$ is nicely related to $\widehat{f}_n$ much as $\mathbb{H}_n$ does, (ii) characterization $\mathbb{Y}_n^{\mathrm{loc}}(t)\leq \mathbb{H}_n^{\mathrm{loc}}(t)$ is preserved with equality taken at jumps of $(\mathbb{H}_n^{\mathrm{loc}})^{(3)}$, and (iii) a non-trivial weak limit of $\mathbb{Y}_n^{\mathrm{loc}}$ can be computed, and the sequence $\{\mathbb{H}_n^{\mathrm{loc}}\}$ along with its derivatives up to order three remain tight as $n \to \infty$ in the suitable sense. Then a standard argument shows the existence of a limiting process for $\{ \mathbb{H}_n^{\mathrm{loc}} \}$
	satisfying conditions indicated above. The uniqueness for such processes with these conditions then well defines the desired process. As a technical subtle point, (i) and (iii) seem not feasible simultaneously for one single process, so we will define two asymptotically equivalent processes $\mathbb{H}_n^{\mathrm{loc}}(t)$ and $\widetilde{\mathbb{H}}_n^{\mathrm{loc}}(t)$ below that satisfy these two requirements separately. Now let us construct local process counterparts of $\mathbb{Y}_n(t)$ and $\mathbb{H}_n(t)$ as follows. Define
	\begin{align}\label{def:local_proc}
	\mathbb{Y}_n^{\mathrm{loc}}(t)&\equiv n^{4/5} \int_{x_0}^{x_0 + n^{-1/5}t} \bigg[ \mathbb{S}_n(v) - \mathbb{S}_n(x_0)
	\\ \nonumber
	&\qquad\qquad - \int_{x_0}^v \big(f_0(x_0)+(u-x_0)f_0'(x_0)\big)\,\d{\mathbb{F}_n(u)}\bigg]\,\d{v},
	\\ \nonumber
	\mathbb{H}_n^{\mathrm{loc}}(t) &\equiv n^{4/5} \int_{x_0}^{x_0 + n^{-1/5}t} \bigg[ \mathbb{R}_n(v) - \mathbb{R}_n(x_0)
	\\ \nonumber
	&\qquad\qquad  - \int_{x_0}^v \big(f_0(x_0)+(u-x_0)f_0'(x_0)\big)\,\d{\mathbb{F}_n(u)}\bigg]\,\d{v} +A_n+B_n t,
	\\ \nonumber
	\widetilde{\mathbb{H}}_n^{\mathrm{loc}}(t) &\equiv n^{4/5} \int_{x_0}^{x_0 + n^{-1/5}t} \bigg[ \widetilde{\mathbb{R}}_n(v) - \widetilde{\mathbb{R}}_n(x_0)
	\\ \nonumber
	&\qquad\qquad  - \int_{x_0}^v \big(f_0(x_0)+(u-x_0)f_0'(x_0)\big)\,\d{u}\bigg]\,\d{v} +A_n+B_n t,
	\end{align}
	where $\mathbb{F}_n(u) \equiv n^{-1}\sum_{i=1}^n \bm{1}_{\{ X_i \le u\}}$, and
	\begin{align*}
	A_n = n^{4/5} \big(\mathbb{H}_n(x_0)-\mathbb{Y}_n(x_0)\big) \hbox{ and } B_n = n^{3/5} \big(\mathbb{R}_n(x_0)-\mathbb{S}_n(x_0) \big).
	\end{align*}
	The following statements are proved
	in \cite{groeneboom2001estimation}:
	\begin{enumerate}
		\item $\mathbb{H}_n^{\mathrm{loc}}(t) \ge \mathbb{Y}_n^{\mathrm{loc}}(t)$ with equality holds when $x_0 + n^{-1/5}t$ is a kink.
		\item It holds that
		\begin{align*}
		\mathbb{Y}_n^{\mathrm{loc}}(t) \rightsquigarrow \mathbb{Y}(t; f_0) \equiv \sigma \int_{0}^t \mathbb{B}(s)\,\d{s} + f_0^{(2)}(x_0) t^4/4!,
		\end{align*}
		in $C([-T,T])$ for any $T>0$.
		\item For any $T>0$, 
		\begin{align*}
		\sup_{t \in [-T,T]} \abs{\mathbb{H}_n^{\mathrm{loc}}(t) - \widetilde{\mathbb{H}}_n^{\mathrm{loc}}(t)} = \mathfrak{o}_{\mathbf{P}}(1).
		\end{align*}
		\item The process $\widetilde{\mathbb{H}}_n^{\mathrm{loc}}$ has derivatives 
		\begin{align*}
		(\widetilde{\mathbb{H}}_n^{\mathrm{loc}})^{(2)}(t)& = n^{2/5} \big\{ \widehat{f}_n(x_0+n^{-1/5}t)-f_0(x_0)-n^{-1/5}f_0'(x_0)t \big\},
		\cr
		(\widetilde{\mathbb{H}}_n^{\mathrm{loc}})^{(3)}(t)& = n^{1/5}\big\{ \widehat{f}_n'(x_0+n^{-1/5}t)-f_0'(x_0) \big\},
		\end{align*}
		so we have the key identities:
		\begin{align*}
		n^{2/5} \big\{ \widehat{f}_n(x_0)-f_0(x_0)\big\} &= (\widetilde{\mathbb{H}}_n^{\mathrm{loc}})^{(2)}(0),\\
		n^{1/5} \big\{\widehat{f}_n'(x_0)-f_0'(x_0)\big\} & = (\widetilde{\mathbb{H}}_n^{\mathrm{loc}})^{(3)}(0).
		\end{align*}
		\item $\widetilde{\mathbb{H}}_n^{\mathrm{loc}}(t)$ and its derivatives up to order $3$ are tight on compacta.
	\end{enumerate}
	
	Now the limit distribution theory for the least square estimator $\widehat{f}_n$ follows by taking $n \to \infty$ and the a.s.~uniqueness of $\mathbb{H}(t; f_0)$, the ``invelope'' function of $\mathbb{Y}(t; f_0)$. Note that with 
	\begin{align}\label{value:gamma}
	\gamma_0 = \sigma \big( 4!\sigma /f_{0}^{(2)}(x_0) \big)^{3/5} \hbox{ and } \gamma_1 = \big(f_{0}^{(2)}(x_0)/(4! \sigma) \big)^{2/5},
	\end{align}
	we have $\gamma_0 \gamma_1^{3/2} = \sigma$ and $\gamma_0\gamma_1^{4} = f_0^{(2)}(x_0)/4!$, so that
	\begin{align}\label{eqn:brownian_scaling}
	\gamma_0 \mathbb{Y}(\gamma_1 t) &= \gamma_0 \bigg(\int_0^{\gamma_1 t} \mathbb{B}(s)\,\d{s} + \gamma_1^4 t^4\bigg)\\
	& =_d \gamma_0 \bigg( \gamma_1^{3/2}\int_0^{ t} \mathbb{B}(s)\,\d{s} + \gamma_1^4 t^4\bigg)\quad (\textrm{by Brownian scaling}) \nonumber\\
	& = \mathbb{Y}(t; f_0). \nonumber
	\end{align}
	Consequently
	\begin{align*}
	\begin{pmatrix}
	(\widetilde{\mathbb{H}}_n^{\mathrm{loc}})^{(2)}(0) \\
	(\widetilde{\mathbb{H}}_n^{\mathrm{loc}})^{(3)}(0) 
	\end{pmatrix} \rightsquigarrow 
	\begin{pmatrix}
	\mathbb{H}^{(2)}(0; f_0) \\
	\mathbb{H}^{(3)}(0; f_0)
	\end{pmatrix} =_d
	\begin{pmatrix}
	\gamma_0\gamma_1^2 \mathbb{H}^{(2)}(0) \\
	\gamma_0\gamma_1^3 \mathbb{H}^{(3)}(0) 
	\end{pmatrix},
	\end{align*}
	where $\gamma_0 \gamma_1^2 = \sigma^{4/5} d_{2}^{(0)}(f_0, x_0)$ and $\gamma_0 \gamma_1^3 = \sigma^{2/5} d_{2}^{(1)}(f_0, x_0)$.
	
The proof sketch of Theorem \ref{thm:limiting_dist_convex} is now complete.
	\end{proof}
	
	\begin{remark}
		For $f_0$ locally $C^{\alpha}$ with general $\alpha$, $\mathbb{Y}(t; f_0) = \sigma \int_{0}^t \mathbb{B}(s)\,\d{s} + f_0^{(\alpha)}(x_0) t^{\alpha + 2} / (\alpha + 2)!$, the scaling relationship reads as follows: Let
		\begin{align*}
		\gamma_0 \gamma_1^{3/2} = \sigma,\quad \gamma_0\gamma_1^{\alpha+2} = \frac{f_0^{(\alpha)}(x_0)}{(\alpha+2)!},
		\end{align*}
		so that 
		\begin{align*}
		\gamma_0 = \sigma \Big( \frac{\sigma (\alpha + 2)!}{f_0^{(\alpha)}(x_0)} \Big)^{3/(2\alpha +1)}  \hbox{ and } \gamma_1 = \Big( \frac{f_{0}^{(\alpha)}(x_0)}{\sigma(\alpha + 2)!} \Big)^{2/(2 \alpha + 1)}.
		\end{align*}
		The rest remains the same.
	\end{remark}
	
	\begin{remark}
		The existence and a.s. uniqueness of the process $\mathbb{H}_\alpha$ in Theorem \ref{thm:limiting_dist_convex} is established formally for $\alpha=2$ in \cite{groeneboom2001canonical}, but an entirely similar arguments applies to general $\alpha$. 
	\end{remark}
	
	\subsection{Proof of Theorem \ref{thm:pivotal_limit_fcn}}\label{pf:pivotal_limit_fcn} 
		
	\begin{proof}[Additional notation]\renewcommand{\qedsymbol}{}
	Let $\widehat{h}_- = \widehat{h}_-(x_0) \equiv n^{1/5}(x_0-\widehat{u}(x_0))$ and $\widehat{h}_+ = \widehat{h}_+(x_0) \equiv n^{1/5}(\widehat{v}(x_0)-x_0)$. Then $\widehat{h}_\pm= \mathcal{O}_{\mathbf{P}}(1)$ by (essentially) \cite[Lemma 8]{mammen1991nonparametric}. Let $h^\ast_-(f_0)$ (resp.~$h^\ast_+(f_0)$) be the absolute value of the location of the first touch point of the pair $(\mathbb{H}(\cdot\,; f_0),\mathbb{Y}(\cdot\,; f_0))$ to the left (resp.~right) of $0$, where $\mathbb{H}(\cdot\,; f_0)$ is the limit process satisfying the characterization conditions with respect to $\mathbb{Y}(\cdot\,; f_0)$. As $\mathbb{H}(\cdot\,; f_0)$ is a random piecewise cubic polynomial, while $\mathbb{Y}(t; f_0) = O_{\mathrm{a.s.}}(t^4)$ as $t \to \infty$, we see that $h^\ast_-(f_0) \vee h^\ast_+(f_0) <\infty$ a.s.~The fact that $h^\ast_-(f_0) =0 , h^\ast_+(f_0)=0$ occurs with probability $0$ follows from \cite[Corollary 2.1]{groeneboom2001canonical}. So w.p.~$1$, $h^\ast_-(f_0), h^\ast_+(f_0) \in (0,\infty)$. 
	\end{proof} 
	
	\begin{proof}[High level idea and difficulty]\renewcommand{\qedsymbol}{}
	One intuitive and tempting idea of the proof is to write $\widehat{h}_{\pm}$ as a functional of $(\widetilde{\mathbb{H}}_n^{\mathrm{loc}})^{(2)}$, that is, $\widehat{h}_{\pm} = \mathcal{H}_{\pm}\big((\widetilde{\mathbb{H}}_n^{\mathrm{loc}})^{(2)}\big)$, and then apply continuous mapping theory. However, as the process $(\widetilde{\mathbb{H}}_n^{\mathrm{loc}})^{(2)}$ converges uniformly to its limit on compact intervals, this approach requires continuity of the the functional $\mathcal{H}_{\pm}$ with respect to the topology of compact uniform convergence. Unfortunately, continuity of $\mathcal{H}_{\pm}$ in this topology is false in general, as can be seen by the following counter-example. Let $\{f_n\},f_\infty$ be convex functions symmetric about $0$, where 
	\begin{align}\label{ineq:pivotal_lim_fcn_0}
	f_n(x)\equiv\max\{n^{-1}(x-1)_+, (x-2)_+\},\quad f_\infty(x)\equiv (x-2)_+
	\end{align}
	on $[0,\infty)$. Then $f_n$ converges to $f_\infty$ uniformly on compacta, but the first positive kink of $f_n$ is $1$ for any $n$, while the first positive kink of $f_\infty$ is $2$.  On the other hand, one would expect that counter-examples of the type (\ref{ineq:pivotal_lim_fcn_0}) can happen for the process $(\widetilde{\mathbb{H}}_n^{\mathrm{loc}})^{(2)}$ only with vanishing probability, as otherwise one of the key characterizations (\ref{ineq:pivotal_lim_fcn_1})-(\ref{ineq:pivotal_lim_fcn_2}) below will be violated in the limit; or put it geometrically, one of the touch points of $(\mathbb{H}_n^{\mathrm{loc}},\mathbb{Y}_n^{\mathrm{loc}})$ will be lost in the limit by violation of (\ref{ineq:pivotal_lim_fcn_2}) ahead. 
	\end{proof}
    
    \begin{proof}[Proof of Theorem \ref{thm:pivotal_limit_fcn}]
    Now we make the intuition outlined above precise via a dual characterization of $\widehat{h}_{\pm}$ using both $(\widetilde{\mathbb{H}}_n^{\mathrm{loc}})^{(2)}$ and the pair $(\mathbb{H}_n^{\mathrm{loc}},\mathbb{Y}_n^{\mathrm{loc}})$. Let 
	\begin{align*}
	\widetilde{\Delta}_{n, \pm}^{\mathrm{loc}}(\mathfrak{w}) = 2 (\widetilde{\mathbb{H}}_n^{\mathrm{loc}})^{(2)}(\pm \mathfrak{w} /2)  - (\widetilde{\mathbb{H}}_n^{\mathrm{loc}})^{(2)}(\pm \mathfrak{w}) - (\widetilde{\mathbb{H}}_n^{\mathrm{loc}})^{(2)}(0)\leq 0.
	\end{align*}
	Due to convexity of $\widehat{f}_n$, for any $ \mathfrak{w}  \in \R_{>0}$,
	\begin{align}\label{ineq:pivotal_lim_fcn_1}
	\widehat{h}_\pm<  \mathfrak{w}  &\Leftrightarrow 2 \widehat{f}_n(x_0 \pm n^{-1/5} \mathfrak{w}/2) - \widehat{f}_n(x_0 \pm n^{-1/5} \mathfrak{w}) - \widehat{f}_n(x_0) <0\\
	& \Leftrightarrow \widetilde{\Delta}_{n, \pm}^{\mathrm{loc}}(\mathfrak{w}) <0.\nonumber
	\end{align}
	On the other hand, on the event $E_n$ that $x_0$ is not a kink of $\widehat{f}_n$ (which occurs with probability tending to one),
	\begin{align}\label{ineq:pivotal_lim_fcn_2}
	\widehat{h}_\pm \leq \mathfrak{w}  
	&\Leftrightarrow \sup_{t \in \pm[0, \mathfrak{w}]}\{\mathbb{Y}_n^{\mathrm{loc}}(t)-\mathbb{H}_n^{\mathrm{loc}}(t)\} = 0\\
	& \Leftrightarrow (\mathbb{H}_n^{\mathrm{loc}},\mathbb{Y}_n^{\mathrm{loc}})|_{\pm[0,\mathfrak{w}]} \in S_{\pm}(\mathfrak{w}),\nonumber
	\end{align}
	where for $0\leq u_1\leq u_2$, $\pm[u_1,u_2]$ is interpreted as $[u_1,u_2]$ for $+$ and $[-u_2,-u_1]$ for $-$, 
	and
	\begin{align*}
	S_{\pm}(\mathfrak{w})&\equiv \bigg\{(h,y) \in \big(C(\pm[0,\mathfrak{w} ])\big)^2,  \sup_{t \in \pm[0,\mathfrak{w}]}\{y(t)-h(t)\}=0 \bigg\}
	\end{align*}
	 is a closed set of $\big(C(\pm[0,\mathfrak{w}])\big)^2$ with respect to the topology induced by the product supremum norm. See Figure \ref{fig:proof} for an illustration of the above equivalence (\ref{ineq:pivotal_lim_fcn_1})-(\ref{ineq:pivotal_lim_fcn_2}). We employ two different characterizations (\ref{ineq:pivotal_lim_fcn_1})-(\ref{ineq:pivotal_lim_fcn_2}) using $(\widetilde{\mathbb{H}}_n^{\mathrm{loc}})^{(2)}$ and $(\mathbb{H}_n^{\mathrm{loc}},\mathbb{Y}_n^{\mathrm{loc}})$ respectively to maintain openness and closedness topological properties in the equivalence characterization of $\widehat{h}_\pm$. As suggested by the counter-example (\ref{ineq:pivotal_lim_fcn_0}), such different characterizations are essential.

	 \begin{figure}
	 	\centering
	 	\includegraphics[width=12cm]{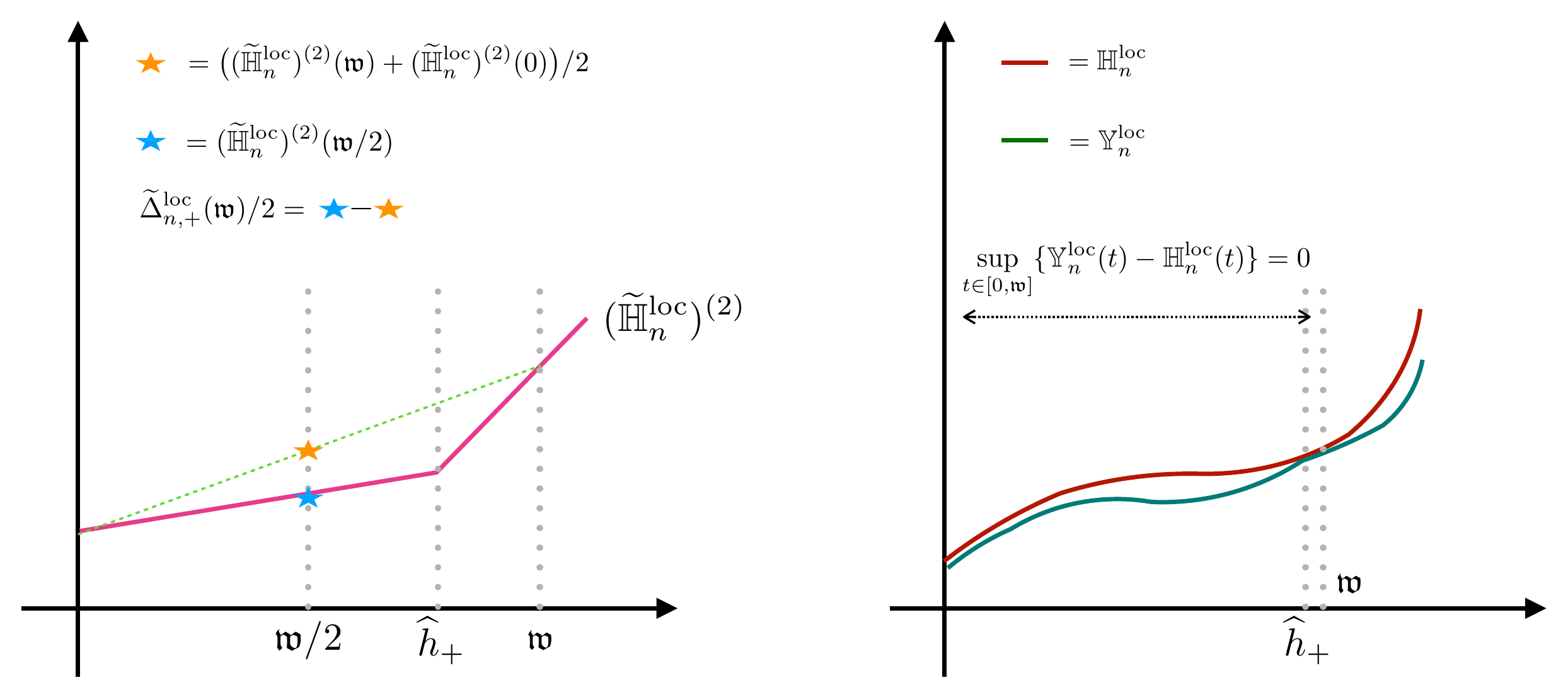}
	 	\caption{Figure illustration of the equivalence (\ref{ineq:pivotal_lim_fcn_1})-(\ref{ineq:pivotal_lim_fcn_2}).}
	 	\label{fig:proof}
	 \end{figure}
	
    Using e.g.,~Skorokhod's representation theorem (see~\cite[Theorem 6.7]{bilingsley1999convergence}), it is easily shown that
	\begin{align*}
	&\bigg((\widetilde{\mathbb{H}}_n^{\mathrm{loc}})^{(2)}(0),\,(\widetilde{\mathbb{H}}_n^{\mathrm{loc}})^{(3)}(0),\,\widetilde{\Delta}_{n, \pm}^{\mathrm{loc}}(\mathfrak{w}),\,(\mathbb{H}_n^{\mathrm{loc}},\mathbb{Y}_n^{\mathrm{loc}})|_{\pm[0,\mathfrak{w}]}\bigg)\\
	&\rightsquigarrow \bigg( \mathbb{H}^{(2)}(0; f_0),\,\mathbb{H}^{(3)}(0; f_0),\,\Delta_{\pm}(\mathfrak{w}),\,(\mathbb{H}(\cdot;f_0),\mathbb{Y}(\cdot;f_0))|_{\pm[0,\mathfrak{w}]} \bigg)
	\end{align*}
	in $\R^3 \times \big(C(\pm[0,\mathfrak{w}])\big)^2$. Here
	\begin{align*}
	\Delta_{\pm}(\mathfrak{w}) = 2 \mathbb{H}^{(2)}(\pm \mathfrak{w}/2; f_0)  - \mathbb{H}^{(2)}(\pm \mathfrak{w}; f_0) - \mathbb{H}^{(2)}(0; f_0)
	\end{align*}
	are the limiting counterparts of $\widetilde{\Delta}_{n, \pm}^{\mathrm{loc}}$.
	
	Fix any continuity point $(\mathfrak{s},\mathfrak{t},\mathfrak{u},\mathfrak{v}) \in \R\times \R \times \R_{> 0}\times \R_{> 0}$ of the random vector $\big(\mathbb{H}^{(2)}(0; f_0), \mathbb{H}^{(3)}(0; f_0), h^\ast_-(f_0), h^\ast_+(f_0) \big)$. By Portmanteau theorem for general metric-space valued random variables (see~\cite[Theorem 11.1.1]{dudley2002real}), 
	\begin{align*}
	&\limsup_{n \to \infty} \Prob\bigg\{ n^{2/5} \big(\widehat{f}_n(x_0)-f_0(x_0)\big)\leq \mathfrak{s}, n^{1/5} \big(\widehat{f}_n'(x_0)-f_0'(x_0)\big) \leq \mathfrak{t}, \\
	&\qquad\qquad\qquad  \widehat{h}_-\leq \mathfrak{u}, \widehat{h}_+\leq \mathfrak{v}\bigg\}\\
	&\leq \limsup_{n \to \infty} \Prob\bigg\{ \bigg((\widetilde{\mathbb{H}}_n^{\mathrm{loc}})^{(2)}(0), (\widetilde{\mathbb{H}}_n^{\mathrm{loc}})^{(3)}(0), (\mathbb{H}_n^{\mathrm{loc}},\mathbb{Y}_n^{\mathrm{loc}})|_{[-\mathfrak{u}, 0]}, (\mathbb{H}_n^{\mathrm{loc}},\mathbb{Y}_n^{\mathrm{loc}})|_{[0,\mathfrak{v}]} \bigg)\\
	&\qquad\qquad\qquad\qquad \in (-\infty, \mathfrak{s}]\times (-\infty,\mathfrak{t}]\times S_{-}(\mathfrak{u})\times S_+(\mathfrak{v}) \bigg\} +\limsup_{n \to \infty} \Prob(E_n^c) \\
    &\leq \Prob\bigg\{ \bigg( \mathbb{H}^{(2)}(0; f_0) , \mathbb{H}^{(3)}(0; f_0),  (\mathbb{H}(\cdot;f_0),\mathbb{Y}(\cdot;f_0))|_{[-\mathfrak{u}, 0]}, (\mathbb{H}(\cdot;f_0),\mathbb{Y}(\cdot;f_0))|_{[0,\mathfrak{v}]}\bigg)\\
    &\qquad\qquad\qquad\qquad \in (-\infty, \mathfrak{s}]\times (-\infty,\mathfrak{t}]\times S_-(\mathfrak{u})\times S_{+}(\mathfrak{v}) \bigg\} \\
	& =  \Prob\bigg\{ \mathbb{H}^{(2)}(0; f_0)\leq \mathfrak{s},\,\mathbb{H}^{(3)}(0; f_0)\leq \mathfrak{t},\,h^\ast_-(f_0) \leq \mathfrak{u},\,h^\ast_+(f_0)\leq \mathfrak{v}\bigg\}.
	\end{align*}
	In the last equality we used that for any $\mathfrak{w} \in \R_{>0}$, with probability $1$ it holds that $h_\pm^\ast(f_0) \leq \mathfrak{w}  \Leftrightarrow (\mathbb{H}(\cdot;f_0),\mathbb{Y}(\cdot;f_0))|_{\pm[0,\mathfrak{w}]} \in S_{\pm}(\mathfrak{w})$.
	
	The other direction is easier, and essentially follows directly from Portmanteau theorem for vector-valued random variables: for any $\epsilon >0$,
	\begin{align*}
	&\liminf_{n \to \infty} \Prob\bigg\{ n^{2/5} \big(\widehat{f}_n(x_0)-f_0(x_0)\big)\leq \mathfrak{s},\,n^{1/5} \big(\widehat{f}_n'(x_0)-f_0'(x_0)\big) \leq \mathfrak{t}, \\
	&\qquad\qquad\qquad\qquad  \widehat{h}_-\leq \mathfrak{u},\,\widehat{h}_+\leq \mathfrak{v}\bigg\}\\
	&\geq  \liminf_{n \to \infty} \Prob\bigg\{ (\widetilde{\mathbb{H}}_n^{\mathrm{loc}})^{(2)}(0) < \mathfrak{s} - \varepsilon,\,(\widetilde{\mathbb{H}}_n^{\mathrm{loc}})^{(3)}(0) < \mathfrak{t} - \varepsilon, \\
	&\qquad\qquad\qquad\qquad \widetilde{\Delta}_{n, -}^{\mathrm{loc}}(\mathfrak{u} - \varepsilon) < 0,\,\widetilde{\Delta}_{n, +}^{\mathrm{loc}}(\mathfrak{v} - \varepsilon) < 0 \bigg\}\\
	& \geq \Prob\bigg\{ \mathbb{H}^{(2)}(0; f_0) < \mathfrak{s} - \varepsilon,\,\mathbb{H}^{(3)}(0; f_0) < \mathfrak{t} - \varepsilon,
	\\
	& \hspace{8em} \Delta_{n, -}(\mathfrak{u} - \varepsilon) < 0,\,\Delta_{n, +}(\mathfrak{v} - \varepsilon)  < 0 \bigg\}\\
	& =  \Prob\bigg\{ \mathbb{H}^{(2)}(0; f_0) < \mathfrak{s} - \varepsilon,\,\mathbb{H}^{(3)}(0; f_0) < \mathfrak{t} - \varepsilon,
	\\
	&\hspace{8em} h^\ast_-(f_0) < \mathfrak{u} - \varepsilon,\,h^\ast_+(f_0) < \mathfrak{v} - \varepsilon \bigg\}.
	\end{align*}
	As $\epsilon \downarrow 0$ and by the continuity of $(\mathfrak{s},\mathfrak{t},\mathfrak{u},\mathfrak{v})$, we have proved the joint convergence in distribution:
	\begin{align*}
	&\bigg(n^{2/5} \big(\widehat{f}_n(x_0)-f_0(x_0)\big),\,n^{1/5}\big(\widehat{f}_n'(x_0)-f_0'(x_0)\big),\,\widehat{h}_-(x_0),\,\widehat{h}_+(x_0) \bigg)\\
	&\qquad \qquad \rightsquigarrow \Big(\mathbb{H}^{(2)}(0; f_0),\,\mathbb{H}^{(3)}(0; f_0),\,h^\ast_-(f_0),\,h^\ast_+(f_0) \Big).
	\end{align*}
	By continuous mapping, we conclude that
	\begin{align*}
	&\begin{pmatrix}
	\sqrt{n\big(\widehat{v}(x_0)-\widehat{u}(x_0)\big)}\big(\widehat{f}_n(x_0)-f_0(x_0)\big) \\
	\sqrt{n\big(\widehat{v}(x_0)-\widehat{u}(x_0)\big)^3}\big(\widehat{f}_n'(x_0)-f_0'(x_0)\big)
	\end{pmatrix}
	\cr
	&= \begin{pmatrix}
	\sqrt{n^{1/5}\big(\widehat{v}(x_0)-\widehat{u}(x_0)\big)} \cdot n^{2/5}\big(\widehat{f}_n(x_0)-f_0(x_0)\big) \\
	\sqrt{n^{3/5}\big(\widehat{v}(x_0)-\widehat{u}(x_0)\big)^3} \cdot n^{1/5} \big(\widehat{f}_n'(x_0)-f_0'(x_0)\big)
	\end{pmatrix}
	\cr
	&= \begin{pmatrix}
	\sqrt{\widehat{h}_-(x_0) + \widehat{h}_+(x_0)} \cdot n^{2/5}\big(\widehat{f}_n(x_0)-f_0(x_0)\big) \\
	\sqrt{ (\widehat{h}_-(x_0) + \widehat{h}_+(x_0) )^3 } \cdot n^{1/5} \big(\widehat{f}_n'(x_0)-f_0'(x_0)\big)
	\end{pmatrix}
	\cr
	&\rightsquigarrow \begin{pmatrix}
	\displaystyle \sqrt{ h^\ast_+(f_0)+ h^\ast_-(f_0) }\cdot \mathbb{H}^{(2)}(0; f_0) \\
	\displaystyle \sqrt{(h^\ast_+(f_0)+ h^\ast_-(f_0) )^3}\cdot \mathbb{H}^{(3)}(0; f_0)
	\end{pmatrix}.
	\end{align*}
	Now we verify that the distribution of 
	\begin{align*}
	\Big( \sqrt{h^\ast_+(f_0) + h^\ast_-(f_0)}\cdot \mathbb{H}^{(2)}(0; f_0),\,\sqrt{(h^\ast_+(f_0) + h^\ast_-(f_0) )^3}\cdot \mathbb{H}^{(3)}(0; f_0) \Big)
	\end{align*}
	is pivotal with respect to the nuisance parameter $f_0''(x_0)$.	We recall $h^\ast_{\pm}$ are the touch points for $(\mathbb{H},\mathbb{Y})$ defined in similar fashion to $h^\ast_\pm(f_0)$ for $(\mathbb{H}(\cdot\,; f_0), \mathbb{Y}(\cdot\,; f_0))$. Here $(\mathbb{H},\mathbb{Y})$ is defined in Theorem \ref{thm:limiting_dist_convex} with $\alpha=2$. We wish to relate $h^\ast_\pm(f_0)$ to $h^\ast_{\pm}$. With the same $\gamma_0$ and $\gamma_1$ as in \eqref{value:gamma} that satisfy $\mathbb{Y}(t; f_0) = \gamma_0 \mathbb{Y}(\gamma_1 t)$, it follows that
	\begin{align*}
	\mathbb{H}^{(2)}(t; f_0) = \gamma_0 \gamma_1^2 \mathbb{H}^{(2)}(\gamma_1 t),\,\mathbb{H}^{(3)}(t; f_0) = \gamma_0 \gamma_1^3 \mathbb{H}^{(3)}(\gamma_1 t),\,h^\ast_{\pm} = \gamma_1 h^\ast_\pm(f_0).
	\end{align*}
	Hence, due to $\gamma_0 \gamma_1^{3/2} = \sigma$,
	\begin{align*}
	\sqrt{h^\ast_+(f_0) +   h^\ast_-(f_0)}\cdot \mathbb{H}^{(2)}(0; f_0) & = \sqrt{\big( h^\ast_{+}+ h^\ast_{-}\big)/\gamma_1}\cdot \gamma_0\gamma_1^2 \mathbb{H}^{(2)}(0)\\
	& = \sigma\cdot \sqrt{ h^\ast_{+} + h^\ast_{-}}\cdot \mathbb{H}^{(2)}(0),
	\end{align*}
	and
	\begin{align*}
	\sqrt{(h^\ast_+(f_0)+ h^\ast_-(f_0))^3}\cdot \mathbb{H}^{(3)}(0; f_0) & = \sqrt{\big( h^\ast_{+} + h^\ast_{-}\big)^3/\gamma_1^3}\cdot \gamma_0\gamma_1^3 \mathbb{H}^{(3)}(0)\\
	& = \sigma\cdot \sqrt{\big( h^\ast_{+} + h^\ast_{-}\big)^3}\cdot \mathbb{H}^{(3)}(0),
	\end{align*}
	as desired. 
\end{proof}

\subsection{Proof of Theorem \ref{thm:CI_fcn}}\label{pf:CI_fcn}
	We only prove the second claim. It follows from the rescaling argument in the end of the proof of Theorem \ref{thm:pivotal_limit_fcn} that
	\begin{align*}
	n^{2/5} \abs{\mathcal{I}_n^{(0)}(c_\delta^{(0)}) } 
	&= \textstyle 2\widehat{\sigma} c_\delta^{(0)} \big/ \sqrt{n^{1/5} (\widehat{v}(x_0)-\widehat{u}(x_0)) }
	\\ \nonumber
	&= \textstyle 2\widehat{\sigma} c_\delta^{(0)} \big/\sqrt{\widehat{h}_++\widehat{h}_-} 
	\\ \nonumber
	&\rightsquigarrow \textstyle 2 \sigma c_{\delta}^{(0)} \big/ \sqrt{h^\ast_+(f_0) + h^\ast_-(f_0)} 
	\\ \nonumber
	&=_d \textstyle \big(2 c_{\delta}^{(0)} / \sqrt{ h^\ast_{+} + h^\ast_{-} } \big) \cdot  \gamma_1^{1/2}\sigma 
	\\ \nonumber
	& = \textstyle \big(2 c_{\delta}^{(0)} / \sqrt{ h^\ast_{+} + h^\ast_{-} } \big) \cdot  \gamma_0\gamma_1^2 \nonumber \\
	&= \textstyle \big(2 c_{\delta}^{(0)} / \sqrt{ h^\ast_{+} + h^\ast_{-} } \big) \cdot \sigma^{4/5} d_2^{(0)}(f_0, x_0).
	\end{align*}
	Similarly, $ n^{1/5}\abs{\mathcal{I}_n^{(1)}(c_\delta^{(1)}) }\rightsquigarrow \big(2 c_{\delta}^{(1)} / \sqrt{(h^\ast_{+} + h^\ast_{-})^3 } \big) \cdot \sigma^{2/5} d_2^{(1)}(f_0, x_0)$. \qed

\subsection{Proof of Theorem \ref{thm:pivotal_limit_mode}} \label{pf:pivotal_limit_mode}

	The high level idea of the proof is to mimic the arguments in the proof of Theorem \ref{thm:pivotal_limit_fcn} by processes centered at the (random) mode. This causes some technical complications as detailed below.

	We continue to consider the processes in the proof sketch of Theorem \ref{thm:limiting_dist_convex} but at anti-mode $x_0 \equiv m_0$, so that
	\begin{align*}
	(\widetilde{\mathbb{H}}_n^{\mathrm{loc}})^{(2)}(t)& = n^{2/5} \big\{ \widehat{f}_n( m_0+n^{-1/5}t)-f_0( m_0)-n^{-1/5}f_0'( m_0)t \big\}
	\cr
	&= n^{2/5} \big\{ \widehat{f}_n( m_0+n^{-1/5}t)-f_0( m_0) \big\}.
	\end{align*}
	Let $$ m_{f_0} \equiv \big[ \mathbb{H}^{(2)}(\cdot\,; f_0) \big]_{ \mathrm{m} } \qquad \mbox{and}\qquad \widetilde{m}_n^{\mathrm{loc}, (2)} \equiv \big[(\widetilde{\mathbb{H}}_n^{\mathrm{loc}})^{(2)}(\cdot)\big]_{\mathrm{m}} ,$$ so that $\widetilde{m}_n^{\mathrm{loc}, (2)} = n^{1/5}(\widehat{m}_n-  m_0)$.  By similar arguments as in~\cite[pp.~1327]{balabdaoui2009limit} for the mode of the MLE of a log-concave density, we have
	\begin{align*}
	\widetilde{m}_n^{\mathrm{loc}, (2)} = n^{1/5}\big(\widehat{m}_n- m_0\big)= \big[(\widetilde{\mathbb{H}}_n^{\mathrm{loc}})^{(2)}(\cdot)\big]_{\mathrm{m}} \rightsquigarrow \big[ \mathbb{H}^{(2)}(\cdot\,; f_0) \big]_{ \mathrm{m} } =  m_{f_0}.
	\end{align*}
	For notational convenience, let $\widehat{h}_{ \mathrm{m} ;+} \equiv n^{1/5}(\widehat{v}_{\mathrm{m}}-\widehat{m}_n)$ and $\widehat{h}_{ \mathrm{m} ;-}\equiv n^{1/5}(\widehat{m}_n- \widehat{u}_{\mathrm{m}})$. Then for any $ \mathfrak{w} $, similar to (\ref{ineq:pivotal_lim_fcn_1}), due to the convexity of $\widehat{f}_n$,
	\begin{align}\label{ineq:proof_mode_reg_1}
	\widehat{h}_{ \mathrm{m} ;\pm} <  \mathfrak{w}  
	&\Leftrightarrow 2\widehat{f}_n \big(\widehat{m}_n\pm n^{-1/5} \mathfrak{w} /2 \big)- \widehat{f}_n(\widehat{m}_n\pm n^{-1/5} \mathfrak{w} ) - \widehat{f}_n(\widehat{m}_n)<0, \nonumber \\
	&\Leftrightarrow  \widetilde{\Delta}_{n, \pm}^{\mathrm{loc}}(\mathfrak{w}) \equiv 2(\widetilde{\mathbb{H}}_n^{\mathrm{loc}})^{(2)}\big(  \widetilde{m}_n^{\mathrm{loc}, (2)}\pm \mathfrak{w} /2\big) - (\widetilde{\mathbb{H}}_n^{\mathrm{loc}})^{(2)}( \widetilde{m}_n^{\mathrm{loc}, (2)}\pm \mathfrak{w} ) \nonumber\\
	&\qquad\qquad \qquad \qquad - (\widetilde{\mathbb{H}}_n^{\mathrm{loc}})^{(2)}\big( \widetilde{m}_n^{\mathrm{loc}, (2)}\big) <0, 
	\end{align}
	and similar to (\ref{ineq:pivotal_lim_fcn_2}), on the event $E_{\mathrm{m}; n,\pm}(\epsilon)\equiv \{\widehat{h}_{ \mathrm{m} ;\pm}\geq \epsilon\}$, 
	\begin{align}\label{ineq:proof_mode_reg_2}
	&\widehat{h}_{ \mathrm{m} ;\pm} \leq   \mathfrak{w}  \\
	&\quad  \Leftrightarrow  \big\{\mathbb{H}_n^{\mathrm{loc}} \big(  \widetilde{m}_n^{\mathrm{loc}, (2)}+ \cdot \big) ,\mathbb{Y}_n^{\mathrm{loc}}\big(  \widetilde{m}_n^{\mathrm{loc}, (2)}+ \cdot \big)\big\}|_{\pm[0,\mathfrak{w}]} \in S_{\pm}(\mathfrak{w};\epsilon),\nonumber
	\end{align}
	where $S_{\pm}(\mathfrak{w};\epsilon)\equiv \big\{(h,y) \in \big(C(\pm[0,\mathfrak{w} ])\big)^2,  \sup_{t \in \pm[\epsilon,\mathfrak{w}]}\{y(t)-h(t)\}=0 \big\}$.
	
	We first show for any $T>0$,
	\begin{align}\label{ineq:pivotal_limit_mode_1}
	&(\widetilde{\mathbb{H}}_n^{\mathrm{loc}})^{(2)}\big(  \widetilde{m}_n^{\mathrm{loc}, (2)}+ \cdot \big)
	\rightsquigarrow \mathbb{H}^{(2)}\big(  m_{f_0}+ \cdot \,; f_0\big) \hbox{ in $C([-T,T])$}; 
	\\ \nonumber
	& \big( \mathbb{H}_n^{\mathrm{loc}} (  \widetilde{m}_n^{\mathrm{loc}, (2)}+ \cdot ) ,\mathbb{Y}_n^{\mathrm{loc}} (  \widetilde{m}_n^{\mathrm{loc}, (2)}+ \cdot ) \big) 
	\\ \nonumber
	& \hspace{4em} \rightsquigarrow \big( \mathbb{H} (  m_{f_0}+ \cdot;f_0) ,\mathbb{Y} (   m_{f_0}+ \cdot;f_0 ) \big) \hbox{ in $\big(C([-T,T])\big)^2$}. 
	\end{align}
	We only prove the first claim in (\ref{ineq:pivotal_limit_mode_1}); the second one is analogous. The main challenge to show the first claim of (\ref{ineq:pivotal_limit_mode_1}) is the fact that the process $(\widetilde{\mathbb{H}}_n^{\mathrm{loc}})^{(2)}(\cdot)$ is centered at the random point $\widetilde{m}_n^{\mathrm{loc}, (2)}$, which is different from the random center $ m_{f_0}$ of the limit process $\mathbb{H}^{(2)}(\cdot;f_0)$. 
	
	To this end, let $(\Omega_n,\mathcal{A}_n,P_n)$ be the probability space on which the process $(\widetilde{\mathbb{H}}_n^{\mathrm{loc}})^{(2)}$ is defined, and $(\Omega_\infty, \mathcal{A}_\infty, P_\infty)$ be the one for $\mathbb{H}^{(2)}(\cdot\,; f_0)$. By the uniform tightness of $ \widetilde{m}_n^{\mathrm{loc}, (2)}$ and $  m_{f_0}$, for any $\epsilon>0$, there exists some $K \equiv K(\epsilon)$ such that $\abs{ \widetilde{m}_n^{\mathrm{loc}, (2)}} \leq K/2$ holds on $E_{n,\epsilon} \subset \Omega_n$ and $ \abs{ m_{f_0}}\leq K/2$ holds on $E_{\infty,\epsilon}\subset \Omega_\infty$, with $P_n(E_{n,\epsilon})\wedge P_\infty(E_{\infty,\epsilon})\geq 1-\epsilon$. Note that $(\widetilde{\mathbb{H}}_n^{\mathrm{loc}})^{(2)}(\cdot) \rightsquigarrow \mathbb{H}^{(2)}(\cdot\,; f_0)$ in $C([-(T+K),(T+K)])$. By Skorokhod's representation theorem (see e.g.,~\cite[Theorem 6.7]{bilingsley1999convergence}), there exists another probability space $(\widetilde{\Omega},\widetilde{\mathcal{A}},\widetilde{P})$ and measurable maps $\phi_{n}: \widetilde{\Omega}\to \Omega_n$ with $P_{n} = \widetilde{P}\circ \phi_{n}^{-1}$ ($n\leq \infty$) such that with $\widetilde{H}_{n}^{(2)}\equiv (\widetilde{\mathbb{H}}_n^{\mathrm{loc}})^{(2)}\circ \phi_{n} $ and $\widetilde{H}^{(2)}\equiv \mathbb{H}^{(2)}(\cdot\,;f_0) \circ \phi_{\infty}$, the processes $\{\widetilde{H}_{n}^{(2)}\}, \widetilde{H}^{(2)}$ are all defined on $(\widetilde{\Omega}, \widetilde{\mathcal{A}}, \widetilde{P})$, and
	\begin{align*}
	\sup_{t \in [-(T+K),(T+K)]}\abs{\widetilde{H}_{n}^{(2)}(t)-\widetilde{H}^{(2)}(t)}\to 0
	\end{align*}
	on a $\widetilde{P}$-probability 1 event $\widetilde{E}_0$. Let $\widetilde{E}_1$ be the event on which the piecewise linear convex function $\widetilde{H}^{(2)}$ has a unique minimizer. By Lemma \ref{lem:uniqueness_H_2_minimzer}, $\widetilde{P}(\widetilde{E}_1) = 1$.
	
	Now we are ready to prove the first claim of (\ref{ineq:pivotal_limit_mode_1}) on the `good event' $\widetilde{E}_\epsilon\equiv \phi_n^{-1}(E_{n,\epsilon})\cap \phi_\infty^{-1}(E_{\infty,\epsilon})\cap \widetilde{E}_0 \cap \widetilde{E}_1$:
	\begin{align*}
	&\sup_{t \in [-T,T]} \big|\widetilde{H}_{n}^{(2)}( \big[ \widetilde{H}_n^{(2)}\big]_{ \mathrm{m} }  +t)-\widetilde{H}^{(2)}(\big[ \widetilde{H}^{(2)}\big]_{ \mathrm{m} } +t) \big|\\
	&\leq \sup_{t \in [-(T+K),(T+K)]} \abs{ \widetilde{H}_{n}^{(2)}(t)-\widetilde{H}^{(2)}(t) }\\
	&\qquad\qquad +\sup_{t \in [-T,T]} \big|\widetilde{H}^{(2)}( \big[ \widetilde{H}_n^{(2)}\big]_{ \mathrm{m} }  +t)-\widetilde{H}^{(2)}(\big[ \widetilde{H}^{(2)}\big]_{ \mathrm{m} } +t) \big|\to 0.
	\end{align*}
	The second term vanishes by the uniform continuity of $\widetilde{H}^{(2)}(\cdot)$ over compact sets and the fact that $\big[ \widetilde{H}_n^{(2)}\big]_{ \mathrm{m} }  \to \big[\widetilde{H}^{(2)}\big]_{ \mathrm{m} } $ on $\widetilde{E}_1$. 
	
    Putting the pieces together, for any bounded and Lipschitz function $\mathfrak{H}$ on $C([-T,T])$,
	\begin{align*}
	&\biggabs{ \E \mathfrak{H}\Big[(\widetilde{\mathbb{H}}_n^{\mathrm{loc}})^{(2)}\big(  \widetilde{m}_n^{\mathrm{loc}, (2)}+\cdot \big)\Big] -\E \mathfrak{H}\Big[\mathbb{H}^{(2)}\big(  m_{f_0}+\cdot \, ; f_0\big)\Big] }
	\\
	&\leq \biggabs{\widetilde{E} \Big\{\mathfrak{H}\big[\widetilde{H}_{n}^{(2)}\Big( \Big[ \widetilde{H}_n^{(2)}\big]_{ \mathrm{m} }  +\cdot\Big)\Big] -\mathfrak{H} \Big[\widetilde{H}^{(2)}\Big(\big[ \widetilde{H}^{(2)}\big]_{ \mathrm{m} } +\cdot\Big)\Big]\Big\}\bm{1}_{\widetilde{E}_\epsilon}  }+2 \pnorm{\mathfrak{H}}{\infty}\widetilde{P}(\widetilde{E}_\epsilon^c)
	\\
    &\leq \widetilde{E} \bigg\{2\pnorm{\mathfrak{H}}{\infty}\bigwedge  \bigg[\pnorm{\mathfrak{H}}{\mathrm{Lip}}\sup_{t \in [-T,T]} \biggabs{\widetilde{H}_{n}^{(2)}\Big( \big[ \widetilde{H}_n^{(2)}\big]_{ \mathrm{m} }  + t\Big)-\widetilde{H}^{(2)}\Big(\big[ \widetilde{H}^{(2)}\big]_{ \mathrm{m} } +t\Big)} \bm{1}_{\widetilde{E}_\epsilon} \bigg]\bigg\} \\
	&\qquad\qquad +4\pnorm{\mathfrak{H}}{\infty}\epsilon.
	\end{align*}
	where in the last inequality we used $\widetilde{P}(\widetilde{E}_\epsilon^c)\leq 2\epsilon$. Hence with $\mathrm{BL}_1(C([-T,T]))\equiv \{\mathfrak{H}: C([-T,T]) \to \R, \pnorm{\mathfrak{H}}{\infty}\vee \pnorm{\mathfrak{H}}{\mathrm{Lip}}\leq 1\}$, we have
	\begin{align*}
	\sup_{\mathfrak{H} \in \mathrm{BL}_1(C([-T,T])) }\biggabs{ \E \mathfrak{H} \Big[(\widetilde{\mathbb{H}}_n^{\mathrm{loc}})^{(2)}\big(  \widetilde{m}_n^{\mathrm{loc}, (2)}+\cdot \big)\Big] -\E \mathfrak{H}\Big[\mathbb{H}^{(2)}\big(  m_{f_0}+\cdot\,;\, f_0 \big)\Big]   }\to 0
	\end{align*}
	by first taking supremum over $\mathfrak{H} \in \mathrm{BL}_1(C([-T,T]))$, and then letting $n \to \infty$ followed by $\epsilon \downarrow 0$ in the previous display. This shows (\ref{ineq:pivotal_limit_mode_1}). Using again Skorokhod's representation theorem, we conclude the weak convergence of
	\begin{align*}
	\bigg( \widetilde{m}_n^{\mathrm{loc}, (2)},\,\widetilde{\Delta}_{n, \pm}^{\mathrm{loc}}(\mathfrak{w} - \varepsilon),\,\big\{\mathbb{H}_n^{\mathrm{loc}} \big(  \widetilde{m}_n^{\mathrm{loc}, (2)}+ \cdot \big) ,\mathbb{Y}_n^{\mathrm{loc}}\big(  \widetilde{m}_n^{\mathrm{loc}, (2)}+ \cdot \big)\big\}|_{\pm[0,\mathfrak{w}]} \bigg)
	\end{align*}
	in $\R^2\times \big(C(\pm[0,\mathfrak{w}])\big)^2$. Using the equivalence (\ref{ineq:proof_mode_reg_1})-(\ref{ineq:proof_mode_reg_2}) and similar arguments as in the proof of Theorem \ref{thm:pivotal_limit_fcn} along with $\lim_{\epsilon \downarrow 0}\limsup_{n \uparrow \infty} \Prob(E_{\mathrm{m};n,\pm}(\epsilon))=0$ where $E_{\mathrm{m};n,\pm}(\epsilon)$ is defined before (\ref{ineq:proof_mode_reg_2}), we conclude that
	\begin{align*}
	\big( \widetilde{m}_n^{\mathrm{loc}, (2)}, \widehat{h}_{ \mathrm{m} ;-},\widehat{h}_{ \mathrm{m} ;+}\big) \rightsquigarrow \big( m_{f_0},  h^\ast_{\mathrm{m};-}(f_0), h^\ast_{\mathrm{m};+}(f_0) \big),
	\end{align*}
	where  $ h^\ast_{\mathrm{m};-}(f_0)$ (resp.~$h^\ast_{\mathrm{m};+}(f_0) $) is the first kink of $\mathbb{H}^{(2)}(\cdot\,;f_0)$ to the left (resp.~right) of its anti-mode $ m_{f_0}$. As $h^\ast_{\mathrm{m};+}(f_0) + h^\ast_{\mathrm{m};-}(f_0) \in (0,\infty)$ a.s., by continuous mapping we have
	\begin{align*}
	\frac{\widehat{m}_n - m_0}{\widehat{v}_{\mathrm{m}}- \widehat{u}_{\mathrm{m}}} = \frac{n^{1/5}\big(\widehat{m}_n- m_0\big)}{\widehat{h}_{ \mathrm{m} ;+}+\widehat{h}_{ \mathrm{m} ;-}} \rightsquigarrow \frac{m_{f_0} }{h^\ast_{\mathrm{m};+}(f_0) + h^\ast_{\mathrm{m};-}(f_0) }.
	\end{align*}
	Now we check the distribution of the random variable in the far right hand side of the above display is pivotal with respect to the nuisance parameters. This follows from the arguments in the proof of Theorem \ref{thm:pivotal_limit_fcn}: Using the same notation therein, we have that 
	\begin{align*}
	\frac{m_{f_0}}{h^\ast_{\mathrm{m};+}(f_0) + h^\ast_{\mathrm{m};-}(f_0) }& = \frac{ \big[ \gamma_0\gamma_1^2\mathbb{H}^{(2)}(\gamma_1\cdot) \big]_{ \mathrm{m} }  }{ \big( h^\ast_{\mathrm{m};+} + h^\ast_{\mathrm{m};-} \big)/\gamma_1} = \frac{  \big[ \mathbb{H}^{(2)} \big]_{ \mathrm{m} }  }{ h^\ast_{\mathrm{m};+} + h^\ast_{\mathrm{m};-} },
	\end{align*}
	as desired. \qed

\begin{lemma}\label{lem:uniqueness_H_2_minimzer}
With probability $1$, the random piecewise linear convex function $\mathbb{H}^{(2)}_2$ defined in Theorem \ref{thm:limiting_dist_convex} has a unique minimizer.
\end{lemma}
\begin{proof}
Let the probability space be $(\Omega,\mathcal{A},\mathbb{P})$. Let $E$ be the event on which the piecewise linear convex function $\mathbb{H}^{(2)}_2$ has a unique minimizer. Note that $E = (S_0(\mathbb{B}))^c$ where, for $b \in \R$ and $K>0$,
\begin{align*}
S_{b,K}(\mathbb{B}) \equiv \{\omega \in \Omega: b \in \{\mathbb{H}^{(3)}_2(t,\omega):t \in [-K,K] \}\},\, S_{b}(\mathbb{B}) \equiv S_{b,\infty}(\mathbb{B});
\end{align*}
the dependence on $\mathbb{B}$, the two-sided Brownian motion, is emphasized as $\mathbb{H}^{(3)}_2$ depends on $\mathbb{B}$. Let  $T_{b,K}(\mathbb{B})\equiv \cup_{b'\leq b} S_{b',K}(\mathbb{B})$ and $T_{b}(\mathbb{B})\equiv T_{b,\infty}(\mathbb{B})$, so $S_b(\mathbb{B})\subset \cap_{b'< b}\{T_b(\mathbb{B})\setminus T_{b'}(\mathbb{B})\}$, and hence $\Prob(S_b(\mathbb{B}))\leq \lim_{b' \uparrow b}\{\Prob(T_b(\mathbb{B}))-\Prob(T_{b'}(\mathbb{B}))\}$. As the pair $(\mathbb{H}_2(t)-bt^3/6,  \mathbb{Y}_2(t)-bt^3/6)$ satisfies the characterization conditions of the envelope process in \cite{groeneboom2001canonical}, it is determined by the drifted Brownian motion $t\mapsto \mathbb{B}(t)-bt^2/2$.  Let $E_{K}$ be the event that $\mathbb{H}^{(2)}_2$ attains its minimum inside $[-K,K]$. By localization, for any $\epsilon>0$, there exists $K_\epsilon>0$ such that $\Prob(E_{K_\epsilon})\geq 1-\epsilon$. By Cameron-Martin formula,
\begin{align*}
\Prob(T_b(\mathbb{B})) &	\leq \Prob \bigg( \Big\{\omega \in \Omega: \exists b'\leq 0, b' \in \{\mathbb{H}^{(3)}_2(t,\omega)-b:t \in \R \} \Big\}\cap E_{K_\epsilon} \bigg)+\epsilon\\
&\leq  \Prob \big( \omega \in \Omega: \exists b'\leq 0, b' \in \{\mathbb{H}^{(3)}_2(t,\omega)-b:t \in [-K_\epsilon,K_\epsilon] \}  \big)+\epsilon \\
&=\Prob \big(T_{0,K_\epsilon}(\mathbb{B}(\cdot)-b(\cdot)^2/2)\big)+\epsilon\\
&=\E \big(\bm{1}_{T_{0,K_\epsilon}(\mathbb{B})}\cdot e^{-b\int_{-K_\epsilon}^{K_\epsilon}t\,\d{\mathbb{B}(t)}- b^2\int_{-K_\epsilon}^{K_\epsilon} t^2\,\d{t}/2 }\big)+\epsilon\equiv \Gamma_\epsilon(b)+\epsilon.
\end{align*}
On the other hand, $\Prob(T_b(\mathbb{B}))\geq \Prob \big(T_{0,K_\epsilon}(\mathbb{B}(\cdot)-b(\cdot)^2/2)\big)\geq \Gamma_\epsilon(b)$ by simply restricting the process to $[-K_\epsilon,K_\epsilon]$. Hence we have $\abs{\Prob(T_b(\mathbb{B}))-\Gamma_\epsilon(b)}\leq \epsilon$. It is easy to check that $\Gamma_\epsilon(b)$ is continuous in $b$, so $
\Prob(S_b(\mathbb{B}))\leq \lim_{b' \uparrow b}\{\Prob(T_b(\mathbb{B}))-\Prob(T_{b'}(\mathbb{B}))\}\leq  2\epsilon$. Letting $\epsilon \downarrow 0$ yields that $\Prob(S_b(\mathbb{B}))=0$ for any $b \in \R$, which proves the claim of the lemma as $\Prob(E)=\Prob((S_0(\mathbb{B}))^c)=1$.
\end{proof}

\subsection{Proof of Theorem \ref{thm:CI_mode}}\label{pf:CI_mode}
	It follows from the rescaling argument in the end of the proof of Theorem \ref{thm:pivotal_limit_mode} that
	\begin{align*}
	n^{1/5}\abs{\mathcal{I}_n^{\mathrm{m}} (c_\delta^{(0)}) } &= 2c_\delta^{\mathrm{m}} n^{1/5} (\widehat{v}_{\mathrm{m}}- \widehat{u}_{\mathrm{m}})
	\\ \nonumber
	&\rightsquigarrow 2c_\delta^{\mathrm{m}} \big( h^\ast_{\mathrm{m};+}(f_0) + h^\ast_{\mathrm{m};-}(f_0) \big)
	\\ \nonumber
	&=_d 2c_\delta^{\mathrm{m}} \big(h^\ast_{\mathrm{m};+} + h^\ast_{\mathrm{m};-}\big) \cdot \gamma_1^{-1}
	\\ \nonumber
	&= 2c_\delta^{\mathrm{m}} \big(h^\ast_{\mathrm{m};+} + h^\ast_{\mathrm{m};-}\big) \cdot \sigma^{2/5} d_2^{\mathrm{m}}(f_0),
	\end{align*}
	as desired. \qed

\section{Proof of results in Section \ref{section:other_convex}}

\subsection{Proof of Theorems \ref{thm:pivotal_limit_log_concave} and \ref{thm:pivotal_limit_mode_log_concave}}\label{pf:pivotal_limit_log_concave}

	The log-concave MLE $\widehat{f}_n=\exp(\widehat{\varphi}_n)$ can be characterized as follows using the notation of \cite{balabdaoui2009limit}. For a concave function $\varphi$, let
	\begin{align*}
	\mathbb{Y}_n(t)&\equiv \int_{-\infty}^{t} \mathbb{F}_n(s)\,\d{s},\qquad \mathbb{H}_n(t;\varphi ) \equiv \int_{-\infty}^{t} F_n(u;\varphi)\,\d{u},
	\end{align*}
	where $F_n(t;\varphi)\equiv \int_{-\infty}^{x} e^{\varphi (t)}\,\d{t}$. We write $\mathbb{H}_n(t)\equiv \mathbb{H}_n(t;\widehat{\varphi}_n)$ and $ \widehat{F}_n(t) \equiv F_n(t,\widehat{\varphi}_n)$ 
	for notational simplicity. Then $f=\exp(\varphi)$ with a piecewise linear concave $\varphi$ is the log-concave MLE if and only if $\mathbb{H}_n(t;\varphi)\leq \mathbb{Y}_n(t)$ with equality taken at kinks of $\varphi$ including the boundary points. In other words, $\mathbb{H}_n(t)\leq \mathbb{Y}_n(t)$ with equality taken at kinks of $\widehat{\varphi}_n$ including $X_{(1)}$ and $X_{(n)}$. The direction of the inequality is reversed as we work with concave rather than convex underlying functions. Define the local processes $\mathbb{Y}_n^{\mathrm{loc}},\mathbb{H}_n^{\mathrm{loc}}$ by
	\begin{align}\label{ineq:log_concave_1}
	\mathbb{Y}_n^{\mathrm{loc}}(t) &\equiv n^{4/5}\int_{x_0}^{x_0+n^{-1/5}t} \bigg[\mathbb{F}_n(v)-\mathbb{F}_n(x_0) \\
	&\qquad\qquad - \int_{x_0}^v \big(f_0(x_0)+(u-x_0)f_0'(x_0)\big) \,\d{u} \bigg] \,\d{v},\nonumber\\
	\mathbb{H}_n^{\mathrm{loc}}(t) &\equiv n^{4/5}\int_{x_0}^{x_0+n^{-1/5}t} \bigg[ \widehat{f}_n(v)-\widehat{f}_n(x_0) \nonumber\\
	&\qquad\qquad - \int_{x_0}^v \big(f_0(x_0)+(u-x_0)f_0'(x_0)\big) \,\d{u} \bigg]\,\d{v}+A_n +B_nt, \nonumber
	\end{align}
	where $A_n\equiv n^{4/5}( \mathbb{H}_n(x_0)-\mathbb{Y}_n(x_0))$ and $B_n \equiv n^{3/5}\big( \widehat{F}_n(x_0)-\mathbb{F}_n(x_0)\big)$ so that $\mathbb{H}_n^{\mathrm{loc}}(t)\leq \mathbb{Y}_n^{\mathrm{loc}}(t)$. As we wish to explore the underlying concavity of $\varphi_0$, we further define
	\begin{align}\label{ineq:log_concave_2}
	\mathbb{Y}_n^{\mathrm{locmod}}(t) &\equiv \frac{\mathbb{Y}_n^{\mathrm{loc}}(t)}{f_0(x_0)}-n^{4/5}\int_{x_0}^{x_0+n^{-1/5}t} \int_{x_0}^v \Psi_{n, f}(u)\,\d{u}\d{v},\\
	\mathbb{H}_n^{\mathrm{locmod}}(t) &\equiv \frac{\mathbb{H}_n^{\mathrm{loc}}(t)}{f_0(x_0)}-n^{4/5}\int_{x_0}^{x_0+n^{-1/5}t} \int_{x_0}^v \Psi_{n, f}(u)\,\d{u}\d{v} \nonumber\\
	& = n^{4/5}\int_{x_0}^{x_0+n^{-1/5} t} \int_{x_0}^v \Psi_{n, \varphi}(u) \,\d{u}\d{v}+A_n+B_nt,\nonumber
	\end{align}
	where
	\begin{align*}
	\Psi_{n, \varphi}(u) &\equiv \widehat{\varphi}_n(u)-\varphi_0(x_0)-(u-x_0)\varphi_0'(x_0),
	\cr
	\Psi_{n, f}(u)&\equiv \frac{1}{f_0(x_0)} \big(\widehat{f}_n(u)-f_0(x_0)-(u-x_0)f_0'(x_0)\big) -\Psi_{n, \varphi}(u) \\
	&= \sum_{\ell=2}^{\infty} \frac{1}{\ell!} \big(\widehat{\varphi}_n(u)-\varphi_0(x_0)\big)^\ell.
	\end{align*}
	Clearly we still have $\mathbb{H}_n^{\mathrm{locmod}}(t)\leq \mathbb{Y}_n^{\mathrm{locmod}}(t)$ with equality taken at kinks of $\widehat{\varphi}_n$ including $X_{(1)}$ and $X_{(n)}$, and
	\begin{align*}
	n^{2/5}\big(\widehat{\varphi}_n(x_0)-\varphi_0(x_0)\big)  &= (\mathbb{H}_n^{\mathrm{locmod}})^{(2)}(0),\\
	n^{1/5}\big(\widehat{\varphi}_n'(x_0)-\varphi_0'(x_0)\big)  &= (\mathbb{H}_n^{\mathrm{locmod}})^{(3)}(0).
	\end{align*}
	Now following a similar technique as  before, we only need to compute the limit of $\mathbb{Y}_n^{\mathrm{locmod}}$ and rescale the process. First note that, after localization (see~\cite[Lemma 4.5]{balabdaoui2009limit}), 
	\begin{align*}
	&n^{4/5}\int_{x_0}^{x_0+n^{-1/5}t} \int_{x_0}^v \Psi_{n, f}(u)\,\d{u}\d{v} \\
	&= \mathfrak{o}_{\mathbf{P}}(1)+n^{4/5}  \int_{x_0}^{x_0+n^{-1/5}t} \int_{x_0}^v \frac{1}{2} \big(\widehat{\varphi}_n(u)-\varphi_0(x_0)\big)^2\,\d{u}\d{v}\\
	& = \mathfrak{o}_{\mathbf{P}}(1)+n^{4/5}  \int_{x_0}^{x_0+n^{-1/5}t} \int_{x_0}^v \frac{1}{2} (u-x_0)^2(\varphi_0'(x_0))^2\,\d{u}\d{v} 
	\cr
	& = \mathfrak{o}_{\mathbf{P}}(1) + \frac{ (\varphi_0'(x_0))^2}{4!} t^4. 
	\end{align*}
	The $\mathfrak{o}_{\mathbf{P}}(1)$ is uniform for $t$ on compacta. Also a standard argument shows that
	\begin{align*}
	\mathbb{Y}_n^{\mathrm{loc}}(t)\rightsquigarrow \sqrt{f_0(x_0)} \int_0^t \mathbb{B}(s)\,\d{s} + \frac{f_0''(x_0)}{4!} t^4,\quad \textrm{ in }C([-K,K])
	\end{align*}
	for any $K>0$. This means that
	\begin{align*}
	\mathbb{Y}_n^{\mathrm{locmod}}(t)&\rightsquigarrow (f_0(x_0))^{-1/2} \int_0^t \mathbb{B}(s)\,\d{s} + \frac{t^4}{4!}\bigg(\frac{f_0''(x_0)}{f_0(x_0)}-(\varphi_0'(x_0))^2\bigg)\\
	& = (f_0(x_0))^{-1/2} \int_0^t \mathbb{B}(s)\,\d{s} + \frac{\varphi_0''(x_0)}{4!} t^4\equiv -\mathbb{Y}(t; f_0), \quad \textrm{ in }C([-K,K])
	\end{align*}
	for any $K>0$. Let $\mathbb{H}$ be the a.s.~uniquely determined piecewise cubic function that majorizes $\mathbb{Y}$ with touch points only at jumps of $\mathbb{H}^{(3)}$. Now we choose the scaling factors $\gamma_0,\gamma_1$ by
	\begin{align*}
	\gamma_0 \gamma_1^{3/2} = \frac{1}{\sqrt{f_0(x_0)}},\qquad \gamma_0 \gamma_1^4 = \frac{ \abs{\varphi_0''(x_0)} }{4!};
	\end{align*}
	then we have $\gamma_0 \mathbb{Y}(\gamma_1 t)=_d \mathbb{Y}(t; f_0), \mathbb{H}^{(2)}(t; f_0) = \gamma_0 \gamma_1^2 \mathbb{H}^{(2)}(\gamma_1 t), \mathbb{H}^{(3)}(t; f_0) = \gamma_0 \gamma_1^3 \mathbb{H}^{(3)}(\gamma_1 t)$ and $h^\ast_\pm = \gamma_1 h^\ast_\pm(f_0)$. Hence, following similar arguments as in \eqref{eqn:brownian_scaling} from the proof of Theorem \ref{thm:pivotal_limit_fcn}, we have
	\begin{align*}
	&\sqrt{n(\widehat{v}(x_0)-\widehat{u}(x_0))}\big(\widehat{\varphi}_n(x_0)-\varphi_0(x_0)\big)\\
	&\rightsquigarrow - \sqrt{h^\ast_+(f_0) + h^\ast_-(f_0)}\cdot \mathbb{H}^{(2)}(0; f_0)
	\cr
	& =_d - \sqrt{ \big( h^\ast_{+} + h^\ast_{-}\big)/\gamma_1  }\cdot \gamma_0\gamma_1^2\mathbb{H}^{(2)}(0) = -(f_0(x_0))^{-1/2} \cdot \mathbb{L}^{(0)}.
	\end{align*}
	Similarly
	\begin{align*}
	\sqrt{n(\widehat{v}(x_0)-\widehat{u}(x_0))^3 }\big(\widehat{\varphi}_n'(x_0)-\varphi_0'(x_0)\big) \rightsquigarrow -(f_0(x_0))^{-1/2} \cdot \mathbb{L}^{(1)}. 
	\end{align*}
	The claims for $\widehat{f}_n,\widehat{f}_n'$ follow from the delta method. For the mode, similar to the proof of Theorem \ref{thm:pivotal_limit_mode} and using notation therein, we have
	\begin{align*}
	&\frac{1}{\widehat{v}(\widehat{m}_n)-\widehat{u}(\widehat{m}_n)}\big(\widehat{m}_n- m_0\big)\\
	&\rightsquigarrow\frac{  m_{f_0} }{h^\ast_{\mathrm{m};+}(f_0) + h^\ast_{\mathrm{m};-}(f_0) } = \frac{ \big[ \gamma_0\gamma_1^2\mathbb{H}^{(2)}(\gamma_1\cdot) \big]_{ \mathrm{m} }  }{ \big( h^\ast_{\mathrm{m};+} + h^\ast_{\mathrm{m};-} \big)/\gamma_1} = \frac{ \big[ \mathbb{H}^{(2)} \big]_{ \mathrm{m} }  }{ h^\ast_{\mathrm{m};+} + h^\ast_{\mathrm{m};-} }.
	\end{align*}
	The rest of the claims follow as in the proof of Theorem \ref{thm:CI_fcn}. \qed

\subsection{Proof of Theorems \ref{thm:pivotal_limit_s_concave} and \ref{thm:pivotal_limit_mode_s_concave}}\label{pf:pivotal_limit_s_concave}

	The proof is similar to that of Theorems \ref{thm:pivotal_limit_log_concave} and \ref{thm:pivotal_limit_mode_log_concave} so we only give a sketch here. Let $\mathbb{Y}_{n,s}^{\mathrm{loc}}, \mathbb{H}_{n,s}^{\mathrm{loc}}$ be defined similarly as (\ref{ineq:log_concave_1}) by replacing $\widehat{\varphi}_n$ with $\widehat{\varphi}_{n,s}$, and $\mathbb{Y}_{n,s}^{\mathrm{locmod}}, \mathbb{H}_{n,s}^{\mathrm{locmod}}$ be defined similarly as (\ref{ineq:log_concave_2}) by replacing $\mathbb{Y}_{n}^{\mathrm{loc}}, \mathbb{H}_{n}^{\mathrm{loc}}$ with $\mathbb{Y}_{n,s}^{\mathrm{loc}}$, $\mathbb{H}_{n,s}^{\mathrm{loc}}$, and $\Psi_{n, \varphi}$, $\Psi_{n, f}$ with
	\begin{align*}
	\Psi_{n, \varphi, s}(u) &\equiv \widehat{\varphi}_{n, s}(u)-\varphi_s(x_0)-(u-x_0)\varphi_s'(x_0),
	\cr
	\Psi_{n,f, s}(u) &\equiv \frac{1}{f_0(x_0)} \big(\widehat{f}_n(u)-f_0(x_0)-(u-x_0)f_0'(x_0)\big) - \frac{(-r_s)}{\varphi_s(x_0)} \Psi_{n, \varphi, s}(u).
	\end{align*}
	Then in the proof of \cite[Theorem 6.4]{han2015approximation}, it is shown that
	\begin{align*}
	\mathbb{Y}_{n,s}^{\mathrm{locmod}}(t)&\rightsquigarrow (f_0(x_0))^{-1/2} \int_0^t \mathbb{B}(s)\,\d{s} + \frac{(-r_s)\varphi_s''(x_0)}{\varphi_s(x_0) 4!}t^4,\quad \textrm{ in }C([-K,K])
	\end{align*}
	for any $K>0$, so the scaling constants $\gamma_0,\gamma_1$ can be chosen as 
	\begin{align*}
	\gamma_0 \gamma_1^{3/2} = \frac{1}{\sqrt{f_0(x_0)}},\qquad \gamma_0 \gamma_1^4 = \frac{r_s\varphi_s''(x_0)}{\varphi_s(x_0) 4!}.
	\end{align*}
	The rest of the proof parallels that of Theorem \ref{thm:pivotal_limit_log_concave}. \qed

\subsection{Proof of Proposition \ref{prop:constat_log_s_concave}}\label{pf:constat_log_s_concave}
	It is easy to calculate that 
	\begin{align*}
	\varphi_s''(x_0)=\frac{1}{r_s^2} \varphi_s(x_0)\big(\varphi_0'(x_0)^2- r_s \varphi_0''(x_0)\big).
	\end{align*}
	Then
	\begin{align*}
	d_{2, \mathrm{sc}}^{(0)} (f_0,x_0) &= \bigg(\frac{f_0(x_0)^3 \varphi_0'(x_0)^2}{4!r_s}+\frac{f_0(x_0)^3 \abs{\varphi_0''(x_0)}}{4!}\bigg)^{1/5}> d_{2, \mathrm{lc}}^{(0)} (f_0,x_0),\\
	d_{2, \mathrm{sc}}^{(1)}(f_0,x_0)& = \bigg(\frac{f_0(x_0)^{4/3} \varphi_0'(x_0)^2}{4!r_s}+\frac{f_0(x_0)^{4/3} \abs{\varphi_0''(x_0)}}{4!}\bigg)^{3/5}> d_{2, \mathrm{lc}}^{(1)} (f_0,x_0),\\
	d_{2, \mathrm{sc}}^{\mathrm{m}}(f_0) & = \bigg(\frac{4!}{ \frac{1}{r_s} (f_0( m_0))^{1/2}\varphi_0'( m_0) + (f_0( m_0))^{1/2} \abs{\varphi_0''( m_0) }}\bigg)^{2/5}\\
	&= \bigg(\frac{4!}{(f_0( m_0))^{1/2} \abs{\varphi_0''( m_0) }}\bigg)^{2/5}= d_{2, \mathrm{lc}}^{\mathrm{m}}(f_0),
	\end{align*}
	where in the last equality we have used the fact that $\varphi_0'( m_0)=0$ at the mode. The limit over $s \uparrow 0$ is equivalent to $r_s \uparrow \infty$. This completes the proof. \qed

\subsection{Proof of Theorem \ref{thm:pivotal_limit_density}}\label{pf:pivotal_limit_density}
	We only sketch the proof. The MLE $\widehat{f}_n$ can be characterized as follows. For a convex nonincreasing density $f$, let
	\begin{align*}
	\mathbb{H}_n(t;f)\equiv \int_0^t \frac{t-u}{f(u)}\,\d{\mathbb{F}_n(u)}.
	\end{align*}
	\cite[Lemma 2.4]{groeneboom2001estimation} showed that a piecewise linear convex nonincreasing density $f$ is the MLE if and only if $\mathbb{H}_n(t;f)\leq t^2/2$ with equality taken at kinks of $f$. We write $\mathbb{H}_n(t)\equiv \mathbb{H}_n(t;\widehat{f}_n)$ for simplicity. Then $\mathbb{H}_n(t)\leq t^2/2$ with equality taken at kinks of the MLE $\widehat{f}_n$. 
	
	Define the local processes $\mathbb{Y}_n^{\mathrm{loc}},\mathbb{H}_n^{\mathrm{loc}}$ by 
	\begin{align*}
	\mathbb{Y}_n^{\mathrm{loc}}(t) &\equiv n^{4/5} f_0(x_0) \int_{x_0}^{x_0+ n^{-1/5} t} \int_{x_0}^v \bigg(\frac{f_0(u)-f_0(x_0)-(u-x_0)f_0'(x_0)}{ \widehat{f}_n(u)}\bigg)\,\d{u}\d{v}\\
	&\qquad\qquad + n^{4/5} f_0(x_0) \int_{x_0}^{x_0+n^{-1/5}t} \int_{x_0}^v \frac{1}{\widehat{f}_n(u)}\,\d{(\mathbb{F}_n-F_0)(u)}\d{v},\\
	\mathbb{H}_n^{\mathrm{loc}}(t) &\equiv n^{4/5} f_0(x_0) \int_{x_0}^{x_0+ n^{-1/5} t} \int_{x_0}^v \bigg(\frac{\widehat{f}_n (u)-f_0(x_0)-(u-x_0)f_0'(x_0)}{ \widehat{f}_n(u)}\bigg)\,\d{u}\d{v}\\
	&\qquad\qquad +A_n+B_n t,
	\end{align*}
	where $A_n\equiv -n^{4/5}f_0(x_0)\big(\mathbb{H}_n(x_0)-x_0^2/2\big)$, $B_n\equiv -n^{3/5}f_0(x_0)\big(\mathbb{H}_n'(x_0)-x_0\big)$ and $F_0$ is the true distribution function. Some tedious calculations show that $\mathbb{H}_n^{\mathrm{loc}}(t)\geq \mathbb{Y}_n^{\mathrm{loc}}(t)$ with equality taken where $x_0+n^{-1/5}t$ is a kink of $\widehat{f}_n$, and
	a standard argument yields that
	\begin{align*}
	\mathbb{Y}_n^{\mathrm{loc}}(t)\rightsquigarrow \sqrt{f_0(x_0)} \int_0^t \mathbb{B}(s)\,\d{s} + \frac{f_0''(x_0)}{4!} t^4\equiv \mathbb{Y}(t; f_0),\quad \textrm{ in }C([-K,K])
	\end{align*}
	for any $K>0$. These calculations can be found in the proof of \cite[Theorem 6.2]{groeneboom2001estimation}. Let $\mathbb{H}(\cdot\,; f_0)$ be the a.s.~uniquely determined piecewise cubic function that majorizes $\mathbb{Y}(\cdot\,; f_0)$ with touch points only at jumps of $\mathbb{H}^{(3)}(\cdot; f_0)$. Using similar scaling arguments as in the proof of Theorem \ref{thm:pivotal_limit_log_concave} by choosing $\gamma_0,\gamma_1$ such that
	\begin{align*}
	\gamma_0 \gamma_1^{3/2} = \sqrt{f_0(x_0)},\qquad \gamma_0 \gamma_1^4 = \frac{f_0''(x_0)}{4!},
	\end{align*}
	we may conclude the pivotal limit distribution theory. The rest of the claims follow from the same proof technique as in Theorem \ref{thm:CI_fcn}.
	
	If $\widehat{f}_n$ is the LSE, we re-define the processes as
	\begin{align*}
	\mathbb{Y}_n(t) & \equiv \int_{0}^t \mathbb{F}_n(u)\,\d{u},
	\cr
	\mathbb{H}_n(t) & \equiv \int_{0}^t \int_{0}^v \widehat{f}_n(u)\,\d{u} \d{v},
	\cr
	\mathbb{Y}_n^{\mathrm{loc}}(t) &\equiv n^{4/5} \int_{x_0}^{x_0 + n^{-1/5}t} \Big[ \mathbb{F}_n(v) - \mathbb{F}_n(x_0) - 
	\cr
	&\hspace{12em} \int_{x_0}^v \big( f_0(x_0) + (u - x_0) f_0'(x_0) \big)\,\d{u} \Big]\,\d{v},
	\cr
	\mathbb{H}_n^{\mathrm{loc}}(t) &\equiv n^{4/5} \int_{x_0}^{x_0 + n^{-1/5}t} \int_{x_0}^v \bigg[\widehat{f}_n(u) - f_0(x_0) - (u - x_0) f_0'(x_0) \,\d{u}\d{v}\bigg]
	\cr
	& \hspace{12em} + A_n + B_n t,
	\end{align*}
	where
	\begin{align*}
	A_n = n^{4/5} \big\{ \mathbb{H}_n(x_0) - \mathbb{Y}_n(x_0) \big\} \hbox{ and } B_n = n^{3/5} \Big\{ \int_{0}^{x_0} \widehat{f}_n(u)\,\d{u} - \mathbb{F}_n(x_0) \Big\}.
	\end{align*}
	We omit the rest of the proof as it is almost the same as before for the MLE. \qed

\subsection{Proof of Theorem \ref{thm:pivotal_limit_hazard}}\label{pf:pivotal_limit_hazard}

	Following \cite{jankowski2009nonparametric}, we consider $x_0$ such that $h_0'(x_0)<0$. Then $\widehat{h}_n$ satisfies that
	\begin{align*}
	\int_0^x \int_0^v \frac{1}{\widehat{h}_n(u)}\,\d{\widetilde{ \mathbb{F} }_n(u)}\,\d{v}\leq \int_0^x \int_0^v (1-\mathbb{F}_n(u))\,\d{u}\d{v}
	\end{align*}
	for all $x\geq 0$ with equality taken at kinks of $\widehat{h}_n$ with negative slope (see~\cite[Lemma 2.2]{jankowski2009nonparametric}). Here $\widetilde{\mathbb{F}}_n(t) = n^{-1} \sum_{i=1}^{n-1} \bm{1}_{ \{ 0 \le X_{(i)} \le t \}} $. Define the local processes $\overline{\mathbb{Y}}_n^{\mathrm{loc}},\overline{\mathbb{H}}_n^{\mathrm{loc}}$ by
	\begin{align*}
	&\overline{\mathbb{Y}}_n^{\mathrm{loc}}(t)  \equiv n^{4/5} \frac{h_0(x_0)}{1-F_0(x_0)} \bigg\{\\
	&\qquad  \int_{x_0}^{x_0+n^{-1/5}t} \int_{x_0}^v \bigg[\frac{h_0(u)-h_0(x_0)-(u-x_0)h_0'(x_0)}{\widehat{h}_n(u) }\bigg] (1-\mathbb{F}_n(u))\,\d{u}\d{v}\\
	&\qquad \qquad + \int_{x_0}^{x_0+n^{-1/5}t} \int_{x_0}^v \frac{1-\mathbb{H}_n(u)}{\widehat{h}_n(u)}\,\d{ (\mathbb{H}_n^\ast(u)-H_0(u)) }\d{v}\bigg\},\\
	& \overline{\mathbb{H}}_n^{\mathrm{loc}}(t) \equiv n^{4/5} \frac{h_0(x_0)}{1-F_0(x_0)} \bigg\{\\
	&\qquad  \int_{x_0}^{x_0+n^{-1/5}t} \int_{x_0}^v \bigg[\frac{\widehat{h}_n(u)-h_0(x_0)-(u-x_0)h_0'(x_0)}{\widehat{h}_n(u) }\bigg] (1-\mathbb{F}_n(u))\,\d{u}\d{v}\bigg\}\\
	&\qquad\qquad +A_n+B_n t,
	\end{align*}
	where 
	\begin{align*}
	\d{\mathbb{H}_n(u)}&\equiv \big(1-\mathbb{F}_n(u-)\big)^{-1}\,\d{\mathbb{F}_n(u)},\\
	\d{\mathbb{H}_n^\ast(u)}&\equiv \frac{1-\mathbb{F}_n(u-)}{1-\mathbb{F}_n(u)}\,\d{\mathbb{H}_n(u)} = \frac{1}{1-\mathbb{F}_n(u)}\,\d{\mathbb{F}_n(u)},\\
	A_n &\equiv -n^{4/5} \frac{h_0(x_0)}{1-F_0(x_0)} \bigg(\int_0^{x_0} \int_0^v  \frac{1}{\widehat{h}_n(u)}\,\d{\mathbb{F}_n(u)}\d{v}-\int_{0}^{x_0} \int_0^v (1-\mathbb{F}_n(u))\,\d{u}\d{v} \bigg),\\
	B_n &\equiv -n^{3/5}\frac{h_0(x_0)}{1-F_0(x_0)}  \bigg(\int_0^{x_0}  \frac{1}{\widehat{h}_n(u)}\,\d{\mathbb{F}_n(u)}-\int_0^{x_0} (1-\mathbb{F}_n(u))\,\d{u} \bigg).
	\end{align*}
	Some tedious calculations show that $\overline{\mathbb{H}}_n^{\mathrm{loc}}(t)\geq \overline{\mathbb{Y}}_n^{\mathrm{loc}}(t)$ with equality taken at kinks of $\widehat{h}_n$ with negative slope. We consider a slight modification of the local process $\overline{\mathbb{H}}_n^{\mathrm{loc}}$ by replacing $\mathbb{F}_n$ in the integrand by $F_0$:
	\begin{align*}
	\mathbb{H}_n^{\mathrm{loc}}(t) &\equiv n^{4/5} \frac{h_0(x_0)}{1-F_0(x_0)} \bigg\{\\
	&\qquad  \int_{x_0}^{x_0+n^{-1/5}t} \int_{x_0}^v \bigg[\frac{\widehat{h}_n(u)-h_0(x_0)-(u-x_0)h_0'(x_0)}{\widehat{h}_n(u) }\bigg] (1-F_0(u))\,\d{u}\d{v}\bigg\}\\
	&\qquad\qquad +A_n+B_n t,\\
	\mathbb{Y}_n^{\mathrm{loc}}(t) &  = \overline{\mathbb{Y}}_n^{\mathrm{loc}}(t)+ \mathbb{H}_n^{\mathrm{loc}}(t) - \overline{\mathbb{H}}_n^{\mathrm{loc}}(t).
	\end{align*}
	It is easy to show that $\mathbb{H}_n^{\mathrm{loc}}- \overline{\mathbb{H}}_n^{\mathrm{loc}}\to 0$ a.s.~on compacta, and hence combined with the limit for $\overline{\mathbb{Y}}_n^{\mathrm{loc}}$ derived in \cite[pp.~1030-1031]{jankowski2009nonparametric}, we have
	\begin{align*}
	\mathbb{Y}_n^{\mathrm{loc}}(t)\rightsquigarrow \sqrt{\frac{h_0(x_0)}{1-F_0(x_0)}} \int_{0}^{t} \mathbb{B}(s)\,\d{s} + \frac{h_0''(x_0)}{4!} t^4,\quad \textrm{ in }C([-K,K])
	\end{align*}
	for any $K>0$, and $\mathbb{H}_n^{\mathrm{loc}}(t)\geq \mathbb{Y}_n^{\mathrm{loc}}(t)$ with equality taken at kinks of $\widehat{h}_n$ with negative slope. As the process $\mathbb{H}_n^{\mathrm{loc}}$ can be localized with arbitrarily high probability with touch points occuring at kinks of $\widehat{h}_n$ with negative slope, the rest of the arguments parallel the same pattern as before. In particular, let the scaling constants $\gamma_0,\gamma_1$ be chosen as
	\begin{align*}
	\gamma_0 \gamma_1^{3/2} = \sqrt{\frac{h_0(x_0)}{1-F_0(x_0)}},\qquad \gamma_0 \gamma_1^4 = \frac{h_0''(x_0)}{4!}.
	\end{align*}
	Then repeating the arguments in the proof of Theorem \ref{thm:pivotal_limit_log_concave} we obtain the asymptotically pivotal LNE theory. The rest of the claims follow from the same proof as in the proof of Theorem \ref{thm:CI_fcn}. \qed

\subsection{Proof of Theorem \ref{thm:pivotal_limit_deconvolution}}\label{pf:pivotal_limit_deconvolution}
	The LSE $\widehat{s}_n$ can be characterized as follows. Let
	\begin{align*}
	\mathbb{Y}_n(t)&\equiv \int_0^t U_n(u)\,\d{u},\\
	\mathbb{H}_n(t)&\equiv \int_0^t \bigg[\int_0^u\widehat{s}_n(v)\,\d{v} -  \int \widehat{s}_n^2 + \int \widehat{s}_n\,\d{U_n} \bigg]\d{u}\equiv \int_0^t S_n(u)\,\d{u}. 
	\end{align*}
	By \cite[Theorem 2.10]{jongbloed2009estimating}, $\mathbb{Y}_n(t)\leq \mathbb{H}_n(t)$ if and only if $t$ is a kink of $\widehat{s}_n$. Define the local processes $\mathbb{Y}_n^{\mathrm{loc}}, \mathbb{H}_n^{\mathrm{loc}}$ by
	\begin{align*}
	\mathbb{Y}_n^{\mathrm{loc}} (t) &\equiv n^{4/5}\int_{x_0}^{x_0+ n^{-1/5}t} \bigg[U_n(v)-U_n(x_0)-\\
	&\qquad\qquad\qquad\qquad - \int_{x_0}^v \big(s_0(x_0)+(u-x_0)s_0'(x_0)\big)\,\d{u}\bigg]\d{v},\\
	\mathbb{H}_n^{\mathrm{loc}}(t) &\equiv n^{4/5}\int_{x_0}^{x_0+ n^{-1/5}t} \bigg[S_n(v)-S_n(x_0)-\\
	&\qquad\qquad\qquad\qquad - \int_{x_0}^v \big(s_0(x_0)+(u-x_0)s_0'(x_0)\big)\,\d{u}\bigg]\d{v}\\
	&\quad +A_n +B_nt,
	\end{align*}
	where $A_n = n^{4/5}(\mathbb{H}_n(x_0)-\mathbb{Y}_n(x_0))$ and $B_n = n^{3/5} (\mathbb{H}_n'(x_0)-\mathbb{Y}_n'(x_0))$. Then $\mathbb{Y}_n^{\mathrm{loc}} (t)\leq \mathbb{H}_n^{\mathrm{loc}} (t)$ with equality taken where $x_0 + n^{-1/5}t$ is a kink of $\widehat{s}_n$. By the proof of \cite[Theorem 6.1]{jongbloed2009estimating},
	\begin{align*}
	\mathbb{Y}_n^{\mathrm{loc}} (t) \rightsquigarrow \frac{\sqrt{g_0(x_0)} }{k(0)} \int_0^t \mathbb{B}(s)\,\d{s} + \frac{s_0''(x_0)}{4!} t^4, \quad \textrm{ in }C([-K,K])
	\end{align*}
	for any $K>0$. The rest of the proof for the pivotal limit distribution theory is similar to that of Theorem \ref{thm:pivotal_limit_log_concave} by choosing the scaling factors $\gamma_0,\gamma_1$ such that $\gamma_0 \gamma_1^{3/2} = \sqrt{g_0(x_0)}/k(0)$ and $\gamma_0\gamma_1^4 = s_0''(x_0)/4!$. Hence the details are omitted. For the rest of the statements, it suffices to show that $\widehat{g}_n(x_0)\to_p g_0(x_0)$. As $n^{1/5}\big(\widehat{v}(x_0)-x_0\big)$ and $n^{1/5}\big(x_0-\widehat{u}(x_0)\big)$ converge to limiting random variables that put mass $0$ at $0$, so with probability at least $1-\epsilon$, there exists some $c=c_\epsilon>1$ such that
	\begin{align*}
	\bigabs{\widehat{g}_n(x_0)-g_0(x_0)} &\leq \sup_{ \substack{u\leq x_0\leq v,\\c^{-1} \leq n^{1/5}(v-x_0)\leq c, \\ c^{-1} \leq n^{1/5}(x_0-u)\leq c} } \biggabs{\frac{1}{v-u}\big(\G_n-G_0)(\bm{1}_{[u,v]})}\\
	& \qquad + \sup_{ \substack{u\leq x_0\leq v,\\ c^{-1} \leq n^{1/5}(v-x_0)\leq c, \\ c^{-1} \leq n^{1/5}(x_0-u)\leq c} }\biggabs{ \frac{1}{v-u}\int_u^v g_0(t)\,\d{t}-g_0(x_0)}.
	\end{align*}
	The second term is of order $\mathfrak{o}(1)$ by continuity of $g_0$ at $x_0$, so we only need to handle the first term which equals
	\begin{align*}
	&\sup_{h_1,h_2 \in [c^{-1},c]} \biggabs{\frac{1}{h_1+h_2} (\G_n-G_0)(n^{1/5}\bm{1}_{[x_0-h_1 n^{-1/5},x_0+h_2 n^{-1/5}] }) }\\
	&\leq (c/2) \sup_{h_1,h_2 \in [c^{-1},c]} \abs{\sqrt{n}(\G_n-G_0) (f_{n,h_1,h_2})},
	\end{align*}
	where $f_{n,h_1,h_2}\equiv n^{1/5-1/2}\bm{1}_{[x_0-h_1 n^{-1/5},x_0+h_2 n^{-1/5}] }$. Let $\mathcal{F}_n\equiv \{f_{n,h_1,h_2}: h_1,h_2 \in [c^{-1},c]\}$. Then $\mathcal{F}_n$ is VC-subgraph with index uniformly bounded in $n$, and has an envelope $F_n = f_{n,c,c}$. 
	As $G_0 F_n^2\leq \pnorm{g_0}{\infty} n^{2/5-1} \cdot 2c n^{-1/5} \to 0$, the above display converges in probability to $0$ by, e.g., \cite[Theorem 2.14.1]{van1996weak}. This completes the proof. \qed

\section{Proof of results in Section \ref{section:tail_estimate}}

\subsection{Proof of Theorem \ref{thm:tail_uniform}}\label{pf:tail_uniform}
	We consider the LSE in the convex nonincreasing density model with 
	\begin{align*}
	f_0(x) = b_0^{-1}\Big[(1-x)+  \alpha^{-1} 16^{-\alpha}  (x-x_0)^{\alpha}\Big]\bm{1}_{[0,\gamma_\alpha]},
	\end{align*}
	where $\gamma_\alpha = \inf \big\{ x>0: (1-x)+  \alpha^{-1} 16^{-\alpha} (x-x_0)^{\alpha}=0 \big\}$, and $x_0=1/2$. Here $b_0=b_0(\alpha)$ is a normalizing constant making $\int_0^{\gamma_\alpha} f_0 =1$, and $\gamma_\alpha \in [1, 3/2]$, so we have
	\begin{align*}
	b_0 &= \int_0^{\gamma_\alpha}\{ (1-x)+ \alpha^{-1} 16^{-\alpha}   (x- 1/2)^\alpha\} \,\d{x}\\
	& = \big(\gamma_\alpha-\gamma_\alpha^2/2\big)+ \frac{  \alpha^{-1} 16^{-\alpha}  }{(\alpha+1)} \big[(\gamma_\alpha-1/2)^{ \alpha+1 }-(1/2)^{ \alpha+1 } \big].
	\end{align*}
	Hence $3/8 \leq b_0 \leq 1/2 + 1/\big(\alpha(\alpha+1)16^{\alpha} \big) \leq 1$ as $\alpha \ge 2$. Let $r_n \equiv n^{-1/(2\alpha+1)}$.
	
	\smallskip
	\noindent (\textbf{Step 1}). Let $\{x_n \in I_n\}$ be a sequence of possibly random points, where $I_n \subset (0,\gamma_\alpha)$ is a non-random interval of length $Mr_n$, $M\geq 1$. Let $[\widehat{u},\widehat{v}]$, where
	$\widehat{v}\equiv \widehat{v}(x_n)$ and $\widehat{u}\equiv \widehat{u}(x_n)$, be the maximal interval containing $x_n$ on which $\widehat{f}_n$ is linear. 
	
	Using the same arguments as in \cite[pp.~1678]{groeneboom2001estimation}, we have 
	\begin{align*}
	-\int_{\widehat{u}}^{\widehat{v}} f_{\widehat{u},\widehat{v}}(x) f_0(x)\,\d{x}+ \abs{\mathbb{U}_n(\widehat{u},\widehat{w})}+\abs{\mathbb{U}_n(\widehat{w},\widehat{v})}\geq 0,
	\end{align*}
	where $f_{u,v}(x)= |x - (u+v)/2| - (v-u)/4$, $\mathbb{U}_n(u,v) = \int_u^v \big(x-(u+v)/2\big)\,\d{(\mathbb{F}_n-F_0)(x)}$ and $\widehat{w} = (\widehat{u} + \widehat{v})/2$. As $f_{\widehat{u},\widehat{v}}(x)$ is symmetric with respect to $x = \widehat{w}$,
	\begin{align*}
	& \int_{\widehat{u}}^{\widehat{v}} f_{\widehat{u},\widehat{v}}(x) f_0(x)\,\d{x} 
	\cr
	&= b_0^{-1} \int_{\widehat{u}}^{\widehat{v}} (1- \widehat{w} + \widehat{w} - x) f_{\widehat{u}, \widehat{v}}(x)\,\d{x} + b_0^{-1} \alpha^{-1} 16^{-\alpha}  \int_{\widehat{u}}^{\widehat{v}} f_{\widehat{u},\widehat{v}}(x)(x-x_0)^\alpha\,\d{x} \\
	& = 0 + b_0^{-1} \alpha^{-1} 16^{-\alpha}  \sum_{0\leq \beta \leq \alpha} \binom{\alpha}{\beta}(\widehat{w}-x_0)^{\alpha-\beta} \int_{\widehat{u}}^{\widehat{v}} f_{\widehat{u},\widehat{v}}(x)(x-\widehat{w})^\beta\,\d{x}\\
	& = b_0^{-1} \alpha^{-1} 16^{-\alpha}  \sum_{0\leq \beta \leq \alpha} \binom{\alpha}{\beta}(\widehat{w}-x_0)^{\alpha-\beta} \cdot 2 \int_{\widehat{u}}^{\widehat{w}} \Big( \widehat{w} - x - \frac{\widehat{v} - \widehat{u}}{4} \Big) (x-\widehat{w})^\beta\,\d{x}
	\cr
	& = \sum_{2\leq \beta\leq \alpha: \beta\textrm{ even}} \frac{b_0^{-1} \alpha^{-1} 16^{-\alpha}  \beta\binom{\alpha}{\beta}}{2^{\alpha +2}(\beta+1)(\beta+2)} \Big( \frac{ 2\widehat{w} - 2x_0}{\widehat{v} - \widehat{u} }\Big)^{\alpha - \beta} \cdot (\widehat{v}-\widehat{u})^{ \alpha +2} 
	\cr
	& \ge \frac{b_0^{-1}  16^{-\alpha} }{2^{\alpha +2}( \alpha +1)(\alpha +2)} \cdot (\widehat{v}-\widehat{u})^{ \alpha +2} \equiv  c_0 (\widehat{v}-\widehat{u})^{\alpha+2},
	\end{align*}
	where the last inequality follows from the fact that all summands are positive as $\alpha$ is even. Hence
	\begin{align}
	c_0 (\widehat{v}-\widehat{u})^{\alpha+2}\leq \abs{\mathbb{U}_n(\widehat{u},\widehat{w})}+\abs{\mathbb{U}_n(\widehat{w},\widehat{v})},
	\end{align}
	By choosing $\epsilon = c_0/4$, we see that 
	\begin{align*}
	0&\leq \big( |\mathbb{U}_n(\widehat{u},\widehat{w})| -\epsilon (\widehat{w}-\widehat{u})^{\alpha+2} \big)_+ + \big( |\mathbb{U}_n(\widehat{w},\widehat{v})| - \epsilon (\widehat{v}-\widehat{w})^{\alpha+2}\big)_+ \\
	&\qquad + (2\epsilon/2^{\alpha+2})(\widehat{v}-\widehat{u})^{\alpha+2} - c_0(\widehat{v}-\widehat{u})^{\alpha+2}\\
	&\leq \big( |\mathbb{U}_n(\widehat{u},\widehat{w})|-\epsilon (\widehat{w}-\widehat{u})^{\alpha+2} \big)_+ + \big( | \mathbb{U}_n(\widehat{w},\widehat{v})| - \epsilon (\widehat{v}-\widehat{w})^{\alpha+2}\big)_+
	\cr
	& \qquad - (c_0/2)(\widehat{v}-\widehat{u})^{\alpha+2}.
	\end{align*}
	As $c_0^{1/(\alpha+2)}$ stays away from $0$ and $\infty$ for all $\alpha$, Lemma \ref{lem:drifted_ep_size} with $s=0$ and $z_0(x,y)=(x+y)/2$ then implies the following: For any sequence $\{x_n \in I_n \}$ where $I_n \subset (0,\gamma_\alpha)$ is a non-random interval of length $Mr_n$($M\geq 1$), there exist absolute constants $t_0>0,C_1>0$ such that if $t\geq M^{1/2}t_0$, 
	\begin{align}\label{ineq:tail_uniform_1}
	\Prob\big( r_n^{-1}(\widehat{v}(x_n)-\widehat{u}(x_n)) >t^{1/(\alpha+2)}\big) \leq C_1 e^{-t^{1/2}/C_1}.
	\end{align}
	The exact numerical value of $t_0$ may change from line to line in the proof below.
	
	\smallskip
	\noindent (\textbf{Step 2}). This step is inspired by the proof of \cite[Lemma 4.3]{groeneboom2001estimation} but now using (\ref{ineq:tail_uniform_1}) and Lemma \ref{lem:drifted_ep_size} with $s=\alpha$. We replicate some details for the convenience of the readers. Let $\{\xi_n \in I_n\}$ be a sequence of possibly random points, where $\xi_n \pm r_n \in I_n$ and $I_n \subset (0, \gamma_{\alpha})$ is a non-random interval of length $Mr_n$ ($M\geq 1$). Applying (\ref{ineq:tail_uniform_1}) to $x_n = \xi_n \pm r_n$, we have with $\tau_n^-\equiv \widehat{u}(\xi_n-r_n), \tau_n^+\equiv \widehat{v}(\xi_n+r_n)$, it holds with probability at least $1-2C_1 e^{-t^{1/2}/C_1}$ that, $\xi_n- (t^{1/(\alpha+2)}+1)r_n\leq \tau_n^-\leq \xi_n - r_n < \xi_n + r_n\leq \tau_n^+\leq \xi_n+(t^{1/(\alpha+2)}+1)r_n$ for $t \geq M^{1/2} t_0$, where $t_0>0$ is a large absolute constant. In particular, $ 2r_n \leq \tau_n^{+} - \tau_n^{-} \leq 2r_n(t^{1/(\alpha + 2)} + 1)$. On the other hand, as both $\widehat{f}_n$ and $f_0$ are continuous, it holds on the event $\{\inf_{x \in [\tau_n^-,\tau_n^+]} \abs{\widehat{f}_n(x)-f_0(x)}\geq t r_n^\alpha \}$ that 
	\begin{align*}
	\biggabs{ \int_{\tau_n^-}^{\tau_n^+} \big(\widehat{f}_n(x)-f_0(x)\big)(\tau_n^+-x)\,\d{x} } \geq tr_n^\alpha (\tau_n^+-\tau_n^-)^2/2\geq (2t) r_n^{\alpha+2}.
	\end{align*}
	Recall the $\widehat{f}_n$ is the LSE if and only if 
	\begin{align*}
	\mathbb{H}_n(t) = \int_{0}^t \int_0^v \widehat{f}_n(u)\,\d{u} \d{v} \ge  \mathbb{Y}_n(t) = \int_0^t \mathbb{F}_n(v)\,\d{v}
	\end{align*}
	with equality taken at kink points (see \cite[Lemma 2.2]{groeneboom2001estimation}). As a result, $\mathbb{F}_n(v) = \int_{0}^v \widehat{f}_n(u)\,\d{u}$ when $v$ is a kink and therefore
	\begin{align*}
	\int_{\tau_n^-}^{\tau_n^+} \big(\widehat{f}_n(x)-f_0(x)\big)(\tau_n^+-x)\,\d{x} = \int_{\tau_n^-}^{\tau_n^+} (\tau_n^+-x)\,\d{ (\mathbb{F}_n-F_0)(x)}.
	\end{align*}
	Together with Lemma \ref{lem:drifted_ep_size} below (with $s=\alpha$ and $z_0(x,y) = y $), it follows that with probability at least $1-C_2 e^{-t^{1/2}/C_2}$, for $t \geq M^{1/2} t_0, M\geq 1$ with some large absolute constant $t_0>0$, 
	\begin{align*}
	&(2t) r_n^{\alpha+2} \leq \biggabs{\int_{\tau_n^-}^{\tau_n^+} (\tau_n^+-x)\,\d{ (\mathbb{F}_n-F_0)(x)} }\leq t r_n^{\alpha+2} + r_n^\alpha (\tau_n^+-\tau_n^-)^{2},
	\end{align*}
	which requires $ \tau_n^+-\tau_n^- \geq r_n \sqrt{t}$. This leads to a contradiction to $\tau_n^{+} - \tau_n^{-} \leq 2 r_n(t^{1/(\alpha + 2)} + 1)$ for large absolute $t$.

	In other words, for any (possibly random) sequence $\{\xi_n \in I_n\}$ where $\xi_n \pm r_n \in I_n$ and $I_n \subset (0, \gamma_{\alpha})$ is a non-random interval of length $Mr_n$($M\ge1$), there exist absolute constants $t_0>0,C_3>0$ such that if $t\geq M^{1/2}t_0$, 
	\begin{align}\label{ineq:tail_uniform_2}
	&\Prob\bigg(\inf_{x \in [\widehat{u}(\xi_n-r_n),\widehat{v}(\xi_n+r_n)]} \abs{\widehat{f}_n(x)-f_0(x)}\geq tr_n^\alpha\bigg) \leq C_3e^{-t^{1/2}/C_3}.
	\end{align}

	\smallskip
	\noindent (\textbf{Step 3}). This step is inspired by the proof of \cite[Lemma 4.4]{groeneboom2001estimation}. Let $\sigma_{n,1}$ be the first kink of $\widehat{f}_n$ to the right of $x_0+Lr_n$,  and $\xi_{n,1} \equiv \sigma_{n,1}+2r_n$. Let $\widehat{u}_{n,1} \equiv  \widehat{u}(\xi_{n,1}-r_n), \widehat{v}_{n,1} \equiv  \widehat{v}(\xi_{n,1}+r_n)$. Let $\sigma_{n,2}\equiv \widehat{v}(\widehat{v}_{n,1}+r_n)$, $\xi_{n,2}\equiv \sigma_{n,2}+2r_n$, $\widehat{u}_{n,2} \equiv  \widehat{u}(\xi_{n,2}-r_n)$ and $\widehat{v}_{n,2} \equiv  \widehat{v}(\xi_{n,2}+r_n)$.

	Fix $t > 0$. For each $n$, let $I_n \equiv [x_0 + Lr_n, x_0 + (L + M)r_n]$ where $M \equiv 3t^{1/(\alpha+2)}+ 7$. When $n$ is large enough, we have $x_0 + (L + M + t^{1/(\alpha + 2)}) r_n < 1 \le \gamma_{\alpha}$, so that $I_n \subset (0, \gamma_{\alpha})$. By repeated application of (\ref{ineq:tail_uniform_1}) with $x_n$ being $x_0 + Lr_n$, $\xi_{n,1} + r_n$, $\widehat{v}_{n,1} + r_n$ and $\xi_{n,2}  +r_n$, on an event $E_1$ with probability at least $1-C_4e^{-t^{1/2}/C_4}$, $\widehat{v}_{n,2}\leq x_0+(L + M + t^{1/(\alpha+2)}) r_n  < 1$ for $n$ large enough. By (\ref{ineq:tail_uniform_2}) with $\xi_n$ being $\xi_{n,1}$ and $\xi_{n,2}$,	if $t^{(\alpha-1)/(\alpha+2)} \geq M^{1/2}t_0$, on $E_1 \cap E_2$ where event $E_2$ has probability at least $1-2C_3 e^{- t^{(\alpha-1)/2(\alpha+2)}/C_3}$, there exist $\pi_{n,i} \in  [\widehat{u}_{n,i}, \widehat{v}_{n,i}] (i=1,2)$ such that $\abs{\widehat{f}_n(\pi_{n,i})-f_0(\pi_{n,i})}\leq t^{(\alpha-1)/(\alpha+2)} r_n^\alpha$. Hence on $E_1\cap E_2$,  if $t^{(\alpha-1)/(\alpha+2)}\geq M^{1/2}t_0$,
	\begin{align*}
	&\widehat{f}_n'(x_0+Lr_n)\leq \frac{\widehat{f}_n(\pi_{n,2})-\widehat{f}_n(\pi_{n,1})}{\pi_{n,2}-\pi_{n,1} } \\
	&\leq \frac{f_0(\pi_{n,2})-f_0(\pi_{n,1})+2t^{(\alpha-1)/(\alpha+2)} r_n^\alpha}{\pi_{n,2}-\pi_{n,1} }
	\\
	&\leq f_0'(\pi_{n,2})+ 2t^{(\alpha-1)/(\alpha+2)} r_n^{\alpha-1}
	\\
	&\leq f_0'(x_0) + b_0^{-1} (\alpha^{-1} 16^{-\alpha}) \alpha (\pi_{n,2}-x_0)^{\alpha-1}+2t^{(\alpha-1)/(\alpha+2)} r_n^{\alpha-1} \\
	&\leq f_0'(x_0) +  b_0^{-1} 16^{-\alpha}  \cdot (L+M + t^{1/(\alpha+2)})^{\alpha-1} r_n^{\alpha-1}+2t^{(\alpha-1)/(\alpha+2)} r_n^{\alpha-1}
	\\
	&\leq f_0'(x_0)+16^{1-\alpha}(4t^{1/(\alpha+2)}+ L + 7)^{\alpha-1} r_n^{\alpha-1}+2t^{(\alpha-1)/(\alpha+2)} r_n^{\alpha-1}.
	\end{align*}
	The above arguments show that for any choice of $L >0$ and $t > t_0$ such that $t^{(\alpha-1)/(\alpha+2)}\geq (3t^{1/(\alpha+2)}+ 7)^{1/2}t_0$ for a certain absolute constant $t_0$, we have by symmetry 
	\begin{align*}
	&\Prob \bigg( r_n^{-(\alpha-1)} \sup_{ |u - x_0| \le Lr_n
	}\abs{\widehat{f}_n'(u)-f_0'(x_0)}\\
    &\quad \quad > 16^{1-\alpha}(4t^{1/(\alpha+2)}+ L + 7)^{\alpha-1} + 2t^{(\alpha-1)/(\alpha+2)} \bigg)\leq C_5e^{- t^{(\alpha-1)/(2\alpha + 4)} /C_5}. \nonumber
	\end{align*}
	Here $C_5>0$ is an absolute constant. In particular, if we choose $L=2t^{1/(\alpha+2)}+3$, we have for $t\geq t_0$ (where $t_0\geq 1$ is a large enough absolute constant different from the previous display), it holds for $n\geq n_0(t,\alpha)$ where $n_0(t,\alpha)\in \N$ that
	\begin{align}\label{ineq:tail_uniform_3}
	\Prob\bigg( r_n^{-(\alpha-1)} \sup_{	|u - x_0| \leq (2t^{1/(\alpha+2)}+3)r_n
	}\abs{\widehat{f}_n'(u)-f_0'(x_0)}> 3t \bigg)\leq C_5 e^{-t^{1/2}/C_5}.
	\end{align}
	
	\smallskip
	
	\noindent (\textbf{Step 4}).  Let $\sigma_{n,+} \equiv \widehat{v}(x_0)$ and $\sigma_{n,-} \equiv \widehat{u}(x_0)$. Define $\xi_{n,\pm} \equiv \sigma_{n, \pm} \pm 2r_n$, $\widehat{u}_{n, \pm } = \widehat{u}(\xi_{n, \pm } - r_n)$ and $\widehat{v}_{n, \pm } = \widehat{v}(\xi_{n, \pm } + r_n)$. We set $I_n \equiv [x_0 - (t^{1/(\alpha + 2)}+3)r_n, x_0 + (t^{1/(\alpha + 2)}+3)r_n]$ so that $M \equiv 2t^{1/(\alpha + 2)}+6 $. It holds for large enough $n$ that $I_n \subset (0, \gamma_{\alpha})$. Then, by repeated application of \eqref{ineq:tail_uniform_1} on an event $E_3$ with probability at least $1-C_6 e^{-t^{1/2}/C_6}$, $ (\widehat{v}_{n,+}-x_0) \vee (x_0 - \widehat{u}_{n, -})
	\leq (2t^{1/(\alpha+2)}+3)r_n$. By \eqref{ineq:tail_uniform_2} with $\xi_n = \xi_{n, \pm}$,	on $E_3 \cap E_4$ where event $E_4$ has probability at least $1 - 2 C_3 e^{-t^{1/2}/C_3}$, there exists
	$\pi_{n,\pm} \in  [\widehat{u}_{n,\pm}, \widehat{v}_{n,\pm}]$ such that $\abs{\widehat{f}_n(\pi_{n,\pm})-f_0(\pi_{n,\pm})}\leq t r_n^\alpha$ when $t \ge M^{1/2}t_0$.  Using (\ref{ineq:tail_uniform_3}), on $E_3 \cap E_5$ where event $E_5$ has probability at least $1-C_5 e^{-t^{1/2}/C_5}$, $\sup_{| u - x_0| \le (2t^{1/(\alpha+2)} + 3)r_n
	}\abs{\widehat{f}_n'(u)-f_0'(x_0)}\leq 3t r_n^{\alpha-1}$ holds for $n$ large enough. Hence on the event $\cap_{j=3}^5 E_j$,
	\begin{align*}
	\widehat{f}_n(x_0)&\geq \widehat{f}_n(\pi_{n,+})+\widehat{f}_n'(\pi_{n,+})(x_0-\pi_{n,+})\\
	&\geq f_0(\pi_{n,+})-tr_n^\alpha+ \big(f_0'(x_0)+ 3t r_n^{\alpha-1} \big)(x_0-\pi_{n,+})\\
	&\geq f_0(x_0) +(\pi_{n,+}-x_0) f_0'(x_0)+(x_0-\pi_{n,+})f_0'(x_0) - \big(10t+ 6t^{\frac{\alpha + 3}{\alpha + 2} } \big) r_n^{\alpha} \\
	&\ge f_0(x_0) - K t^{5/4} r_n^{\alpha}.
	\end{align*}
	Reversely, 
	\begin{align*}
	\widehat{f}_n(x_0)&\leq \widehat{f}_n(\pi_{n,-})+ \frac{ \widehat{f}_n(\pi_{n,+})-\widehat{f}_n(\pi_{n,-}) }{\pi_{n,+}-\pi_{n,-}}(x_0-\pi_{n,-})\\
	& \leq f_0(\pi_{n,-})+ tr_n^\alpha+ \frac{ f_0(\pi_{n,+})-f_0(\pi_{n,-})+2tr_n^\alpha }{\pi_{n,+}-\pi_{n,-}}(x_0-\pi_{n,-})\\
	& = f_0(x_0)+b_0^{-1}\Big[-\big(\pi_{n,-}-x_0\big)+ \alpha^{-1} 16^{-\alpha}  (\pi_{n,-}-x_0)^\alpha\Big] + tr_n^{\alpha} \\
	&\qquad + \frac{x_0-\pi_{n,-} }{\pi_{n,+}-\pi_{n,-}}\bigg\{ f_0(x_0)+b_0^{-1}\Big[-\big(\pi_{n,+}-x_0\big)+ \alpha^{-1} 16^{-\alpha}  (\pi_{n,+}-x_0)^\alpha\Big]
	\\
	&\qquad\qquad -f_0(x_0)-b_0^{-1}\Big[-\big(\pi_{n,-}-x_0\big)+ \alpha^{-1} 16^{-\alpha} (\pi_{n,-}-x_0)^\alpha\Big] + 2tr_n^{\alpha} \bigg\}
	\\
	&\qquad+tr_n^\alpha+2tr_n^{\alpha} \cdot \frac{x_0-\pi_{n,-}}{\pi_{n,+}-\pi_{n,-}}
	\\
	& \le  f_0(x_0) + b_0^{-1} \alpha^{-1} 16^{-\alpha} \Big[ (\pi_{n,-}-x_0)^\alpha + \frac{x_0-\pi_{n,-}}{\pi_{n,+}-\pi_{n,-}} \Big( (\pi_{n,+}-x_0)^\alpha + 2tr_n^\alpha \Big) \Big]  \\
	& \leq f_0(x_0) + K t r_n^\alpha,
	\end{align*}
	for some absolute constant $K>0$. Hence we have proved that 
	when $t \ge (2  t^{1/(\alpha + 2)} + 6)^{1/2}t_0$ or equivalently when $t \ge t_0$ for a different large enough absolute constant $t_0$, it holds for $n\geq n_1(t,\alpha)$ where $n_1(t,\alpha)\in \N$ that
	\begin{align}
	\Prob\bigg(r_n^{-\alpha}\abs{\widehat{f}_n(x_0)-f_0(x_0)}>t^{5/4}\bigg)\leq C_7 e^{-t^{1/2}/C_7}. 
	\end{align}
	Here $C_7>0$ is an absolute constant.
	
	\smallskip
	\noindent (\textbf{Step 5}). Finally we take limits: by Portmanteau theorem, for $t\geq t_0$, where $t_0\geq 1$ is a large enough absolute constant that does not depend on $\alpha$, 
	\begin{align*}
	&\Prob \big( h^\ast_{\alpha;+}(f_0) + h^\ast_{\alpha;-}(f_0) >t\big) 
	\cr
	&\leq \liminf_{n \to \infty} \Prob\big(\abs{r_n^{-1}(\widehat{v}(x_0)-\widehat{u}(x_0)}>t\big)\leq C_1 e^{-t^{(\alpha+2)/2}/C_1}\leq C_1 e^{-t^2/C_1},
	\end{align*}
	and
	\begin{align*}
	&\Prob \big(\abs{ \mathbb{H}_\alpha^{(2)}(0; f_0) } > t\big)\leq \liminf_{n \to \infty} \Prob\big(r_n^{-\alpha}\abs{\widehat{f}_n(x_0)-f_0(x_0)}>t\big)\leq C_7 e^{-t^{2/5}/C_7},\\
	& \Prob \big(\abs{\mathbb{H}_\alpha^{(2)}(0; f_0) } > t\big)\leq \liminf_{n \to \infty} \Prob\big(r_n^{-(\alpha-1)}\abs{\widehat{f}_n'(x_0)-f_0'(x_0)}>t\big)\leq C_5 e^{-t^{1/2}/C_5}.
	\end{align*}
	The constants above do not depend on $\alpha$.	Now to translate these estimates to the canonical processes. Let $\gamma_0,\gamma_1>0$ be such that
	\begin{align*}
	\gamma_0\gamma_1^{3/2} = \sqrt{f_0(x_0)} = \sqrt{1/2b_0},\quad \gamma_0\gamma_1^{\alpha+2} = \frac{f_0^{(\alpha)}(x_0)}{(\alpha+2)!} =  \frac{ b_0^{-1} \alpha^{-1} 16^{-\alpha}  }{(\alpha+1)(\alpha+2)}.
	\end{align*}
	Then $\mathbb{H}_\alpha^{(2)}(0), \mathbb{H}_\alpha^{(3)}(0), h^\ast_{\alpha;\pm}$ are related to their canonical versions through
	\begin{align*}
	\mathbb{H}_\alpha^{(2)}(t; f_0) = \gamma_0 \gamma_1^2 \mathbb{H}_{\alpha}^{(2)}(\gamma_1 t),\,\mathbb{H}_\alpha^{(3)}(t; f_0) = \gamma_0 \gamma_1^3 \mathbb{H}_{\alpha}^{(3)}(\gamma_1 t),\,h^\ast_{\alpha; \pm}(f_0) = \gamma_1 h^\ast_{\alpha;\pm}.
	\end{align*}
	It is easy to solve that
	\begin{align*}
	\gamma_0\gamma_1^2 &= (\sqrt{1/2b_0})^{\frac{2\alpha}{2\alpha+1} } \bigg(\frac{b_0^{-1} \alpha^{-1} 16^{-\alpha} }{(\alpha+1)(\alpha+2)}\bigg)^{\frac{1}{2\alpha+1}},\\
	\gamma_0\gamma_1^3& = (\sqrt{1/2b_0})^{\frac{2\alpha-2}{2\alpha+1} } \bigg(\frac{b_0^{-1} \alpha^{-1} 16^{-\alpha} }{(\alpha+1)(\alpha+2)}\bigg)^{\frac{3}{2\alpha+1}},\\
	\gamma_1 &= \bigg(\frac{b_0^{-1} \alpha^{-1} 16^{-\alpha} }{(\alpha+1)(\alpha+2)} \bigg)^{\frac{2}{2\alpha+1}},
	\end{align*}
	which stay bounded away from $0$ and $\infty$ for all $\alpha$'s. The claim easily follows. \qed

\subsection{Proof of Corollary \ref{cor:unif_tail_pivot}}\label{pf:unif_tail_pivot}

	Note that by definition of $\mathbb{L}_\alpha^{(0)}$, 
	\begin{align*}
	\Prob\big(\abs{\mathbb{L}_\alpha^{(0)} }>t\big) &= \Prob\bigg(\biggabs{ \sqrt{ h^\ast_{\alpha;+} + h^\ast_{\alpha;-} }\cdot \mathbb{H}_{\alpha}^{(2)}(0)}>t\bigg)\\
	&\leq \Prob\bigg(\sqrt{ h^\ast_{\alpha;+} + h^\ast_{\alpha;-} }>t^{1/2} \bigg)+\Prob\big(\abs{ \mathbb{H}_{\alpha}^{(2)}(0) } >t^{1/2}\big).
	\end{align*}
	The desired tail bound now follows by Theorem \ref{thm:tail_uniform}. A similar argument works for $\mathbb{L}_\alpha^{(1)}$. \qed

\section{Technical lemmas}

\begin{lemma}\label{lem:drifted_ep_size}
	Fix any measurable function $z_0:\R\times \R\to \R$ such that $x\leq z_0(x,y)\leq y$ for $x\leq y$. Let
	\begin{align*}
	\mathbb{U}_n(x,y; z_0) = \int_x^y \big( z - z_0(x,y) \big)\,\d{(\mathbb{F}_n-F_0)(z)},
	\end{align*}
	where $\mathbb{F}_n$ is the empirical distribution function based on i.i.d.~observations with distribution function $F_0$ with a uniformly bounded Lebesgue density function. Let $r_n \equiv  n^{-1/(2\alpha+1)}$ and $M\geq 1$. For small enough $\epsilon>0$, there exists some $t_0 = t_0(\epsilon)$ such that for $t\geq M^{1/2} t_0$, we have for any $0\leq  s \leq \alpha$, and any interval $I$ of length $Mr_n$ contained in the support of $F_0$, 
	\begin{align*}
	\Prob\bigg(\sup_{\substack{x,y: x\leq x_0\leq y, \\ r_n\leq y - x\leq 1, \\ x_0 \in I}} r_n^{-(\alpha+2)}\Big( |\mathbb{U}_n(x,y ; z_0)| -\epsilon\cdot r_n^s (y-x)^{\alpha+2-s}\Big)_+>t\bigg)\leq Ce^{-t^{1/2}/C}.
	\end{align*}
	Here the constant $C>0$ does not depend on $\alpha,s,M,I$. 
\end{lemma}

\begin{proof}
	Define $S_{I}(a, b) = \{(x, y): x \le x_0 \le y, a \le y -x \le b, x_0 \in I\}$. Let $\mathcal{F}_R\equiv \{f_{x,y}(z)\equiv \big(z- z_0(x,y)
	\big) \bm{1}_{ \{z \in [x,y] \} }: 	(x, y) \in S_{I}(0, R)
	\}$. Then an envelope of $\mathcal{F}_R$ is given by $F_R=R  \bm{1}_{ \{ z \in [I_\ell-R,I_u+R] \}} $ where $[I_\ell,I_u]=I$, and hence $\E F_R^2 = \mathcal{O} (R^2(R \vee Mr_n))$. By a standard empirical process bound (see e.g.,~\cite[Theorem 2.14.1]{van1996weak}) upon noting that the class $\mathcal{F}_R$ is VC-subgraph, we have
	\begin{align*}
	\E \sup_{(x,y) \in S_I(0, R)
	} \abs{\mathbb{U}_n(x,y;z_0)}=\E \sup_{f \in \mathcal{F}_R} \abs{(\Prob_n-P)(f)}\lesssim n^{-1/2}\big( R^2(R \vee Mr_n)\big)^{1/2}.
	\end{align*}
	Hence with $L_n$ being the smallest integer such that $r_n2^{L_n-1}>1$, for any $1\leq \ell \leq L_n$, we have 
	\begin{align*}
	&\E \sup_{(x, y) \in S_I(0, 2^{\ell}r_n)
	} \abs{\mathbb{U}_n(x,y;z_0)}\leq K n^{-1/2} \big(2^{2\ell} (2^\ell \vee M) r_n^3\big)^{1/2} = K 2^{\ell}(2^\ell \vee M)^{1/2} r_n^{\alpha+2},\\
	&\sup_{f \in \mathcal{F}_{2^\ell r_n} }  \mathrm{Var}_{P} (f)  = \mathcal{O}\big( 2^{3\ell} r_n^3\big).
	\end{align*}
	Here $K>0$ is an absolute constant. Hence 
	\begin{align*}
	&\Prob\bigg(\sup_{(x, y) \in S_I(r_n, 1) } r_n^{-(\alpha+2)}\big(\abs{\mathbb{U}_n(x,y;z_0)}-\epsilon\cdot r_n^s(y-x)^{\alpha+2-s}\big)_+>t\bigg)\\
	&\leq \sum_{\ell=1}^{L_n }\Prob\bigg(\sup_{ (x, y) \in S_{I}(2^{\ell-1}r_n, 2^{\ell} r_n \wedge 1) 
	} r_n^{-(\alpha+2)}\big(\abs{\mathbb{U}_n(x,y;z_0)}-\epsilon\cdot r_n^s(y-x)^{\alpha+2-s}\big)_+>t\bigg)\\
	&\leq \sum_{\ell=1}^{L_n } \Prob\bigg(\sup_{ (x, y) \in S_I(2^{\ell-1}r_n, 2^{\ell} r_n \wedge 1) 
	} r_n^{-(\alpha+2)}\abs{\mathbb{U}_n(x,y;z_0)} >t+\epsilon \cdot 2^{(\ell-1)(\alpha+2-s)}\bigg)\\
	& \leq \sum_{\ell=1}^{L_n } \Prob\bigg(\sup_{f \in \mathcal{F}_{2^\ell r_n\wedge 1} } \abs{\sqrt{n}(\Prob_n-P)(f)}-2\E\sup_{f \in \mathcal{F}_{2^\ell r_n\wedge 1} } \abs{\sqrt{n}(\Prob_n-P)(f)}\\
	&\qquad\qquad\qquad\qquad\qquad > r_n^{3/2}\big(t+\epsilon\cdot 2^{(\ell-1)(\alpha+2-s)}-2K2^{\ell}(2^\ell \vee M)^{1/2}\big)_+  \bigg)\\
	&\leq \sum_{\ell=1}^{L_n } \exp\bigg(-K_1^{-1}\cdot \frac{r_n^3\big(t+\epsilon\cdot 2^{(\ell-1)(\alpha+2-s)}-2K2^{\ell}(2^\ell \vee M)^{1/2}\big)_+^2}{2^{3\ell} r_n^3+ (2^\ell r_n\wedge 1) r_n^{\alpha+2} \big(t+\epsilon \cdot 2^{(\ell-1)(\alpha+2-s)}-2K2^{\ell}(2^\ell \vee M)^{1/2}\big)_+ }\bigg)\\
	&\leq \sum_{\ell=1}^{L_n} \exp\bigg(-K_2^{-1}\cdot \frac{ (t+\epsilon \cdot 2^{(\ell-1)(\alpha+2-s)})^2 }{2^{3\ell} + r_n^{\alpha-1} (t+\epsilon \cdot 2^{(\ell-1)(\alpha+2-s)}) }\bigg)  \ \ \cdots \cdots (*),
	\end{align*}
	where in the second last inequality we used Talagrand's concentration inequality (see e.g., Lemma~\ref{lem:talagrand_conc_ineq} below), and in the last inequality we used the fact that for $t\geq \sup_{\ell \in \N} (4K 2^{\ell}(2^\ell \cdot M)^{1/2}-\epsilon 2^{2(\ell-1)}) \vee 0 \geq M^{1/2} \sup_{\ell \in \N} (4K 2^{3\ell/2}-\epsilon 2^{2(\ell-1)}) \equiv M^{1/2} t_0$, where $t_0\equiv t_0(K,\epsilon)$, so
	\begin{align*}
	t+\epsilon \cdot 2^{(\ell-1)(\alpha+2-s)}-2 K2^{\ell}(2^\ell \vee M)^{1/2} \geq  (t+\epsilon\cdot 2^{(\ell-1)(\alpha+2-s)})/2,\quad \forall \ell\geq 1.
	\end{align*}
	The probability bound $(*)$ can be further bounded by
	\begin{align*}
	&\sum_{\ell=1}^{L_n} \exp\bigg(-K_3^{-1} \min \bigg\{ \frac{(t+2^{(\ell-1)(\alpha+2-s)})^2}{ 2^{3\ell}  }, r_n^{-(\alpha-1)} (t+2^{(\ell-1)(\alpha+2-s)})\bigg\}\bigg)\\
	&\leq \sum_{\ell=1}^{L_n} \exp\bigg(-  \frac{t^2+2^{4(\ell-1)}}{K_3 2^{3\ell}}  \bigg) +  \sum_{\ell = 1}^{L_n} \exp \bigg( - \frac{ t+ 2^{(\ell-1)(\alpha+2-s)} }{ K_3 r_n^{\alpha-1}  } \bigg) \\
	& \leq \sum_{\ell: t> 2^{2(\ell-1)} } e^{-t^2/(K_4 2^{3\ell} )}+\sum_{\ell\leq L_n: t\leq 2^{2(\ell-1)}} e^{- 2^{\ell}/K_4 }+ e^{-t/K_4}\sum_{\ell} e^{-r_n^{-(\alpha-1)} 2^{2\ell}/K_4}\\
	&\leq K_5\log_+(t)\cdot e^{- t^{1/2}/K_5} + K_5 e^{- t^{1/2}/K_5}+ K_5e^{-t/K_5} \leq K_6e^{-t^{1/2}/K_6}.
	\end{align*}
	The constants depend on $\epsilon$ only.
\end{proof}

Talagrand's concentration inequality \cite{talagrand1996new} for the empirical process in the form given by Bousquet \cite{bousquet2003concentration} (see also \cite[Theorem 3.3.9]{gine2015mathematical}), is recorded as follows.

\begin{lemma}[Talagrand's concentration inequality]\label{lem:talagrand_conc_ineq}
	Let $\mathcal{F}$ be a countable class of real-valued measurable functions such that $\sup_{f \in \mathcal{F}} \pnorm{f}{\infty}\leq b$ and $X_1,\ldots,X_n$  be i.i.d.~random variables with law $P$. Then there exists some absolute constant  $K>1$ such that
	\begin{align*}
	&\Prob\bigg( K^{-1}\sup_{f \in \mathcal{F}}\abs{\sqrt{n}(\Prob_n-P) f} \geq \E\sup_{f \in \mathcal{F}}\abs{\sqrt{n}(\Prob_n-P) f} +x \bigg)\\
	&\leq \exp\bigg(-\frac{x^2}{  K(\sigma^2 + bx/\sqrt{n}) }\bigg),
	\end{align*}
	where $\sigma^2\equiv \sup_{f \in \mathcal{F}} \mathrm{Var}_P f$ and $\Prob_n$ denotes the empirical distribution of $X_{1},\dots,X_{n}$. 
\end{lemma}

\bibliographystyle{amsalpha}
\bibliography{mybib}

\end{document}